\documentclass[11pt,reqno]{amsart}
\usepackage[margin=1.1in]{geometry}
\usepackage{amsmath,amssymb,amsthm,mathtools}
\usepackage[activate={true,nocompatibility},final,tracking=false,kerning=false,spacing=false,factor=1100,stretch=10,shrink=10,expansion=false]{microtype}
\usepackage{enumitem}
\usepackage{xcolor}
\usepackage{tikz}
\usetikzlibrary{cd,positioning}
\usepackage[colorlinks=true,linkcolor=blue!60!black,citecolor=blue!60!black,urlcolor=blue!60!black]{hyperref}

\theoremstyle{plain}
\newtheorem{theorem}{Theorem}[section]
\newtheorem{proposition}[theorem]{Proposition}
\newtheorem{lemma}[theorem]{Lemma}
\newtheorem{corollary}[theorem]{Corollary}
\theoremstyle{definition}
\newtheorem{definition}[theorem]{Definition}
\newtheorem{example}[theorem]{Example}
\newtheorem{remark}[theorem]{Remark}
\newtheorem{convention}[theorem]{Convention}

\newcommand{\g}{\mathfrak{g}}
\newcommand{\gdual}{\mathfrak{g}^{\vee}}
\newcommand{\R}{\mathbb{R}}
\newcommand{\Q}{\mathbb{Q}}
\newcommand{\Z}{\mathbb{Z}}
\newcommand{\dF}{d_{F}}

\newcommand{\Om}{\Omega}
\newcommand{\Ham}{\mathrm{Ham}}
\newcommand{\Lie}{L_{\infty}}
\newcommand{\rvp}{\varpi}
\newcommand{\io}{\iota}
\newcommand{\Lder}{\mathcal{L}}
\newcommand{\vs}{\varsigma}
\newcommand{\XF}{\mathfrak{X}(F)}
\newcommand{\id}{\operatorname{id}}
\newcommand{\Ad}{\operatorname{Ad}}
\newcommand{\thL}{\theta^{L}}
\newcommand{\thR}{\theta^{R}}
\newcommand{\ip}[2]{\langle #1,\,#2\rangle}
\newcommand{\pair}[2]{\left\langle #1,\,#2\right\rangle}
\newcommand{\hol}{\operatorname{hol}}
\newcommand{\btot}{\mathbf{d}}
\providecommand{\Hom}{\operatorname{Hom}}
\newcommand{\curv}{\operatorname{curv}}

\numberwithin{equation}{section}

\begin{document}

\title[Periods, prequantization, and rigidity]{Periods, prequantization, and rigidity\\ in relative multisymplectic geometry}

\author{Dinamo Djounvouna}
\address{Department of Mathematics, University of Manitoba,
Winnipeg, MB R3T 2N2, Canada}
\email{djounvod@myumanitoba.ca}


\begin{abstract}
Relative multisymplectic geometry replaces differential forms on a single
manifold by cocycles in the mapping cone of a smooth map $F\colon M\to N$.
Building on the relative Cartan calculus, the Lie $n$-algebras of relative
observables, and the relative homotopy moment maps developed in companion work,
we establish a range of applications showing that the framework is a working
tool rather than a formal generalization.

We first construct the integration pairing between relative differential forms
and smooth relative chains, and prove two structural results that make it
usable: a period criterion, reducing relative integrality to the periods of the
target form together with a defect homomorphism on the classes killed by $F_*$,
and a functoriality theorem for morphisms of arrows. The criterion yields a
structure theorem for levels: the integers $k$ at which $k\rvp$ is relatively
integral always form a cyclic group $k_0\Z$, with no hypothesis on $F$, and
$k_0$ is computable from finitely many integrals whenever the relevant homology
is finitely generated. These yield a
characterization of homotopy-invariant bulk--boundary action functionals, hence
a precise treatment of Wess--Zumino terms; a relative Weil--Kostant theorem,
whose specialization to a Lagrangian submanifold is the Bohr--Sommerfeld
condition of geometric quantization; a relative Noether identity, with a
bulk--boundary splitting of the conserved charges; and a rigidity theorem making
comoment maps unique, strict and equivariant, so that the Kostant--Souriau
cocycle disappears. Two closing sections analyse the degenerate edges
$M=\varnothing$ and $N=\mathrm{pt}$, at which the absolute theories are
recovered, and separate weak from strong nondegeneracy, determining which
results require which.
\end{abstract}

\maketitle
\setcounter{tocdepth}{2}
\tableofcontents

\section{Introduction}\label{sec:intro}

Classical multisymplectic geometry describes higher Hamiltonian systems by
means of closed nondegenerate differential forms.  A large class of geometric
constructions arising in field theory and symplectic geometry is, however,
intrinsically relative: the relevant data live not on one manifold but on a
smooth map $F\colon M\to N$.  Wess--Zumino terms, prequantum data equipped with
a prescribed trivialization, bulk--boundary action functionals, and
group-valued moment maps all have this form.

The appropriate differential complex is the mapping cone
\[
  \Om^{k}(F):=\Om^{k}(N)\oplus\Om^{k-1}(M),
  \qquad
  \dF(\omega,\eta)=(d\omega,\,F^*\omega-d\eta).
\]
Thus a closed relative form is a pair
$\rvp=(\omega,\eta)$ consisting of a closed form on $N$ and a coherent
trivialization of its pullback to $M$.  The central point is that the geometry
is carried by the arrow $F$, rather than by either endpoint separately.  For
example, a Wess--Zumino functional pairs a target cocycle with a bulk filling
while the source form records the boundary correction; similarly, the
quasi-Hamiltonian identity $d\omega=\mu^*\eta$ says precisely that the pair
$(\eta,\omega)$ is closed in the mapping cone of the group-valued moment map
$\mu\colon M\to G$.

This relative viewpoint is more than a convenient packaging of two absolute
objects.  It supplies the correct chain complex, Stokes formula, period group,
and cohomological notion of exactness for bulk--boundary problems.  As a
result, topological actions, prequantization, conserved charges, and moment-map
rigidity can be derived from one common formalism rather than treated by
separate ad hoc arguments.

In \cite{DjThesis} the author developed the differential geometry of such
objects systematically: the mapping cone $\Om^\bullet(F)$ carries a full
\emph{relative Cartan calculus} in which all classical identities hold on the nose;
closed nondegenerate relative $(n+1)$-forms (\emph{relative $n$-plectic structures})
possess Lie $n$-algebras of \emph{relative observables} $\Lie(F,\rvp)$
generalizing the construction of Rogers \cite{Rogers}; and quasi-Hamiltonian
$G$-spaces fit the framework as relative $2$-plectic maps. The companion paper
\cite{RelHMM} developed the accompanying moment map theory -- \emph{relative
homotopy moment maps}, generalizing Callies--Fr\'egier--Rogers--Zambon \cite{CFRZ} --
including component equations, a relative Cartan model, explicit
extension-to-moment-map formulas, and an existence theory in which the Lie-algebra
cohomology obstruction of the absolute theory vanishes identically.

The present article is devoted to \emph{applications}. Our aim is to show that the
relative framework is not a formal generalization but a working tool: it isolates
mechanisms that are invisible, or must be treated ad hoc, in the absolute theory.
The applications rest on a small amount of shared infrastructure -- a relative
integration theory -- which we develop in full; the statements and proofs of
the applications themselves are complete modulo explicitly cited results from
the companion papers \cite{DjThesis,RelHMM} (the relative Cartan calculus, the
Lie $n$-algebra of observables, and the relative homotopy moment-map
equations) and from the classical quasi-Hamiltonian literature
\cite{AMM,MeinrenkenLectures}, and we indicate at each point of use exactly
what is imported.

The results of the paper are of several kinds --- new theorems, formal
consequences of the relative Cartan calculus, reinterpretations of known
constructions, and expository synthesis --- and it seems only fair to say
which is which at the outset:
\begin{center}
\renewcommand{\arraystretch}{1.25}
\begin{tabular}{@{}lp{0.26\textwidth}p{0.5\textwidth}@{}}
\hline
\textbf{Section} & \textbf{Topic} & \textbf{Status}\\
\hline
\ref{sec:integration} & Relative integration, periods, functoriality &
cycles, Stokes and periods are standard mapping-cone theory; the period
criterion \textup{(}Theorem~\ref{thm:period-criterion}\textup{)}, the level
group \textup{(}Proposition~\ref{prop:level-cyclic} and
Theorem~\ref{thm:level-group}\textup{)} and the functoriality
theorem \textup{(}Theorem~\ref{thm:functoriality}\textup{)} are new and are used
throughout\\
\ref{sec:actions} & Topological actions &
new systematic formulation; the converse
\textup{(}Theorem~\ref{thm:topterm}(iii)\textup{)} is proved here via
localized test fields\\
\ref{sec:prequantization} & Relative Weil--Kostant &
principal new theorem, with a complete \v{C}ech mapping-cone proof; the
Bohr--Sommerfeld specialization \textup{(}Theorem~\ref{thm:bohr-sommerfeld}\textup{)}
is classical in substance but, so far as we know, new in this formulation, and the relative-gerbe subsection recalls existing
literature\\
\ref{sec:noether} & Relative Noether &
new formal application of the relative Cartan calculus, with the scope
noted there\\
\ref{sec:rigidity} & Rigidity &
principal new structural theorem\\
\ref{sec:qham} & Fusion and moduli &
illustration: consequences and reinterpretations of \cite{AMM} and
\cite{RelHMM} on the unreduced representation space\\
App.~\ref{sec:edges} & The two edges &
consistency analysis; the specializations are classical, but the degeneration of
nondegeneracy \textup{(}Proposition~\ref{prop:edge-nondeg}\textup{)}, the degree
shift, and the sharpness of Theorem~\ref{thm:rigidity} are recorded here for the
first time\\
\ref{sec:nondeg} & Weak versus strong nondegeneracy &
new: the comparison theorem \textup{(}Theorem~\ref{thm:weak-vs-strong}\textup{)},
the ideal property of the relative kernel, the audit of which results need which
hypothesis, and the corrections of
Remarks~\ref{rem:fix-higher}--\ref{rem:fix-rigidity}\\
\hline
\end{tabular}
\end{center}

\subsection{The applications}
After recalling the necessary background (Section~\ref{sec:background}) and
constructing relative cycles, the relative Stokes theorem, and relative periods
(Section~\ref{sec:integration}), we prove:

\subsubsection*{0. Periods and functoriality (Section~\ref{sec:integration})}
Two results in the integration section are used by all five applications and are
stated here for reference. The \emph{period criterion}
(Theorem~\ref{thm:period-criterion}) shows that the group $P_\rvp$ of relative
periods is determined by the absolute periods $P_\omega$ of the target form
together with a homomorphism
\[
  \Theta_\rvp\colon
  \ker\bigl(F_*\colon H_n(M;\Z)\to H_n(N;\Z)\bigr)\longrightarrow\R/P_\omega,
\]
so that relative integrality becomes a finite check: $P_\omega\subseteq\Z$ and
$\Theta_\rvp$ integral. Iterating this over levels gives the \emph{level group}: the set of
$k\in\Z$ for which $k\rvp$ is relatively integral is always a subgroup
$k_0\Z$ of $\Z$, with no hypothesis whatever on $F$
(Proposition~\ref{prop:level-cyclic}), and under finite generation $k_0$ is
computed by finitely many integrals followed by one least common multiple
(Theorem~\ref{thm:level-group}). This explains
structurally why the admissible levels in the quasi-Hamiltonian literature
always form a cyclic group \cite{Krepski}, and isolates the dichotomy: either
some period is irrational and no level works at all, or all are rational and the
admissible levels are exactly the multiples of one integer. In particular the source contributes nothing when $F_*$
is injective in degree $n$, and everything when the target is acyclic in degree
$n{+}1$. The \emph{functoriality theorem} (Theorem~\ref{thm:functoriality})
records that a commutative square induces compatible maps on relative forms and
relative chains, adjoint for the pairing, and hence transports integrality,
prequantizations and moment maps; several later compatibility statements are
instances of it (Remark~\ref{rem:functoriality-uses}).

\subsubsection*{A. Topological terms of action functionals
(Section~\ref{sec:actions})}
For a compact oriented $(n{+}1)$-manifold $W$ with boundary, a \emph{relative field}
is a pair $u\colon W\to N$, $v\colon\partial W\to M$ with $F\circ v=u|_{\partial W}$;
its action is $S_\rvp(u,v)=\displaystyle \int_W u^*\omega-\int_{\partial W}v^*\eta$. We show
(Theorem~\ref{thm:topterm}) that a relative field is canonically a relative cycle,
that $S_\rvp$ is the pairing of $[\rvp]$ with this cycle, and hence that
$S_\rvp$ is a \emph{topological term}: invariant under all homotopies of relative
fields, and identically zero when $\rvp$ is $\dF$-exact. For a closed
source $\Sigma$ we prove (Theorem~\ref{thm:WZ}) a three-level well-definedness
statement that identifies the exact ambiguity subgroups: for a fixed boundary
field, $\exp(2\pi i\,\mathrm{WZ})$ is well defined precisely when the
\emph{absolute target periods} are integral; and the phase extends to all
relative cycles --- the condition the geometric realizations require ---
precisely under \emph{full relative integrality}, recovering, for
quasi-Hamiltonian spaces, the level-quantization criterion of group-valued
moment map theory \cite{MeinrenkenLectures,Krepski}.

\subsubsection*{B. Relative prequantization (Section~\ref{sec:prequantization})}
For $n=1$ we prove a relative Weil--Kostant theorem
(Theorem~\ref{thm:prequantization}): a closed relative $2$-form $(\omega,\eta)$
admits a \emph{relative prequantization} -- a Hermitian line bundle with connection
$(L,\nabla)$ on $N$ of curvature $-2\pi i\,\omega$ together with a unit section $s$
of $F^*L$ with $F^*\nabla\,s=-2\pi i\,\eta\otimes s$ -- if and only if the class
$[\rvp]\in H^2(\Om(F))$ is integral. Necessity is proved by a holonomy argument
over relative cycles; sufficiency by an explicit \v{C}ech construction on the
mapping cone. The set of equivalence classes of relative prequantizations, when nonempty, is a
torsor for $H^1(F;U(1))$. Specializing to the inclusion of a Lagrangian
submanifold $L\subseteq N$ with $\rvp=(\omega,0)$ recovers the
\emph{Bohr--Sommerfeld} condition of geometric quantization, the section datum
becoming a flat trivialization over $L$ and the torsor becoming
$H^1(N,L;U(1))$ (Theorem~\ref{thm:bohr-sommerfeld}); that a foundational notion of
geometric quantization is a one-line specialization is some evidence that the
framework is doing work. In degree $3$ the same integrality governs relative
gerbes and the prequantization of quasi-Hamiltonian $G$-spaces
\cite{ShahbaziGerbes,Shahbazi,Krepski,MeinrenkenLectures}, which we explain
(Section~\ref{subsec:gerbes}).

\subsubsection*{C. Conserved quantities and boundary charges
(Section~\ref{sec:noether})}
We prove a relative Noether theorem (Theorem~\ref{thm:noether}): if a Hamiltonian
$F$-pair $v_\sigma$ generates the dynamics and $v_x$ is a Hamiltonian symmetry
preserving $\sigma$, then $\Lder_{v_\sigma}f_1(x)$ is $\dF$-exact with explicit
primitive; consequently the \emph{charges} $Q_x(Z)=\pair{f_1(x)}{Z}$ over relative
$(n{-}1)$-cycles $Z$ are conserved (Corollary~\ref{cor:charges}). The splitting of a
relative cycle into an $N$-chain and an $M$-chain realizes the familiar
bulk-minus-boundary structure of physical charges.

\subsubsection*{D. Rigidity of relative comoment maps
(Section~\ref{sec:rigidity})}
In the relative symplectic case ($n=1$) we prove
(Theorem~\ref{thm:rigidity}) that when $N$ is connected and $M\neq\emptyset$, a
comoment map for a preserving action, if it exists, is \emph{unique},
\emph{automatically a strict Lie algebra morphism}, and \emph{automatically
infinitesimally equivariant}: the Kostant--Souriau cocycle vanishes because the
relative complex has no constants, $H^0(\Om(F))=0$. As an illustration we show that
the canonical equivariant moment map of a cotangent-lifted action is an instance of
this rigidity (Example~\ref{ex:cotangent}).

\subsubsection*{E. Quasi-Hamiltonian geometry, fusion, and moduli of flat
connections (Section~\ref{sec:qham})}
Every quasi-Hamiltonian $G$-space carries a canonical relative homotopy moment map,
whose data live entirely on the group \cite{RelHMM}. We prove that this canonical
structure is natural under fusion (Proposition~\ref{prop:fusion}) and under
equivariant maps of quasi-Hamiltonian spaces, and we spell out the consequence for
the moduli spaces of flat $G$-connections on compact surfaces with boundary
(Corollary~\ref{cor:moduli}), where the boundary holonomy is the group-valued moment
map \cite{AMM,AtiyahBott}.

\subsubsection*{Edges. Recovering the absolute theories
(Appendix~\ref{sec:edges})}
Finally we specialize everything to $M=\varnothing$ and to $N=\mathrm{pt}$. Each
application returns the expected absolute statement, but two features deserve
attention: the relative nondegeneracy condition is vacuous at the target edge,
so that edge recovers \emph{pre}-multisymplectic geometry rather than
multisymplectic geometry (Proposition~\ref{prop:edge-nondeg}), and at the source
edge the plectic degree drops by one while the pairing acquires a global sign
(Proposition~\ref{prop:edge-complexes}). The rigidity theorem \emph{fails} at
the target edge, which shows its hypotheses to be sharp rather than technical,
and the defect homomorphism vanishes at both edges, localizing the new content
of the theory strictly in the interior.

\subsubsection*{Nondegeneracy. Weak versus strong (Section~\ref{sec:nondeg})}
The nondegeneracy condition used throughout tests injectivity only along the
source. We compare it with the stronger requirement that the contraction be
injective on $F$-pairs, and prove that the difference is exactly the
nondegeneracy of $\omega$ off $\overline{\operatorname{im}F}$
(Theorem~\ref{thm:weak-vs-strong}). The relative kernel is an ideal annihilated
by every multicontraction (Proposition~\ref{prop:kernel-ideal}), so the weak
condition suffices for the whole $L_\infty$-package -- the brackets, the
moment-map component equations, the Noether identity and the charges -- while
the strong one is needed only where a Hamiltonian $F$-pair must be a vector
field on $N$. The two conditions coincide whenever $F$ has dense image, and in
particular for every quasi-Hamiltonian space over a semisimple group
(Corollary~\ref{cor:qham-nondeg}), which is why the distinction has not had to
be made before. Two proofs in this article use the strong condition tacitly;
we identify them and restate their hypotheses
(Remarks~\ref{rem:fix-higher} and \ref{rem:fix-rigidity}).

The logical structure of the paper is displayed in
Figure~\ref{fig:roadmap}: a single relative integration theory
(Section~\ref{sec:integration}) feeds all five applications, which
otherwise proceed independently, so the reader may take up any of
Sections~\ref{sec:actions}--\ref{sec:qham} directly after
Section~\ref{sec:integration}.
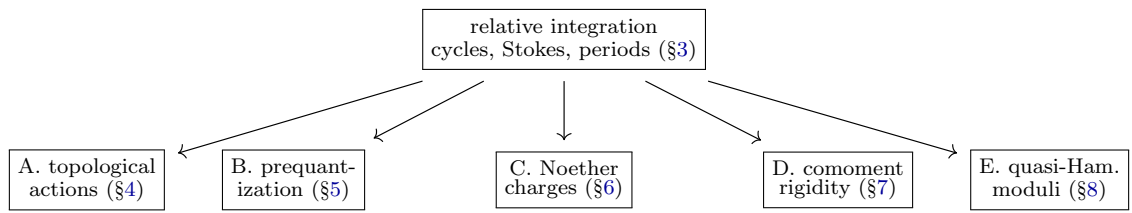
\begin{figure}[htb]
\centering
\begin{tikzcd}[column sep=1.1em, row sep=2.0em,
  /tikz/every label/.append style={font=\footnotesize}]
& & \boxed{\substack{\text{relative integration}\\
  \text{cycles, Stokes, periods (\S\ref{sec:integration})}}}
  \arrow[dll] \arrow[dl] \arrow[d] \arrow[dr] \arrow[drr] & & \\
\boxed{\substack{\text{A. topological}\\\text{actions (\S\ref{sec:actions})}}}
& \boxed{\substack{\text{B. prequant-}\\\text{ization (\S\ref{sec:prequantization})}}}
& \boxed{\substack{\text{C. Noether}\\\text{charges (\S\ref{sec:noether})}}}
& \boxed{\substack{\text{D. comoment}\\\text{rigidity (\S\ref{sec:rigidity})}}}
& \boxed{\substack{\text{E. quasi-Ham.}\\\text{moduli (\S\ref{sec:qham})}}}
\end{tikzcd}
\caption{The relative integration theory of
Section~\ref{sec:integration} -- cycles, Stokes, the period criterion and
functoriality -- is the shared engine; the five applications are logically
independent branches and may be read in any order.}
\label{fig:roadmap}
\end{figure}

\subsection{Relation to the literature}
The integration theory of Section~\ref{sec:integration} is classical homological
algebra adapted to the smooth mapping cone; its specializations recover
Poincar\'e--Lefschetz-type pairings \cite{BottTu}. The observation that Wess--Zumino
terms are governed by relative (co)cycles is implicit in the physics literature and
in the prequantization results of \cite{Shahbazi,Krepski,LGX,MeinrenkenLectures};
our contribution is to derive these mechanisms uniformly from the relative
multisymplectic calculus of \cite{DjThesis,RelHMM}, at arbitrary degree, with moment
maps included. The rigidity theorem of Section~\ref{sec:rigidity} sharpens, in the
relative setting, classical facts about moment map cocycles \cite{Kostant}; its
higher analogue is the obstruction-vanishing theorem of \cite{RelHMM}.

\subsection{Conventions}
We use throughout the conventions of \cite{CFRZ,RelHMM}, summarized in
\cite[App.~A]{RelHMM}: $\vs(k)=-(-1)^{k(k+1)/2}$; fundamental vector fields
$v_x|_p=\frac{d}{dt}\big|_0\exp(-tx)\cdot p$, so that $[v_x,v_y]=v_{[x,y]}$;
multicontractions $\io(v_1\wedge\dots\wedge v_k)=\io_{v_k}\cdots\io_{v_1}$. All
chains are smooth singular chains with $\Z$- or $\R$-coefficients as indicated.
The sign-sensitive formulas involving relative moment maps and higher
contractions --- the component equations of the relative homotopy moment map,
the Noether identity~\eqref{eq:noether}, and the fusion
formula (Proposition~\ref{prop:fusion}) --- have been re-verified against the two boundary
reductions $M=\varnothing$ (absolute multisymplectic geometry \cite{CFRZ,Rogers})
and $N=\mathrm{pt}$ (the source theory of \cite{DjThesis}), on which they must
return the known signs; each reduces correctly. These two reductions are
analysed systematically in Section~\ref{sec:edges}, which is where the
verification just described is carried out.

\subsection{Notation}\label{subsec:notation}
The following symbols are used throughout. Homology and cohomology of spaces
carry $\Z$ coefficients unless indicated otherwise; period groups and
integration pairings take values in $\R$.

\begin{center}
\renewcommand{\arraystretch}{1.25}
\begin{tabular}{@{}ll@{}}
\hline
$\Om^k(F)=\Om^k(N)\oplus\Om^{k-1}(M)$ & relative $k$-forms, the mapping cone of $F^*$\\
$\dF(\alpha,\beta)=(d\alpha,\,F^*\alpha-d\beta)$ & relative de Rham differential\\
$\rvp=(\omega,\eta)$ & a $\dF$-closed relative $(n{+}1)$-form\\
$C_k(F)=C_k(N)\oplus C_{k-1}(M)$ & smooth relative chains, $\Z$ coefficients\\
$\pair{\cdot}{\cdot}$ & integration pairing, $\pair{(\alpha,\beta)}{(S,T)}=\int_S\alpha-\int_T\beta$\\
$P_\omega\subseteq\R$ & absolute periods of $\omega$ on $H_{n+1}(N;\Z)$\\
$P_\rvp\subseteq\R$ & relative periods of $\rvp$ on $Z_{n+1}(F)$\\
$K=\ker\bigl(F_*\colon H_n(M;\Z)\to H_n(N;\Z)\bigr)$ & source classes dying in the target\\
$\Theta_\rvp\colon K\to\R/P_\omega$ & defect homomorphism\\
$L(\rvp)=k_0\Z\subseteq\Z$ & level group and prequantization level\\
$\XF$ & $F$-pairs of vector fields\\
$\mathfrak K(F,\rvp)$ & relative kernel\\
$\Lie(F,\rvp)$ & $L_\infty$-algebra of relative observables\\
$H^1(F;U(1))$ & flat relative data; torsor for prequantizations\\
\hline
\end{tabular}
\end{center}

\section{Background from relative multisymplectic geometry}
\label{sec:background}

\begin{remark}[Bibliographical notes]\label{rem:biblio-background}
Multisymplectic geometry in the sense used here goes back to
\cite{CIL,GIMMSY}; the $L_\infty$-algebra of observables is due to Rogers
\cite{Rogers,RogersThesis}, with the Lie $2$-algebra case treated in
\cite{BHR,BaezRogers} and the general homotopy moment map in \cite{CFRZ,FLZ}.
For a survey we recommend \cite{RWinvitation}. Multi-moment maps for closed
forms are developed independently in \cite{MS}, existence and uniqueness of
comoments in \cite{RW}, conserved quantities in \cite{RWZ}, and reduction in
\cite{Blacker,BMR}. The relative theory used below was developed in
\cite{DjThesis,RelHMM,RelRed}.
\end{remark}

We briefly recall the objects and results from \cite{DjThesis,RelHMM} used in the
sequel; proofs may be found there.

\subsection{The relative Cartan calculus}
Let $F\colon M\to N$ be smooth. The \emph{relative de Rham complex} is the mapping
cone
\begin{equation}\label{eq:cone}
  \Om^k(F):=\Om^k(N)\oplus\Om^{k-1}(M),
  \qquad
  \dF(\alpha,\beta):=(d\alpha,\;F^*\alpha-d\beta),
\end{equation}
whose cohomology $H^\bullet(\Om(F))$ sits in a long exact sequence
\begin{equation}\label{eq:LES}
  \cdots\to H^{k-1}_{\mathrm{dR}}(M)\to H^{k}(\Om(F))\to
  H^{k}_{\mathrm{dR}}(N)\xrightarrow{F^*}H^{k}_{\mathrm{dR}}(M)\to\cdots .
\end{equation}
An \emph{$F$-pair} of vector fields is $v=(v_N,v_M)$ with $v_M,v_N$ $F$-related; the
space $\XF$ of $F$-pairs is a Lie algebra under the componentwise bracket. The
\emph{relative contraction} and \emph{Lie derivative}
\begin{equation}\label{eq:reliota}
  \io_v(\alpha,\beta)=(\io_{v_N}\alpha,\,-\io_{v_M}\beta),
  \qquad
  \Lder_v(\alpha,\beta)=(\Lder_{v_N}\alpha,\,\Lder_{v_M}\beta),
\end{equation}
satisfy all Cartan identities, in particular the magic formula
$\Lder_v=\dF\io_v+\io_v\dF$ and $[\Lder_u,\io_v]=\io_{[u,v]}$
\cite[Prop.~3.5]{RelHMM}, together with the \emph{master identity}
\cite[Lem.~3.7]{RelHMM}, of which we will use the closed-invariant case: if
$\dF\Psi=0$ and $\Lder_{v_i}\Psi=0$, then
\begin{equation}\label{eq:master}
  \dF\,\io(v_1\wedge\dots\wedge v_m)\Psi
  =(-1)^m\!\!\sum_{1\le i<j\le m}\!\!(-1)^{i+j}
  \io\bigl([v_i,v_j]\wedge v_1\wedge\dots\widehat{v_i}\dots\widehat{v_j}\dots\wedge
  v_m\bigr)\Psi .
\end{equation}

The following elementary but decisive fact will be used repeatedly
\cite[Rem.~3.2]{RelHMM}:
\begin{equation}\label{eq:H0}
  \Om^0(F)=C^\infty(N),\quad
  \dF h=(dh,F^*h);\qquad
  \text{$N$ connected, $M\neq\emptyset$}\ \Longrightarrow\ H^0(\Om(F))=0 .
\end{equation}

Before proceeding we record the dictionary that organizes the rest of the
article. Each row is an instance of the same principle: an absolute object on a
single manifold is replaced by a cocycle in the mapping cone of $F$, and each
classical theorem has a relative counterpart in which the target and source
contributions appear with opposite signs.

\begin{center}
\renewcommand{\arraystretch}{1.28}
\begin{tabular}{@{}p{0.30\textwidth}p{0.36\textwidth}p{0.26\textwidth}@{}}
\hline
\textbf{Absolute} & \textbf{Relative} & \textbf{Where} \\
\hline
form $\alpha\in\Om^k(P)$ &
pair $(\alpha,\beta)\in\Om^k(N)\oplus\Om^{k-1}(M)$ &
\eqref{eq:cone}\\
$d\alpha=0$ &
$d\alpha=0$ and $F^*\alpha=d\beta$ &
\eqref{eq:cone}\\
chain $S\in C_k(P)$ &
pair $(S,T)\in C_k(N)\oplus C_{k-1}(M)$ &
Def.~\ref{def:relchains}\\
$\int_S\alpha$ &
$\int_S\alpha-\int_T\beta$ &
\eqref{eq:pairing}\\
Stokes &
relative Stokes &
Prop.~\ref{prop:stokes}\\
periods of $\alpha$ &
$P_\omega$ and the defect homomorphism $\Theta_\rvp$ &
Thm.~\ref{thm:period-criterion}\\
Weil--Kostant &
relative Weil--Kostant &
Thm.~\ref{thm:prequantization}\\
prequantum bundle &
bundle on $N$ $+$ trivialization on $M$ &
Def.~\ref{def:relprequant}\\
Noether charge &
bulk charge $-$ boundary charge &
Cor.~\ref{cor:charges}\\
Kostant--Souriau cocycle &
absent, since $H^0(\Om(F))=0$ &
Thm.~\ref{thm:rigidity}\\
\hline
\end{tabular}
\end{center}

\subsection{Relative $n$-plectic structures, observables, moment maps}
A \emph{relative pre-$n$-plectic structure} is a $\dF$-closed
$\rvp=(\omega,\eta)\in\Om^{n+1}(F)$; it is \emph{$n$-plectic} if
$w\mapsto(\io_{T F(w)}\omega,\io_w\eta)$ is injective on each $T_mM$. A relative
$(n{-}1)$-form $\sigma$ is \emph{Hamiltonian} if $\dF\sigma=-\io_{v_\sigma}\rvp$
for some $F$-pair $v_\sigma$; Hamiltonian forms are the degree-zero part of a Lie
$n$-algebra $\Lie(F,\rvp)$ with $l_1=\dF$ and
$l_k=\vs(k)\,\io(v_{\sigma_1}\wedge\dots\wedge v_{\sigma_k})\rvp$ in degree zero
\cite{DjThesis}. If a Lie group $G$ acts on $M$ and $N$ with $F$ equivariant,
preserving $\rvp$, a \emph{relative homotopy moment map} is an
$L_\infty$-morphism $(f)\colon\g\rightsquigarrow\Lie(F,\rvp)$ lifting
$x\mapsto v_x=(v_x^N,v_x^M)$; equivalently \cite[Prop.~5.3]{RelHMM}, a family
$f_k\colon\Lambda^k\g\to\Om^{n-k}(F)$, $1\le k\le n$, with
\begin{equation}\label{eq:componenteqs}
  \sum_{i<j}(-1)^{i+j+1}
  f_{k-1}([x_i,x_j],\dots)
  =\dF f_k(x_1,\dots,x_k)+\vs(k)\,\io(v_{x_1}\wedge\dots\wedge v_{x_k})\rvp,
  \qquad 1\le k\le n+1 .
\end{equation}
A \emph{one-step extension} of $\rvp$ is a $G$-equivariant linear map
$\mathrm{M}\colon\g\to\Om^{n-1}(F)$ with $\dF\mathrm{M}(x)=-\io_{v_x}\rvp$ and
$\io_{v_x}\mathrm{M}(x)=0$; it induces the relative homotopy moment map
\cite[Thm.~6.5]{RelHMM}
\begin{equation}\label{eq:onestep}
  f_k(x_1,\dots,x_k)=\vs(k)\,
  \io(v_{x_1}\wedge\dots\wedge v_{x_{k-1}})\,\mathrm{M}(x_k).
\end{equation}
Finally, if $N$ is connected, $M\neq\emptyset$ and $H^i(\Om(F))=0$ for
$1\le i\le n-1$, every Hamiltonian lift extends to a relative homotopy moment map,
with no Lie-algebra cohomology obstruction \cite[Thm.~7.2]{RelHMM}. For the
absolute theory the corresponding obstruction is analysed cohomologically in
\cite{FLZ}, and the closely related theory of multi-moment maps for closed forms
is developed in \cite{MS}; the relative statements below should be read against
those benchmarks.

\subsection{Quasi-Hamiltonian $G$-spaces}
Let $G$ be compact connected with an $\Ad$-invariant inner product, $\eta$ the
Cartan $3$-form, and $\widetilde\mu(x)=\tfrac12\ip{\thL+\thR}{x}$, so that
$\tfrac12 d\ip{\thL+\thR}{x}=-\io_{v_x}\eta$ for the conjugation action. A
quasi-Hamiltonian $G$-space $(M,\omega,\mu)$ \cite{AMM} -- axioms
(QH1) $d\omega=\mu^*\eta$, (QH2) $\io_{v_x}\omega=\tfrac12\mu^*\ip{\thL+\thR}{x}$,
(QH3) minimal degeneracy, in the conventions of \cite[\S8]{RelHMM} -- is the same
thing as a relative $2$-plectic map $(\mu,\rvp)$, $\rvp=(\eta,\omega)$, for
which
\begin{equation}\label{eq:qhamM}
  \mathrm{M}(x)=\bigl(\widetilde\mu(x),\,0\bigr)
\end{equation}
is a one-step extension; the induced canonical relative homotopy moment map is
\begin{equation}\label{eq:qhamf}
  f_1(x)=\bigl(\widetilde\mu(x),0\bigr),
  \qquad
  f_2(x,y)=\Bigl(\tfrac12\ip{(\Ad_{(\cdot)}-\Ad_{(\cdot)^{-1}})x}{y},\,0\Bigr)
\end{equation}
\cite[Thms.~8.4--8.5]{RelHMM}.

\section{Relative cycles, integration, and periods}
\label{sec:integration}

The homological algebra of mapping cones is standard; we follow \cite{Weibel}
for the algebraic side and \cite{Hatcher,BottTu} for the topological one, and we
fix conventions rather than reprove.

We now construct the homological counterpart of \eqref{eq:cone}. Let $C_k(P)$
denote smooth singular $k$-chains on a manifold $P$ (coefficients in $\Z$ unless
stated otherwise), with boundary $\partial$.

\begin{definition}\label{def:relchains}
The \emph{relative chain complex} of $F\colon M\to N$ is
\begin{equation}\label{eq:relchains}
  C_k(F):=C_k(N)\oplus C_{k-1}(M),
  \qquad
  \partial_F(S,T):=\bigl(\partial S-F_*T,\;-\partial T\bigr).
\end{equation}
A \emph{relative $k$-cycle} is a pair $(S,T)$ with $\partial S=F_*T$ and
$\partial T=0$; we write $Z_k(F)$ for the group of relative cycles and
$H_k(F)$ for the homology of $(C_\bullet(F),\partial_F)$.
\end{definition}

\begin{lemma}\label{lem:relchains}
$\partial_F^2=0$, and $H_\bullet(F)$ computes the singular homology of the mapping
cone of $F$; in particular, when $F$ is the inclusion of a closed subspace,
$H_k(F)\cong H_k(N,M;\Z)$.
\end{lemma}

\begin{proof}
$\partial_F^2(S,T)=(\partial(\partial S-F_*T)+F_*\partial T,\;\partial^2T)
=(-F_*\partial T+F_*\partial T,\,0)=0$, using $\partial F_*=F_*\partial$. The
complex \eqref{eq:relchains} is the algebraic mapping cone of
$F_*\colon C_\bullet(M)\to C_\bullet(N)$ (up to the sign of the second component,
which is an isomorphism of complexes), whose homology is the reduced homology of
the topological mapping cone by standard homological algebra; for closed
inclusions this is $H_\bullet(N,M)$ by excision.
\end{proof}

\begin{definition}\label{def:pairing}
The \emph{relative integration pairing} is
\begin{equation}\label{eq:pairing}
  \pair{\cdot}{\cdot}\colon\ \Om^k(F)\times C_k(F)\longrightarrow\R,
  \qquad
  \pair{(\alpha,\beta)}{(S,T)}:=\int_S\alpha-\int_T\beta .
\end{equation}
\end{definition}

\begin{proposition}[Relative Stokes theorem]\label{prop:stokes}
For all $(\alpha,\beta)\in\Om^{k-1}(F)$ and $(S,T)\in C_k(F)$,
\begin{equation}\label{eq:stokes}
  \pair{\dF(\alpha,\beta)}{(S,T)}
  \;=\;
  \pair{(\alpha,\beta)}{\partial_F(S,T)} .
\end{equation}
Consequently the pairing descends to
$H^k(\Om(F))\times H_k(F)\to\R$. The value of a closed relative form on a relative
cycle is called a \emph{relative period}; for $[\rvp]\in H^{n+1}(\Om(F))$ we
write
\begin{equation}\label{eq:periods}
  P_\rvp:=\bigl\{\pair{\rvp}{c}\ :\ c\in Z_{n+1}(F)\bigr\}\subseteq\R
\end{equation}
for the (sub)group of relative periods, and call $\rvp$ \emph{relatively
integral} if $P_\rvp\subseteq\Z$.
\end{proposition}

\begin{proof}
Using the classical Stokes theorem on $N$ and on $M$ and
$\displaystyle \int_T F^*\alpha=\int_{F_*T}\alpha$,
\begin{align*}
\pair{\dF(\alpha,\beta)}{(S,T)}
&=\int_S d\alpha-\int_T\bigl(F^*\alpha-d\beta\bigr)
 =\int_{\partial S}\alpha-\int_{F_*T}\alpha+\int_{\partial T}\beta\\
&=\int_{\partial S-F_*T}\alpha-\int_{-\partial T}\beta
 =\pair{(\alpha,\beta)}{\partial_F(S,T)} .
\end{align*}
Descent to (co)homology is immediate: closed forms annihilate boundaries and exact
forms annihilate cycles by \eqref{eq:stokes}.
\end{proof}

\begin{remark}[Poincar\'e--Lefschetz]\label{rem:lefschetz}
For the inclusion $i\colon\partial W\hookrightarrow W$ of the boundary of a compact
oriented $(n{+}1)$-manifold, $H^{k}(\Om(i))\cong H^k(W,\partial W;\R)$
(de Rham-type theorem for the cone, cf.\ \cite{BottTu}), the pair
$([W],[\partial W])$ is a relative $(n{+}1)$-cycle by definition of the boundary
orientation, and the pairing \eqref{eq:pairing} against it implements evaluation on
the relative fundamental class; nondegeneracy of the resulting pairing between
$H^k(W,\partial W)$ and $H_k(W,\partial W)$ is Lefschetz duality. Thus the
integration theory of this section contains the classical relative pairings as the
special case of inclusions.
\end{remark}

\begin{remark}\label{rem:integralclass}
Since the mapping cone of a smooth map between manifolds has the homotopy type of a
CW complex with finitely generated homology in each degree (for $M,N$ of finite
type), the universal coefficient theorem identifies relative integrality of
$\rvp$ with the condition that $[\rvp]$ lie in the image of
$H^{n+1}(F;\Z)\to H^{n+1}(F;\R)\cong H^{n+1}(\Om(F))$. We use the two formulations
interchangeably for spaces of finite type.
\end{remark}
\subsection{Computing relative periods}\label{subsec:computing}

Relative integrality is the hypothesis of every existence statement below
(Theorems~\ref{thm:WZ} and~\ref{thm:prequantization}), yet Definition
\eqref{eq:periods} quantifies over all relative cycles and is not directly
checkable. We now reduce it to two pieces of data: the ordinary periods of
$\omega$ on $N$, and a single homomorphism defined on the kernel of $F_*$.

Throughout, $\rvp=(\omega,\eta)\in\Om^{n+1}(F)$ is $\dF$-closed, so that
\begin{equation}\label{eq:closedagain}
  d\omega=0,\qquad d\eta=F^*\omega .
\end{equation}
Write
\[
  P_\omega:=\Bigl\{\int_S\omega\ :\ S\in Z_{n+1}(N;\Z)\Bigr\}\subseteq\R
\]
for the group of \emph{absolute} periods of $\omega$ on the target.

\begin{lemma}\label{lem:cycle-shape}
If $(S,T)\in Z_{n+1}(F)$ then $T\in Z_n(M;\Z)$ and $F_*[T]=0$ in
$H_n(N;\Z)$. Conversely, every $T\in Z_n(M;\Z)$ with $F_*[T]=0$ occurs as the
source component of a relative $(n{+}1)$-cycle, and the possible first components
form a coset of $Z_{n+1}(N;\Z)$.
\end{lemma}

\begin{proof}
By Definition~\ref{def:relchains}, $\partial T=0$ and $\partial S=F_*T$, so $F_*T$
bounds. Conversely if $F_*T=\partial S$ for some $S\in C_{n+1}(N)$ then $(S,T)$ is a
relative cycle, and $S$ is determined up to a cycle, since $\partial S=\partial S'$
forces $S-S'\in Z_{n+1}(N)$.
\end{proof}

\begin{theorem}[Period criterion]\label{thm:period-criterion}
Let $\rvp=(\omega,\eta)$ be $\dF$-closed and set
$K:=\ker\bigl(F_*\colon H_n(M;\Z)\to H_n(N;\Z)\bigr)$. Then the assignment
\begin{equation}\label{eq:defect}
  \Theta_\rvp\colon K\longrightarrow \R/P_\omega,
  \qquad
  \Theta_\rvp[T]:=\Bigl(\int_S\omega-\int_T\eta\Bigr)\bmod P_\omega
  \quad(\partial S=F_*T),
\end{equation}
is a well-defined group homomorphism, which we call the \emph{relative defect
homomorphism}. Moreover:
\begin{enumerate}[label=\textup{(\roman*)},leftmargin=2.2em]
\item $P_\omega\subseteq P_\rvp$, and $P_\rvp$ is the full preimage of
$\operatorname{im}\Theta_\rvp$ under the projection $\R\to\R/P_\omega$;
\item $\rvp$ is relatively integral if and only if
\[
  P_\omega\subseteq\Z
  \qquad\text{and}\qquad
  \operatorname{im}\Theta_\rvp\subseteq\Z/P_\omega ;
\]
\item if $F_*\colon H_n(M;\Z)\to H_n(N;\Z)$ is injective, then
$P_\rvp=P_\omega$, and relative integrality is equivalent to ordinary
integrality of $\omega$ on $N$: the source contributes nothing;
\item if $H_{n+1}(N;\Z)=0$, then $P_\omega=0$, the defect homomorphism takes
values in $\R$, and $\rvp$ is relatively integral if and only if
$\Theta_\rvp(K)\subseteq\Z$.
\end{enumerate}
\end{theorem}

\begin{proof}
\emph{Well-definedness.} Fix $T\in Z_n(M;\Z)$ with $F_*[T]=0$. Changing the filling
$S$ to $S'$ changes $\int_S\omega$ by a period of $\omega$ (Lemma
\ref{lem:cycle-shape}), hence does not change the class modulo $P_\omega$. Changing
$T$ within its homology class, say $T'=T+\partial U$ with $U\in C_{n+1}(M)$, we may
take $S'=S+F_*U$, since $\partial(S+F_*U)=F_*T+F_*\partial U=F_*T'$. Then, using
\eqref{eq:closedagain} and Stokes on $M$,
\[
  \int_{S'}\omega-\int_{T'}\eta
  =\int_S\omega+\int_U F^*\omega-\int_T\eta-\int_{\partial U}\eta
  =\int_S\omega+\int_U d\eta-\int_T\eta-\int_U d\eta,
\]
which is $\int_S\omega-\int_T\eta$. So $\Theta_\rvp$ depends only on $[T]$, and it
is additive because fillings may be chosen additively.

\emph{(i).} Taking $T=0$ shows $(S,0)$ is a relative cycle for every
$S\in Z_{n+1}(N)$, with period $\int_S\omega$; hence $P_\omega\subseteq P_\rvp$.
Conversely every relative cycle has the shape of Lemma~\ref{lem:cycle-shape}, and
its period $\int_S\omega-\int_T\eta$ is by definition a representative of
$\Theta_\rvp[T]$; as $S$ ranges over all fillings, these representatives sweep out
the whole coset. So $P_\rvp$ is exactly the union of the cosets in
$\operatorname{im}\Theta_\rvp$, which is the asserted preimage.

\emph{(ii).} By (i), $P_\rvp\subseteq\Z$ forces $P_\omega\subseteq\Z$, and then
$P_\rvp\subseteq\Z$ is equivalent to every coset of $\operatorname{im}\Theta_\rvp$
lying inside $\Z$, i.e.\ to $\operatorname{im}\Theta_\rvp\subseteq\Z/P_\omega$. The
converse is immediate from the same description.

\emph{(iii).} If $F_*$ is injective on $H_n$ then $K=0$, so
$\operatorname{im}\Theta_\rvp=0$ and $P_\rvp=P_\omega$ by (i).

\emph{(iv).} If $H_{n+1}(N;\Z)=0$ then every $(n{+}1)$-cycle in $N$ bounds, so
$P_\omega=0$ by Stokes and \eqref{eq:closedagain}, and $\R/P_\omega=\R$.
\end{proof}

\begin{remark}[Reading the criterion]\label{rem:reading-criterion}
Theorem~\ref{thm:period-criterion} separates the two obstructions that the
three-level statement of Theorem~\ref{thm:WZ} keeps apart by hand: integrality of
$P_\omega$ is the \emph{target} condition, governing the ambiguity in changing a
bulk filling, while $\Theta_\rvp$ is the genuinely \emph{relative} correction,
supported on the source classes that die in the target. Part (iii) is the reason
the theory has no content when $F$ is homologically injective in the relevant
degree, and part (iv) is the reason it has maximal content when the target is
homologically trivial in that degree, as for a disc or a contractible bulk.
\end{remark}

\begin{example}[Flux through a disc against a boundary Wilson line]
\label{ex:disc}
Let $N=D^2$ be the closed disc, $M=S^1=\partial D^2$, and $F$ the inclusion; take
$n=1$. A closed relative $2$-form is a pair $(\omega,\eta)$ with $\omega$ a $2$-form
on $D^2$ and $\eta\in\Om^1(S^1)$ satisfying $d\eta=F^*\omega=0$, the latter
automatically since a $2$-form restricts to zero on a $1$-manifold. Here
$H_2(D^2;\Z)=0$, so $P_\omega=0$, and $K=\ker\bigl(H_1(S^1)\to H_1(D^2)\bigr)=\Z$,
generated by the boundary circle, whose filling is $D^2$ itself. Thus
\[
  \Theta_\rvp\bigl[\,S^1\,\bigr]
  \;=\;
  \int_{D^2}\omega-\int_{S^1}\eta ,
\]
and by Theorem~\ref{thm:period-criterion}(iv) the pair is relatively integral
precisely when this number is an integer. Physically: the magnetic flux through
the disc minus the holonomy of the boundary connection must be quantized. The
absolute theory sees neither condition, since $\omega$ has no periods on a disc
and $\eta$ is a closed $1$-form on a circle with arbitrary period; only the
relative combination is constrained.
\end{example}

\begin{example}[A homologically injective inclusion]\label{ex:torus}
Let $N=T^2=\R^2/\Z^2$ with $\omega=\theta\,dx\wedge dy$ for $\theta\in\R$, let
$M=S^1$ be the circle $\{y=0\}$ with $F$ the inclusion, and $\eta=0$; then
$d\eta=0=F^*\omega$. The map $F_*\colon H_1(S^1;\Z)\to H_1(T^2;\Z)$ sends the
generator to a primitive class, hence is injective, so
Theorem~\ref{thm:period-criterion}(iii) gives $P_\rvp=P_\omega=\theta\,\Z$, and
relative integrality holds exactly when $\theta\in\Z$. Comparing with
Example~\ref{ex:disc} shows how sharply the answer depends on the homological
position of the source, not on the size of either manifold.
\end{example}

\subsection{The level group}\label{subsec:levelgroup}

Theorem~\ref{thm:period-criterion} makes the relative period group computable.
We now record what that computation yields, and it yields more than a decision
procedure: it shows that the set of levels at which a relative form can be
prequantized is always a subgroup of $\Z$, and identifies its generator.

In the applications one rarely prequantizes $\rvp$ itself but rather a
multiple $k\rvp$, the integer $k$ being the \emph{level}. Since both
components of $\rvp$ scale linearly, so does every relative period, and the
following definition is natural.

\begin{definition}\label{def:level-group}
The \emph{level group} of a $\dF$-closed relative form $\rvp$ is
\[
  L(\rvp) \;:=\; \bigl\{\, k\in\Z \;:\; k\rvp \text{ is relatively integral}\,\bigr\}
  \;=\; \bigl\{\, k\in\Z \;:\; k\,P_\rvp\subseteq\Z \,\bigr\}.
\]
\end{definition}

\begin{proposition}[Cyclicity of the level group]\label{prop:level-cyclic}
Let $\rvp\in\Om^{n+1}(F)$ be $\dF$-closed. Then $L(\rvp)$ is a subgroup of $\Z$;
consequently there is a unique integer $k_0=k_0(\rvp)\geq0$ with
\[
  L(\rvp)\;=\;k_0\,\Z ,
\]
called the \emph{prequantization level} of $\rvp$. No hypothesis on $M$, $N$ or
$F$ is required.
\end{proposition}

\begin{proof}
For a relative cycle $(S,T)$ one has $\pair{k\rvp}{(S,T)}=k\pair{\rvp}{(S,T)}$,
so $P_{k\rvp}=k\,P_\rvp$ and $k\in L(\rvp)$ if and only if
$kP_\rvp\subseteq\Z$. Now $0\in L(\rvp)$, and if $k,k'\in L(\rvp)$ then for every
$p\in P_\rvp$ we have $(k-k')p=kp-k'p\in\Z$, so $k-k'\in L(\rvp)$. Thus
$L(\rvp)$ is a subgroup of $\Z$, hence of the form $k_0\Z$ for a unique
$k_0\geq0$.
\end{proof}

\begin{convention}\label{conv:k0-zero}
The value $k_0=0$ encodes $L(\rvp)=\{0\}$, that is, the case in which
\emph{no nonzero level} is relatively integral. It should not be read as saying
that ``level zero is admissible''; the zero form is of course integral, but the
statement carries no content. Whenever $k_0\geq1$ the admissible levels are
exactly the nonzero multiples of $k_0$, and $k_0$ itself is the smallest
positive admissible level.
\end{convention}

Proposition~\ref{prop:level-cyclic} is qualitative: it says the answer has the
shape $k_0\Z$ but not what $k_0$ is. Making $k_0$ computable is where the period
criterion enters, and where finiteness hypotheses become necessary.

\begin{theorem}[Computation of the prequantization level]\label{thm:level-group}
Assume in addition that $H_{n+1}(N;\Z)$ and
$K=\ker\bigl(F_*\colon H_n(M;\Z)\to H_n(N;\Z)\bigr)$ are finitely generated
\textup{(}for instance, $M$ and $N$ of finite type, in particular compact\textup{)}.
Choose classes $S_1,\dots,S_a\in H_{n+1}(N;\Z)$ whose images generate the free
quotient $H_{n+1}(N;\Z)/\mathrm{Tor}$, and classes $T_1,\dots,T_b$ generating
$K$ \emph{in full, torsion included} \textup{(}see
Remark~\ref{rem:torsion-asymmetry}\textup{)}. Set
\[
  r_i:=\int_{S_i}\omega \quad (1\leq i\leq a),
  \qquad
  r_{a+j}:=\int_{S}\omega-\int_{T_j}\eta \quad (1\leq j\leq b),
\]
the second computed from any singular chain $S$ with $\partial S=F_*T_j$. Then:
\begin{enumerate}[label=\textup{(\roman*)},leftmargin=2.2em]
\item $P_\rvp=\langle r_1,\dots,r_{a+b}\rangle$ is a finitely generated subgroup
of $\R$; in particular the choice of the chains $S$ affects the $r_{a+j}$ only
by elements of $\langle r_1,\dots,r_a\rangle$, hence not the group they
generate;
\item $k_0=0$ if and only if some $r_i$ is irrational; in that case no nonzero
level admits a relative prequantization;
\item if every $r_i$ is rational, write the nonzero ones as $p_i/q_i$ in lowest
terms with $q_i\geq1$ and discard the $r_i$ that vanish. Then
\[
  k_0=\operatorname{lcm}\{q_i : r_i\neq0\},
\]
with the convention $k_0=1$ if every $r_i=0$. In particular $k_0$ is determined
by at most $a+b$ integrals followed by one arithmetic computation.
\end{enumerate}
\end{theorem}

\begin{proof}
\emph{Torsion in the target.} The map $[S]\mapsto\int_S\omega$ is a homomorphism
$H_{n+1}(N;\Z)\to\R$; since $\R$ is torsion-free it kills
$\mathrm{Tor}\,H_{n+1}(N;\Z)$ and factors through the free quotient. Hence
$P_\omega=\langle r_1,\dots,r_a\rangle$, and torsion classes in $H_{n+1}(N;\Z)$
may indeed be discarded.

(i) By Theorem~\ref{thm:period-criterion}(i), $P_\rvp$ is the full preimage of
$\operatorname{im}\Theta_\rvp$ under $\R\to\R/P_\omega$; equivalently, $P_\rvp$
is generated by $P_\omega$ together with one real lift of $\Theta_\rvp(T)$ for
each $T$ in a generating set of $K$. By the previous paragraph
$P_\omega=\langle r_1,\dots,r_a\rangle$, and by the previous paragraph again the
classes $T_1,\dots,T_b$ suffice. Changing the filling $S$ of $F_*T_j$ changes
$r_{a+j}$ by a period of $\omega$, that is, by an element of
$\langle r_1,\dots,r_a\rangle$, which is already in the group.

(ii) If some $r_i$ is irrational then $kr_i\notin\Z$ for every $k\neq0$, so
$L(\rvp)=\{0\}$ and $k_0=0$ by Proposition~\ref{prop:level-cyclic}. Conversely
if all $r_i$ are rational then any common denominator lies in $L(\rvp)$, which
is therefore nonzero.

(iii) With all $r_i$ rational, $kP_\rvp\subseteq\Z$ holds if and only if
$kr_i\in\Z$ for every $i$. A vanishing $r_i$ imposes no condition, whence the
restriction to $r_i\neq0$. For $r_i=p_i/q_i$ with $\gcd(p_i,q_i)=1$ the
condition $kp_i/q_i\in\Z$ is equivalent to $q_i\mid k$. Hence
$L(\rvp)=\{k : q_i\mid k \text{ for all } i \text{ with } r_i\neq0\}
=\operatorname{lcm}\{q_i\}\,\Z$. If every $r_i=0$ then $P_\rvp=0$, every $k$ is
admissible, $L(\rvp)=\Z$ and $k_0=1$.
\end{proof}

\begin{remark}[The two torsions behave differently]\label{rem:torsion-asymmetry}
It is tempting to symmetrize Theorem~\ref{thm:level-group} by passing to the
free quotient on both sides. That would be wrong, and the asymmetry is worth
stating explicitly.

Torsion in $H_{n+1}(N;\Z)$ may be discarded, because $[S]\mapsto\int_S\omega$
takes values in the torsion-free group $\R$.

Torsion in $K$ may \emph{not} be discarded, because $\Theta_\rvp$ takes values in
$\R/P_\omega$, which has torsion as soon as $P_\omega\neq0$. Concretely, if
$T\in K$ has order $m$ then $m\Theta_\rvp[T]=\Theta_\rvp[mT]=0$, so
$\Theta_\rvp[T]$ is an $m$-torsion element of $\R/P_\omega$ and its real lifts
lie in $\tfrac1m P_\omega$ --- a group strictly larger than $P_\omega$ whenever
$\Theta_\rvp[T]\neq0$. Such a class therefore enlarges $P_\rvp$ and multiplies
$k_0$ by a divisor of $m$. For instance, with $P_\omega=\Z$ and
$\Theta_\rvp[T]=\tfrac12+\Z$ for a $2$-torsion class $T$, the free-quotient
computation would return $k_0=1$ whereas in fact $\tfrac12\in P_\rvp$ and
$k_0=2$.

The asymmetry has a clean source: the target contributes through a homomorphism
to $\R$, the source through a homomorphism to a quotient of $\R$, and only the
former is torsion-free. It is also invisible at both edges, since $K=0$ at the
target edge and $P_\omega=0$ at the source edge
\textup{(}Corollary~\ref{cor:edge-periods}\textup{)}, so it can only be detected
in the genuinely relative range.
\end{remark}

\begin{remark}[What the theorem buys]\label{rem:level-group-value}
Three things. First, it explains a phenomenon visible throughout the
quasi-Hamiltonian literature: whenever one asks at which levels a given space
admits a prequantization, the answer is always ``the integer multiples of some
$k_0$'' rather than an arbitrary set of integers, and
Proposition~\ref{prop:level-cyclic} shows this is forced, being nothing but the
statement that a subgroup of $\Z$ is cyclic. It is worth stressing that this
half requires no hypothesis whatever on $M$, $N$ or $F$; only the
\emph{computation} of $k_0$ needs finite generation. In particular the tables of
admissible levels computed case by case for non-simply-connected groups in
\cite{Krepski} are, structurally, computations of a single integer $k_0$.
Second, it reduces that computation to $a+b$ integrals, where $a$ and $b$ are
the numbers of generators of $H_{n+1}(N;\Z)$ and of $K$; the second group is
often much smaller than $H_n(M;\Z)$, and is zero whenever $F_*$ is injective, by
Theorem~\ref{thm:period-criterion}(iii). Third, it isolates the qualitative
dichotomy: either some period is irrational, and no level whatsoever works, or
all are rational and the admissible levels form a full cyclic group. There is no
intermediate behaviour.
\end{remark}

\begin{corollary}[Monotonicity under morphisms]\label{cor:level-monotone}
Let $\Phi\colon F\to F'$ be a morphism of arrows and $\Psi$ a
$d_{F'}$-closed relative form. Then
\[
  L(\Psi)\;\subseteq\;L(\Phi^*\Psi),
  \qquad\text{equivalently}\qquad
  k_0(\Phi^*\Psi)\ \big|\ k_0(\Psi).
\]
Pullback can only lower the prequantization level, never raise it.
\end{corollary}

\begin{proof}
By Theorem~\ref{thm:functoriality}(iv), $P_{\Phi^*\Psi}\subseteq P_\Psi$, so
$kP_\Psi\subseteq\Z$ implies $kP_{\Phi^*\Psi}\subseteq\Z$. Hence
$L(\Psi)\subseteq L(\Phi^*\Psi)$, and $k_0\Z\subseteq k_0'\Z$ means
$k_0'\mid k_0$.
\end{proof}

\begin{example}[The disc, revisited]\label{ex:disc-level}
For the disc of Example~\ref{ex:disc}, $H_2(D^2;\Z)=0$ gives $a=0$, and
$K=H_1(S^1;\Z)=\Z$ is generated by the boundary circle, so $b=1$ and
\[
  P_\rvp=\Bigl\langle \int_{D^2}\omega-\int_{S^1}\eta \Bigr\rangle .
\]
Writing $c$ for that single number, the level group is $L(\rvp)=k_0\Z$ with
$k_0$ the denominator of $c$ when $c\in\Q$, and $k_0=0$ otherwise. So a
bulk--boundary system on the disc is prequantizable at some level if and only if
its total flux-minus-holonomy is rational, and then exactly at the multiples of
its denominator. We stress that neither the flux $\int_{D^2}\omega$ nor the
holonomy $\int_{S^1}\eta$ is separately quantized, and neither is separately a
relative period: only the difference is constrained, which is exactly the
content of the relative formulation.
\end{example}

\begin{remark}[Comparison with the absolute theory]\label{rem:level-absolute}
In the absolute case $M=\varnothing$ one has $K=0$, so $b=0$ and $k_0$ is the
lcm of the denominators of the ordinary periods of $\omega$ --- the classical
statement that a closed form is prequantizable at level $k$ exactly when
$k[\omega]$ is integral. The relative theory adds the $b$ generators coming from
$K$, and by Corollary~\ref{cor:edge-periods} these are precisely the generators
that vanish at both edges. The prequantization level is therefore a genuinely
relative invariant: it is computed by the absolute periods of the target
together with a correction supported on the classes that die under $F_*$.
\end{remark}

\subsection{Functoriality: morphisms of arrows}\label{subsec:functoriality}

Every construction in this article is attached to an arrow $F$ rather than to a
space, so the correct notion of a map between two such situations is a
commutative square. Recording its effect once saves repeating the same
verification in Sections~\ref{sec:actions}--\ref{sec:qham}.

\begin{definition}\label{def:arrowmap}
A \emph{morphism of arrows} $\Phi\colon F\to F'$, where $F\colon M\to N$ and
$F'\colon M'\to N'$, is a pair of smooth maps $\phi\colon M\to M'$,
$\psi\colon N\to N'$ with $F'\circ\phi=\psi\circ F$.
\end{definition}

\begin{theorem}[Functoriality of the relative package]\label{thm:functoriality}
Let $\Phi=(\phi,\psi)\colon F\to F'$ be a morphism of arrows. Then:
\begin{enumerate}[label=\textup{(\roman*)},leftmargin=2.2em]
\item $\Phi^*(\alpha,\beta):=(\psi^*\alpha,\ \phi^*\beta)$ defines a morphism of
cochain complexes $\Phi^*\colon\Om^\bullet(F')\to\Om^\bullet(F)$, so
$\Phi^*d_{F'}=\dF\Phi^*$ and $\Phi^*$ descends to
$H^\bullet(\Om(F'))\to H^\bullet(\Om(F))$;
\item $\Phi_*(S,T):=(\psi_*S,\ \phi_*T)$ defines a morphism of chain complexes
$C_\bullet(F)\to C_\bullet(F')$, so $\Phi_*\partial_F=\partial_{F'}\Phi_*$;
\item the pairing \eqref{eq:pairing} is adjoint:
\begin{equation}\label{eq:adjointness}
  \pair{\Phi^*\Psi}{c}=\pair{\Psi}{\Phi_*c}
  \qquad\text{for all }\Psi\in\Om^k(F'),\ c\in C_k(F);
\end{equation}
\item consequently $P_{\Phi^*\Psi}\subseteq P_\Psi$; in particular, if $\Psi$ is
relatively integral then so is $\Phi^*\Psi$;
\item $\Phi^*$ carries relative prequantizations to relative prequantizations:
if $(L,\nabla,s)$ prequantizes $(F',\Psi)$ then
$(\psi^*L,\psi^*\nabla,\phi^*s)$ prequantizes $(F,\Phi^*\Psi)$;
\item if $\Phi$ is equivariant for an action of $G$ on both arrows, and
$(f_k)$ is a relative homotopy moment map for $(F',\Psi)$, then
$(\Phi^*\!\circ f_k)$ is a relative homotopy moment map for $(F,\Phi^*\Psi)$,
provided the fundamental vector fields correspond, that is
$\phi_*v^M_x=v^{M'}_x$ and $\psi_*v^N_x=v^{N'}_x$.
\end{enumerate}
\end{theorem}

\begin{proof}
(i) Using $F'\circ\phi=\psi\circ F$, hence $\phi^*F'^*=F^*\psi^*$,
\[
  \Phi^*d_{F'}(\alpha,\beta)
  =\bigl(\psi^*d\alpha,\ \phi^*(F'^*\alpha-d\beta)\bigr)
  =\bigl(d\psi^*\alpha,\ F^*\psi^*\alpha-d\phi^*\beta\bigr)
  =\dF\Phi^*(\alpha,\beta).
\]
(ii) Similarly $\psi_*F_*=F'_*\phi_*$, so
$\Phi_*\partial_F(S,T)=(\psi_*\partial S-\psi_*F_*T,\,-\phi_*\partial T)
=(\partial\psi_*S-F'_*\phi_*T,\,-\partial\phi_*T)=\partial_{F'}\Phi_*(S,T)$.

(iii) Both sides equal $\int_{\psi_*S}\alpha-\int_{\phi_*T}\beta$ by the change of
variables formula for integration of forms over smooth chains.

(iv) By (ii) the image of a relative cycle is a relative cycle, so
\eqref{eq:adjointness} exhibits every period of $\Phi^*\Psi$ as a period of $\Psi$.

(v) Curvature and pullback commute, $\curv(\psi^*\nabla)=\psi^*\curv(\nabla)
=-2\pi i\,\psi^*\omega$, and $F^*\psi^*L=\phi^*F'^*L$, under which
$\phi^*s$ satisfies $(F^*\psi^*\nabla)\phi^*s=-2\pi i\,\phi^*\beta\otimes\phi^*s$,
which is \eqref{eq:sectioneq} for $\Phi^*\Psi$.

(vi) The component equations \eqref{eq:componenteqs} are built from $\dF$ and from
multicontractions of $\rvp$ with fundamental vector fields; $\Phi^*$ commutes
with the former by (i) and with the latter by the assumed correspondence of
fundamental vector fields, since $\io_{v}\Phi^*=\Phi^*\io_{v'}$ whenever $v$ and
$v'$ are $\Phi$-related componentwise.
\end{proof}

\begin{remark}[Where functoriality is used]\label{rem:functoriality-uses}
Theorem~\ref{thm:functoriality} is the formal reason for several statements that
would otherwise require separate proofs: invariance of the topological term under
homotopies of relative fields (Theorem~\ref{thm:topterm}) is
\eqref{eq:adjointness} applied to the cylinder; the behaviour of the
Wess--Zumino phase under change of bulk filling
(Theorem~\ref{thm:WZ}) is (iv); and the compatibility of the canonical
quasi-Hamiltonian moment map with fusion and with equivariant maps of
quasi-Hamiltonian spaces (Section~\ref{sec:qham}) is (vi).
\end{remark}

\begin{corollary}[Pullback of level quantization]\label{cor:pullback-level}
Let $\Phi\colon F\to F'$ be a morphism of arrows and $\Psi$ a $d_{F'}$-closed
relative form. If $(F',\Psi)$ satisfies the level-quantization condition, i.e.\
$k\Psi$ is relatively integral, then so does $(F,\Phi^*\Psi)$, at the same level.
The converse fails in general: by Theorem~\ref{thm:period-criterion}(iii) any
homologically injective $F$ imposes no relative condition at all, so pulling back
along such an $F$ can destroy all information.
\end{corollary}

\begin{proof}
Immediate from Theorem~\ref{thm:functoriality}(iv) applied to $k\Psi$, and from
Example~\ref{ex:torus} for the failure of the converse.
\end{proof}

\section{Application A: Topological terms of action functionals}
\label{sec:actions}

\begin{remark}[Bibliographical notes]\label{rem:biblio-A}
Topological terms of the kind considered here originate with Novikov
\cite{Novikov} and Witten \cite{Witten}; the cohomological quantization of such
terms, including the boundary case, is due to Alvarez \cite{Alvarez} and
Gaw\c{e}dzki \cite{Gawedzki}, and the gerbe-theoretic treatment of the
boundary sectors to Gaw\c{e}dzki--Reis \cite{GawedzkiReis} and
Carey--Johnson--Murray \cite{CJM}. Freed's account of classical Chern--Simons
theory \cite{Freed} is the model for the three-level structure of
Theorem~\ref{thm:WZ}. What is new here is not the phenomenon but its
formulation: the bulk and boundary contributions are the two components of a
single relative cocycle, and homotopy invariance is the assertion that it is
$\dF$-closed.
\end{remark}

Throughout this section $\rvp=(\omega,\eta)\in\Om^{n+1}(F)$ is $\dF$-closed.

\subsection{Relative fields and their actions}

\begin{definition}\label{def:relfield}
Let $W$ be a compact oriented $(n{+}1)$-manifold with boundary $\partial W$
(possibly empty). A \emph{relative field} on $W$ is a pair of smooth maps
\[
  u\colon W\to N,\qquad v\colon\partial W\to M,
  \qquad\text{with}\quad F\circ v=u|_{\partial W},
\]
and its \emph{action} is
\begin{equation}\label{eq:action}
  S_\rvp(u,v)\;:=\;\int_W u^*\omega\;-\;\int_{\partial W}v^*\eta .
\end{equation}
\end{definition}

This is exactly the structure of a ``bulk term plus boundary counterterm'': the
bulk Lagrangian density is pulled back from $N$, and the boundary correction from
$M$; the constraint $F\circ v=u|_{\partial W}$ says the boundary datum lies over
the bulk field. The fundamental observation is:

\begin{lemma}\label{lem:fieldcycle}
A relative field $(u,v)$ on $W$ determines the relative $(n{+}1)$-cycle
\[
  c(u,v):=\bigl(u_*[W],\;v_*[\partial W]\bigr)\in Z_{n+1}(F),
\]
and $S_\rvp(u,v)=\pair{\rvp}{c(u,v)}$. Moreover, a smooth homotopy of relative
fields $(u_t,v_t)$, $t\in[0,1]$ -- i.e.\ smooth $U\colon W\times[0,1]\to N$,
$V\colon\partial W\times[0,1]\to M$ with $F\circ V=U|_{\partial W\times[0,1]}$ --
determines a relative chain $h$ with $\partial_F h=c(u_1,v_1)-c(u_0,v_0)$.
\end{lemma}

\begin{proof}
For the first claim, $\partial_F c(u,v)=\bigl(u_*\partial[W]-F_*v_*[\partial W],
\ -v_*\partial[\partial W]\bigr)=\bigl(u_*[\partial W]-(F\circ v)_*[\partial W],
\,0\bigr)=0$ since $F\circ v=u|_{\partial W}$ and $\partial[\partial W]=0$; the
identification of $S_\rvp$ with the pairing is the definition
\eqref{eq:pairing}. For the second claim set
\[
  h:=\bigl(U_*[W\times I],\;-V_*[\partial W\times I]\bigr),\qquad I=[0,1],
\]
where the fundamental chains are taken with the product orientation and the prism
convention $\partial[P\times I]=[P\times\{1\}]-[P\times\{0\}]-[\partial P\times I]$.
Then
\begin{align*}
\partial_F h
&=\Bigl(U_*\partial[W\times I]+F_*V_*[\partial W\times I],\;
   \partial V_*[\partial W\times I]\Bigr)\\
&=\Bigl(u_{1*}[W]-u_{0*}[W]-U_*[\partial W\times I]+U_*[\partial W\times I],\;
   v_{1*}[\partial W]-v_{0*}[\partial W]\Bigr)\\
&=c(u_1,v_1)-c(u_0,v_0),
\end{align*}
where we used $F_*V_*=U_*$ on $\partial W\times I$ (from
$F\circ V=U|_{\partial W\times I}$) and $\partial[\partial W\times I]
=[\partial W\times\{1\}]-[\partial W\times\{0\}]$ ($\partial W$ being closed).
\end{proof}

\begin{theorem}[Closed relative forms are topological terms]\label{thm:topterm}
Let $\dF\rvp=0$. Then:
\begin{enumerate}[label=\textup{(\roman*)},leftmargin=2.2em]
\item $S_\rvp(u,v)$ is invariant under smooth homotopies of relative fields;
\item $S_\rvp$ depends only on the cohomology class $[\rvp]\in
H^{n+1}(\Om(F))$; in particular, $S_\rvp\equiv0$ whenever $\rvp$ is
$\dF$-exact;
\item \textup{(Converse)} conversely, if $\rvp\in\Om^{n+1}(F)$ is arbitrary
and $S_\rvp(u,v)$ is invariant under smooth homotopies of relative fields,
then $\dF\rvp=0$. Only two test manifolds are needed: $W=S^{n+1}$
\textup{(}closed, detecting the target component\textup{)} and $W=D^{n+1}$
\textup{(}with boundary, detecting the source component\textup{)}.
\end{enumerate}
Thus closed relative forms are \emph{precisely} the homotopy-invariant
topological terms of this type.
\end{theorem}

\begin{proof}
(i) By Lemma~\ref{lem:fieldcycle} and the relative Stokes theorem
(Proposition~\ref{prop:stokes}),
\[
  S_\rvp(u_1,v_1)-S_\rvp(u_0,v_0)
  =\pair{\rvp}{\partial_F h}
  =\pair{\dF\rvp}{h}=0 .
\]
(ii) If $\rvp'=\rvp+\dF\theta$, then
$S_{\rvp'}(u,v)-S_\rvp(u,v)=\pair{\dF\theta}{c(u,v)}
=\pair{\theta}{\partial_F c(u,v)}=0$ by Lemma~\ref{lem:fieldcycle}.

(iii) Write $\dF\rvp=(d\omega,\,F^*\omega-d\eta)$; we show separately that
each component vanishes, by producing localized test fields whose homotopy
variations detect it. We first record the elementary construction that makes
``sweeping a simplex'' precise.

\smallskip
\emph{Sweeping lemma.} Let $P$ be a compact oriented $p$-manifold without
boundary, let $U\subseteq\R^{q}$ be convex, let $\phi_0\colon P\to U$ be smooth,
and let $x\in U$. Define the straight-line homotopy
$\phi\colon[0,1]\times P\to U$, $\phi(t,z)=(1-t)\phi_0(z)+t\,x$. Then the
swept chain of $\phi$ --- that is, the image under $\phi$ of the fundamental
chain of $[0,1]\times P$ --- is the cone $x*(\phi_0)_*[P]$, a smooth singular
$(p{+}1)$-chain in $U$ with
\[
  \partial\bigl(x*(\phi_0)_*[P]\bigr)=(\phi_0)_*[P],
\]
since $P$ is closed. Taking $P=\partial\Delta$ for an affine $(p{+}1)$-simplex
$\Delta\subseteq U$, $\phi_0$ a degree-one smooth map $P\to\partial\Delta$, and
$x$ the barycentre of $\Delta$, the swept chain equals $\Delta$ up to
degenerate simplices, which integrate to zero. All maps here are smooth and
all chains are smooth singular chains, so the pairing \eqref{eq:pairing} and
Proposition~\ref{prop:stokes} apply verbatim.

\emph{Target component.}  Take $W=S^{n+1}$, which is closed and oriented; a
relative field on it is simply a smooth map $u\colon S^{n+1}\to N$, the boundary
datum being vacuous, and $S_\rvp(u)=\int_{S^{n+1}}u^*\omega$. Fix $y\in N$, a
chart $\varphi\colon U\xrightarrow{\ \sim\ }U'\subseteq\R^{\dim N}$ around $y$
with $U'$ convex, and $n{+}2$ linearly independent tangent vectors at $y$; for
$\varepsilon>0$ small let $\Delta_\varepsilon\subseteq U'$ be the affine
$(n{+}2)$-simplex they span, scaled by $\varepsilon$. Apply the sweeping lemma
with $P=S^{n+1}$, $\phi_0$ a smooth degree-one map onto
$\partial\Delta_\varepsilon$, and $x$ the barycentre; transporting through
$\varphi^{-1}$ gives a smooth homotopy $(u_t)_{t\in[0,1]}$ of maps
$S^{n+1}\to N$, supported in $U$, whose swept chain is $\Delta_\varepsilon$
modulo degenerate simplices. Homotopy invariance and the variation formula of
(i), applied with $v$ vacuous, give
\[
  0=\pair{\dF\rvp}{(\Delta_\varepsilon,0)}
   =\int_{\Delta_\varepsilon}d\omega .
\]
Dividing by $\operatorname{vol}(\Delta_\varepsilon)$ and letting
$\varepsilon\to0$, continuity of $d\omega$ gives
$d\omega_y(w_1,\dots,w_{n+2})=0$ for the chosen vectors; as $y$, the chart and
the frame are arbitrary, $d\omega=0$.

\emph{Source component.}  Now that $d\omega=0$, the variation formula of
(i) reduces, for any homotopy $(u_t,v_t)$ with swept chains
$(h_N,h_M)$, to
\[
  S_\rvp(u_1,v_1)-S_\rvp(u_0,v_0)
  =\pair{\dF\rvp}{(h_N,h_M)}
  =-\int_{h_M}\bigl(F^*\omega-d\eta\bigr).
\]
Fix $m\in M$ and a chart $\psi\colon V\xrightarrow{\ \sim\ }V'\subseteq
\R^{\dim M}$ around $m$ with $V'$ convex, and let
$\Delta'_\varepsilon\subseteq V'$ be a small affine $(n{+}1)$-simplex as before.
Take $W=D^{n+1}$, so $\partial W=S^{n}$ and $W$ is compact oriented with
boundary. Apply the sweeping lemma with $P=S^{n}$ to obtain a smooth homotopy
$(v_t)$ of maps $S^{n}\to M$, supported in $V$, with swept chain
$\Delta'_\varepsilon$ modulo degenerate simplices. Because $V'$ is convex, each
$v_t$ extends to $\widetilde v_t\colon D^{n+1}\to M$ with image in $V$: in the
chart, extend by the straight-line cone on the barycentre of $V'$, which is
smooth on the closed disc after reparametrizing radially by a smooth function
vanishing to infinite order at $0$; the family $(\widetilde v_t)$ is smooth in
$t$ because $(v_t)$ is. Set $u_t:=F\circ\widetilde v_t$, so
$u_t|_{\partial W}=F\circ v_t$ and $(u_t,v_t)$ is a relative field for every
$t$. Since $d\omega=0$, the variation formula reduces to
$0=-\int_{\Delta'_\varepsilon}(F^*\omega-d\eta)$; dividing by the volume and
letting $\varepsilon\to0$ gives $(F^*\omega-d\eta)_m=0$ as above. As $m$ and
the chart are arbitrary, $F^*\omega-d\eta=0$, and hence $\dF\rvp=0$.
\end{proof}

\begin{remark}\label{rem:whyexactvanishes}
Part (ii) deserves emphasis: an exact \emph{relative} class contributes exactly
zero -- not merely a boundary term -- to the action. Indeed, writing
$\theta=(\vartheta,\tau)$,
\[
  S_{\dF\theta}(u,v)
  =\int_W u^*d\vartheta-\int_{\partial W}v^*\bigl(F^*\vartheta-d\tau\bigr)
  =\int_{\partial W}\Bigl((u|_{\partial W})^*\vartheta-(F\circ v)^*\vartheta\Bigr)
  +\int_{\partial(\partial W)}v^*\tau
  =0 .
\]
The bulk boundary term is cancelled \emph{on the nose} by the variation of the
counterterm: this cancellation, which in Lagrangian field theory is engineered case
by case, is automatic once bulk term and counterterm are assembled into a single
relative form.

A word of caution for readers accustomed to the absolute theory, where an
\emph{exact} bulk form famously does \emph{not} give a trivial functional:
if $\omega=d\vartheta$ as an absolute equality, the bulk term collapses by
Stokes to the boundary functional $\int_{\partial W}(u|_{\partial
W})^*\vartheta$, which is generally nonzero --- this is the usual statement
that a Wess--Zumino term with exact curvature is a boundary term, not zero.
Relative exactness is strictly stronger: $(\omega,\eta)=\dF(\vartheta,\tau)$
requires \emph{in addition} that the boundary counterterm $\eta$ equal
$F^*\vartheta-d\tau$, i.e.\ that the counterterm be built from the same
primitive; it is exactly this matching that makes the boundary functional
above cancel against the variation of the counterterm in the display. The
distinction is the entire point of the relative complex: absolute exactness
of $\omega$ says nothing about $\eta$, while relative exactness ties the two
components together, and only the tied pair contributes zero.
\end{remark}

\subsection{Wess--Zumino terms}

We now treat sources without prescribed bulk: this is the classical Wess--Zumino
mechanism.

\begin{definition}\label{def:WZ}
Let $\Sigma$ be a closed oriented $n$-manifold and $v\colon\Sigma\to M$ a smooth
map such that $(F\circ v)_*[\Sigma]=\partial X$ for some
$X\in C_{n+1}(N;\Z)$. The \emph{Wess--Zumino term} of $v$ relative to the
\emph{bulk extension} $X$ is
\begin{equation}\label{eq:WZ}
  \mathrm{WZ}_\rvp(v;X)\;:=\;\int_X\omega-\int_\Sigma v^*\eta
  \;=\;\pair{\rvp}{\bigl(X,\,v_*[\Sigma]\bigr)} .
\end{equation}
\end{definition}

Note that $\bigl(X,v_*[\Sigma]\bigr)$ is a relative $(n{+}1)$-cycle:
$\partial X=F_*v_*[\Sigma]$ and $\partial v_*[\Sigma]=0$.

It is important to distinguish three levels of ambiguity, tested by three
nested subgroups of relative homology.  Write
$P^N_\rvp:=\bigl\{\pair{\rvp}{(Z,0)}:Z\in Z_{n+1}(N)\bigr\}
=\bigl\{\int_Z\omega\bigr\}$ for the group of \emph{absolute target periods},
realized inside the relative periods $P_\rvp$ through the cycles
$(Z,0)$, i.e.\ through the image of $H_{n+1}(N)\to H_{n+1}(F)$; the
quotient of full relative integrality by this condition is governed, via
the long exact sequence of the cone, by the classes mapping to
$\ker\bigl(H_n(M)\to H_n(N)\bigr)$.

\begin{theorem}[Well-definedness of the Wess--Zumino term]\label{thm:WZ}
Let $\dF\rvp=0$.
\begin{enumerate}[label=\textup{(\roman*)},leftmargin=2.2em]
\item \textup{(Fixed boundary field.)} For fixed $v$, the value
$\mathrm{WZ}_\rvp(v;X)$ modulo $P^N_\rvp$ is independent of the bulk
extension $X$; and this is sharp: the set of differences realized by
changing $X$ with $v$ fixed is exactly $P^N_\rvp$.  Consequently
$\exp\bigl(2\pi i\,\mathrm{WZ}_\rvp(v)\bigr)$ is well defined,
independently of $X$, \emph{if and only if} $P^N_\rvp\subseteq\Z$, i.e.\
iff the absolute target periods of $\omega$ are integral.
\item \textup{(Homotopies of the boundary field.)} Under the condition of
\textup{(i)}, the phase $\exp\bigl(2\pi i\,\mathrm{WZ}_\rvp(v)\bigr)$ is
invariant under smooth homotopies of $v$ through maps admitting bulk
extensions; no further integrality is needed for this.
\item \textup{(Relative integrality.)} Since $\dF\rvp=0$, the
pairing $z\mapsto\pair{\rvp}{z}$ already descends from relative cycles to
a homomorphism $H_{n+1}(F;\Z)\to\R$.  The class $\rvp$ is relatively
integral precisely when every relative period is an integer, or equivalently
when the induced character
\[
  H_{n+1}(F;\Z)\longrightarrow U(1),\qquad
  [z]\longmapsto \exp\bigl(2\pi i\pair{\rvp}{z}\bigr),
\]
is trivial.  This condition is stronger than the integrality of the absolute
target periods in \textup{(i)} and is exactly the condition required for the
geometric realizations of Section~\ref{sec:prequantization}: relative line
bundles with connection in degree two and relative gerbes in degree three.
\item \textup{(Exact case.)} If $\rvp=\dF\theta$ is exact, then
$\mathrm{WZ}_\rvp(v;X)=0$ for every $v$ and $X$; in particular
$P_\rvp=\{0\}$, and the Wess--Zumino term carries no information: the bulk
contribution is cancelled on the nose by the counterterm, as in
Remark~\textup{\ref{rem:whyexactvanishes}}.
\end{enumerate}
\end{theorem}

\begin{proof}
(i) If $X'$ is another bulk extension, then $\partial(X-X')=0$, so
$(X-X',0)\in Z_{n+1}(F)$ and
$\mathrm{WZ}_\rvp(v;X)-\mathrm{WZ}_\rvp(v;X')
=\pair{\rvp}{(X-X',0)}=\int_{X-X'}\omega\in P^N_\rvp$.  Sharpness:
given any $Z\in Z_{n+1}(N)$ and any extension $X$ of $v$, the chain
$X':=X+Z$ is another extension of the same $v$, and the difference of the
two values is $\int_Z\omega$; so every element of $P^N_\rvp$ is realized.
The equivalence for the exponential is immediate: the phase is
$X$-independent iff all realized differences are integers iff
$P^N_\rvp\subseteq\Z$.
(ii) By (i) the exponential is independent of $X$. For
homotopy invariance, let $v_t$ be a homotopy with bulk extensions $X_0$ of $v_0$;
the chain $V_*[\Sigma\times I]\in C_{n+1}(M)$ satisfies
$\partial V_*[\Sigma\times I]=v_{1*}[\Sigma]-v_{0*}[\Sigma]$, so
$X_1:=X_0+F_*V_*[\Sigma\times I]$ is a bulk extension of $v_1$, and
\[
  \mathrm{WZ}_\rvp(v_1;X_1)-\mathrm{WZ}_\rvp(v_0;X_0)
  =\int_{F_*V_*[\Sigma\times I]}\omega
  -\int_{\Sigma}v_1^*\eta+\int_{\Sigma}v_0^*\eta
  =\int_{\Sigma\times I}V^*F^*\omega-\int_{\partial(\Sigma\times I)}V^*\eta,
\]
which vanishes by $F^*\omega=d\eta$ and Stokes on $\Sigma\times I$; note
that this computation uses only the closedness of $\rvp$, not any
integrality beyond (i).
(iii) Closedness and the relative Stokes formula imply that the real-valued
pairing is constant on relative homology classes.  By definition,
$\rvp$ is relatively integral exactly when
$\pair{\rvp}{z}\in\Z$ for every integral relative cycle $z$; this is
equivalent to triviality of the displayed $U(1)$-character.  The geometric
interpretation is supplied by Theorem~\ref{thm:prequantization} in degree two
and by the relative-gerbe discussion in Section~\ref{subsec:gerbes} in degree
three.  The condition can be strictly stronger than (i), because relative
cycles with nonzero $M$-component need not arise as differences of two bulk
fillings of the same boundary field.
(iv) For exact $\rvp$ every relative period vanishes
(Proposition~\ref{prop:stokes}), and
$\mathrm{WZ}_{\dF\theta}(v;X)=\pair{\theta}{\partial_F(X,v_*[\Sigma])}=0$.
\end{proof}

\begin{example}[Classical Wess--Zumino--Witten term]\label{ex:WZW}
The absolute mechanism is the degenerate case $M=\emptyset$: there is no boundary
datum, and for a closed $\omega\in\Om^{n+1}(N)$ and a map $\varphi\colon\Sigma\to N$
with $\varphi_*[\Sigma]=\partial X$ one sets
$\mathrm{WZ}_\omega(\varphi;X):=\int_X\omega=\pair{(\omega,0)}{(X,0)}$; the
proofs of Theorem~\ref{thm:WZ}(i)--(ii) apply verbatim with $T=0$, and the
ambiguity group is the group of absolute periods of $\omega$, realized inside
$P_\rvp$ by the relative cycles $(X-X',0)$. For $N=G$ compact simple and
$\omega=k\eta$ with $\eta$ normalized so that $[\eta]$ generates the image of
$H^3(G;\Z)$ in $H^3(G;\R)$, this is the familiar level-$k$ Wess--Zumino term of
the WZW model, well defined in $U(1)$ precisely for $k\in\Z$. The genuinely
relative Definition~\ref{def:WZ} is the refinement in which the source field is
allowed to take values in an auxiliary space $M$ over $N$, with the trivializing
form $\eta_M$ absorbing the boundary of the bulk term -- the situation of the next
example.
\end{example}

\begin{example}[Quasi-Hamiltonian Wess--Zumino term and level quantization]
\label{ex:qhamWZ}
Let $(M,\omega,\mu)$ be a quasi-Hamiltonian $G$-space, $\rvp=(\eta,\omega)\in
\Om^3(\mu)$ its relative $2$-plectic structure, and $k\in\Z_{>0}$. For a closed
oriented surface $\Sigma$ and $v\colon\Sigma\to M$ with bulk extension
$X\in C_3(G;\Z)$ of $\mu\circ v$,
\[
  \mathrm{WZ}_{k\rvp}(v;X)=k\Bigl(\int_X\eta-\int_\Sigma v^*\omega\Bigr).
\]
By Theorem~\ref{thm:WZ}(iii), the class $k\rvp=(k\eta,k\omega)$ is
relatively integral if and only if all of its periods on relative
$3$-cycles $(X,\Sigma)$ of $\mu$ are integers. This is precisely the \emph{level-$k$
prequantization condition} for quasi-Hamiltonian $G$-spaces
\cite[Def.~6.2, Prop.~6.4]{MeinrenkenLectures}, introduced via relative (co)homology
by Krepski \cite{Krepski} and via relative gerbes by Shahbazi \cite{Shahbazi}: the
relative multisymplectic framework derives it as an instance of the general
topological-term mechanism. We return to the geometric (gerbe) interpretation in
Section~\ref{subsec:gerbes}.
\end{example}

\section{Application B: Relative prequantization}
\label{sec:prequantization}

\begin{remark}[Bibliographical notes]\label{rem:biblio-B}
Prequantization in the symplectic setting is due to Souriau \cite{Souriau} and
Kostant \cite{Kostant}. The higher-degree theory is governed by bundle gerbes
\cite{Murray,MurrayStevenson,Waldorf,Brylinski} and, in the form closest to the
one used here, by differential characters and differential cohomology
\cite{CheegerSimons,HopkinsSinger}. The relative case in degree three is
Shahbazi's theory of relative gerbes \cite{ShahbaziGerbes}, applied to
quasi-Hamiltonian spaces in \cite{Shahbazi} and to moduli of flat bundles in
\cite{Krepski,KrepskiThesis}. Theorem~\ref{thm:prequantization} is the degree-two
counterpart, proved directly by a \v{C}ech mapping-cone argument.
\end{remark}

\subsection{Relative prequantization of relative $2$-forms}

Throughout this subsection $\rvp=(\omega,\eta)\in\Om^2(F)$ is a $\dF$-closed
relative $2$-form: $d\omega=0$ and $F^*\omega=d\eta$. Recall that for a Hermitian
line bundle with connection $(L,\nabla)$ we normalize the curvature by
$\curv(\nabla)\in\Om^2(\cdot;i\R)$, locally $\curv=-2\pi i\,da$ for
$\nabla=d-2\pi i\,a$, and the holonomy of a loop $\gamma$ bounding a $2$-chain $S$
is
\begin{equation}\label{eq:holonomy}
  \hol_\nabla(\gamma)=\exp\Bigl(-2\pi i\int_S a_{\mathrm{curv}}\Bigr),
  \qquad \curv(\nabla)=-2\pi i\,a_{\mathrm{curv}},
\end{equation}
cf.\ \cite{Kostant,Brylinski}.

\begin{definition}\label{def:relprequant}
A \emph{relative prequantization} of $(F,\rvp)$ is a triple $(L,\nabla,s)$
where:
\begin{enumerate}[label=\textup{(\roman*)},leftmargin=2.2em]
\item $(L,\nabla)$ is a Hermitian line bundle with unitary connection on $N$ with
$\curv(\nabla)=-2\pi i\,\omega$;
\item $s$ is a unit-norm section of $F^*L\to M$ satisfying
\begin{equation}\label{eq:sectioneq}
  (F^*\nabla)\,s\;=\;-2\pi i\,\eta\otimes s ,
\end{equation}
i.e.\ $\eta$ is the connection $1$-form of $F^*\nabla$ in the trivialization $s$.
\end{enumerate}
\end{definition}

Thus a relative prequantization is a prequantum line bundle for $\omega$ on the
target, together with a trivialization over the source whose logarithmic
derivative is prescribed by $\eta$; equation \eqref{eq:sectioneq} is consistent
precisely because $\curv(F^*\nabla)=-2\pi i\,F^*\omega=-2\pi i\,d\eta$.

\begin{theorem}[Relative Weil--Kostant theorem]\label{thm:prequantization}
$(F,\rvp)$ admits a relative prequantization if and only if $\rvp$ is
relatively integral, i.e.\ $\pair{\rvp}{c}\in\Z$ for every relative $2$-cycle
$c\in Z_2(F)$ \textup{(}equivalently, for $M,N$ of finite type, iff
$[\rvp]$ lies in the image of $H^2(F;\Z)\to H^2(\Om(F))$,
Remark~\textup{\ref{rem:integralclass})}. Moreover, when it exists, the set of
isomorphism classes of relative prequantizations is a torsor over the group
$H^1(F;U(1))$ of flat relative data.
\end{theorem}

\begin{proof}
\emph{Necessity.} Let $(L,\nabla,s)$ be a relative prequantization and
$(S,T)\in Z_2(F)$ a relative $2$-cycle, which we may take smooth. Since
$\partial T=0$, $T$ is an integral combination of smooth loops
$\gamma_1,\dots,\gamma_r$ in $M$, and $\partial S=F_*T$. Compute the parallel
transport of $\nabla$ along the $1$-cycle $F_*T$ in two ways. First, since
$\partial S=F_*T$, formula \eqref{eq:holonomy} (extended multiplicatively to
cycles bounding chains) gives
\[
  \hol_\nabla(F_*T)=\exp\Bigl(-2\pi i\int_S\omega\Bigr).
\]
Second, along each loop $F\circ\gamma_j$ the parallel transport of $\nabla$
coincides with that of $F^*\nabla$ along $\gamma_j$; in the unit trivialization
$s$, equation \eqref{eq:sectioneq} says the connection $1$-form is $\eta$, so
\[
  \hol_\nabla(F_*T)=\hol_{F^*\nabla}(T)
  =\exp\Bigl(-2\pi i\int_T\eta\Bigr).
\]
Equating,
$\exp\bigl(-2\pi i(\int_S\omega-\int_T\eta)\bigr)=1$, i.e.\
$\pair{\rvp}{(S,T)}\in\Z$. Arbitrary (continuous) relative cycles are handled by
smooth approximation.

\emph{Sufficiency.}  We first assemble the \v{C}ech machinery; the reader
willing to grant Lemmas~\ref{lem:cech}--\ref{lem:correction} may skip to the
construction.

\emph{The relative \v{C}ech complex.}  Both $M$ and $N$ are smooth
paracompact manifolds, so good covers exist; no further topological
hypothesis is needed for the theorem (finite type enters only in the
reformulation of Remark~\ref{rem:integralclass}).  Fix a good cover
$\{U_i\}_{i\in I}$ of $N$, a good cover $\{V_a\}_{a\in A}$ of $M$, and a
refinement map $r\colon A\to I$ with $F(V_a)\subseteq U_{r(a)}$ (such data
exist: refine any good cover of $M$ against $F^{-1}\{U_i\}$).  For an
abelian group (or sheaf) $\Lambda$, define the \emph{\v{C}ech mapping cone}
\[
  \check C^{k}(F;\Lambda)
  :=\check C^{k}\bigl(\{U_i\};\Lambda\bigr)\oplus
    \check C^{k-1}\bigl(\{V_a\};\Lambda\bigr),
  \qquad
  D(c,m):=\bigl(\delta c,\ r^{\#}c-\delta m\bigr),
\]
where $\delta$ is the \v{C}ech differential and
$(r^{\#}c)_{a_0\dots a_{k}}:=c_{r(a_0)\dots r(a_{k})}\circ F$ is the
pullback along $(F,r)$; the sign conventions mirror those of the de~Rham
cone \eqref{eq:cone}, and $D^2=0$ because $r^{\#}$ is a chain map.

\begin{lemma}\label{lem:cech}
For $\Lambda=\R$ or $\Z$ there are natural isomorphisms
\[
  \check H^{k}(F;\Lambda)\;\cong\;H^{k}(F;\Lambda),
\]
with $H^{k}(F;\Lambda)$ the singular cohomology of the cone of $F^*$ on
singular cochains \textup{(}for $\Lambda=\R$, also
$\cong H^{k}(\Om(F))$ via the de~Rham map\textup{)}, compatibly with the
change of coefficients $\Z\to\R$, and independent of the choice of good
covers and refinement map up to canonical isomorphism.
\end{lemma}

\begin{proof}
On each row, the \v{C}ech complex of a good cover computes singular (=
sheaf) cohomology of the underlying manifold, naturally in the cover, by
the standard double-complex argument \cite[\S8--\S10]{BottTu}; the maps
$r^{\#}$ and $F^*$ are intertwined by these row quasi-isomorphisms
(both are induced by $(F,r)$ on the \v{C}ech--singular double complex).
A quasi-isomorphism of rows induces a quasi-isomorphism of cones (five
lemma applied to the two long exact cone sequences), which is the claimed
isomorphism; the same argument gives the de~Rham comparison for
$\Lambda=\R$, and naturality in refinements: two refinement maps
$r,r'$ are related by a \v{C}ech chain homotopy (again \cite{BottTu}),
which induces a homotopy of cone maps, so the isomorphism is independent
of $r$; passing to the colimit over pairs of good covers removes the
dependence on the covers.
\end{proof}

\begin{lemma}[Integral correction]\label{lem:correction}
Let $(c,m)\in\check C^{2}(F;\R)$ be a $D$-cocycle with locally constant
components whose class lies in the image of
$\check H^{2}(F;\Z)\to\check H^{2}(F;\R)$.  Then there is a locally
constant cochain $(b,p)\in\check C^{1}(F;\R)$ with
$(c,m)-D(b,p)$ integer-valued.
\end{lemma}

\begin{proof}
Choose an integral $D$-cocycle $(n,q)\in\check C^{2}(F;\Z)$ mapping to
the class of $(c,m)$; then $(c-n,\,m-q)$ is a real $D$-coboundary,
$(c-n,m-q)=D(b,p)$ for some $(b,p)\in\check C^{1}(F;\R)$, which may be
taken locally constant since the cocycle and the cohomology of the good
covers are computed in locally constant cochains.  Then
$(c,m)-D(b,p)=(n,q)$ is integer-valued.
\end{proof}

\emph{Construction.}  Assume relative integrality; by
Remark~\ref{rem:integralclass} the class of $\rvp$ admits an integral
lift.
Since $d\omega=0$, choose $\kappa_i\in\Om^1(U_i)$ with $d\kappa_i=\omega$ and
smooth $f_{ij}\in C^\infty(U_{ij})$ with $\kappa_j-\kappa_i=df_{ij}$ and
$f_{ij}+f_{ji}=0$; then $c_{ijk}:=f_{ij}+f_{jk}-f_{ik}$ is locally constant. On
$M$, since $d\bigl(F^*\kappa_{r(a)}-\eta\bigr)=F^*\omega-F^*\omega=0$ on the
contractible $V_a$, choose $h_a\in C^\infty(V_a)$ with
$dh_a=F^*\kappa_{r(a)}-\eta$; then
$m_{ab}:=h_b-h_a-F^*f_{r(a)r(b)}$ is locally constant on $V_{ab}$. The collection
$(c_{ijk};m_{ab})$ is locally constant and is a $D$-cocycle in
$\check C^{2}(F;\R)$: $\delta c=0$ is the standard computation, and
$(r^{\#}c-\delta m)_{ab c}=0$ follows by expanding the definitions of
$m$ using $c_{r(a)r(b)r(c)}=f_{r(a)r(b)}+f_{r(b)r(c)}-f_{r(a)r(c)}$.  The
double-complex comparison of Lemma~\ref{lem:cech} identifies its class
with $[\rvp]$ (the standard \v{C}ech--de~Rham zig-zag, performed
simultaneously in the two rows and intertwined by $(F,r)$).  By relative
integrality and Lemma~\ref{lem:correction}, there is a locally constant
$(b,p)=\bigl((b_{ij}),(p_a)\bigr)$ such that after replacing
$f_{ij}$ by $f_{ij}-b_{ij}$ and $h_a$ by $h_a-p_a$ --- which changes
neither $d f_{ij}$ nor $dh_a$, hence neither $\kappa$-data nor the
verification above --- one has
\begin{equation}\label{eq:integralcocycle}
  c_{ijk}\in\Z
  \qquad\text{and}\qquad
  m_{ab}\in\Z .
\end{equation} Now define:
\begin{itemize}[leftmargin=2em]
\item $L\to N$ by the transition functions $g_{ij}:=e^{2\pi i f_{ij}}$ on
$U_{ij}$; the cocycle condition $g_{ij}g_{jk}g_{ki}=e^{2\pi i c_{ijk}}=1$ holds by
\eqref{eq:integralcocycle};
\item $\nabla$ by $\nabla:=d-2\pi i\,\kappa_i$ in the unit frame $e_i$ over $U_i$;
compatibility across overlaps holds since
$\kappa_j-\kappa_i=\tfrac{1}{2\pi i}\,d\log g_{ij}$, and
$\curv(\nabla)=-2\pi i\,d\kappa_i=-2\pi i\,\omega$;
\item $s$ over $V_a$ by $s|_{V_a}:=e^{2\pi i h_a}\,F^*e_{r(a)}$. On $V_{ab}$,
using $F^*e_{r(b)}=e^{-2\pi i F^*f_{r(a)r(b)}}F^*e_{r(a)}$,
\[
  e^{2\pi i h_b}F^*e_{r(b)}
  =e^{2\pi i\left(h_b-F^*f_{r(a)r(b)}\right)}F^*e_{r(a)}
  =e^{2\pi i m_{ab}}\,e^{2\pi i h_a}F^*e_{r(a)}
  =s|_{V_a}
\]
by \eqref{eq:integralcocycle}, so $s$ is a global unit section of $F^*L$. Finally,
in the frame $F^*e_{r(a)}$ the connection $F^*\nabla$ has potential
$F^*\kappa_{r(a)}$, so
\[
  (F^*\nabla)s
  =\bigl(2\pi i\,dh_a-2\pi i\,F^*\kappa_{r(a)}\bigr)\otimes s
  =-2\pi i\,\bigl(F^*\kappa_{r(a)}-dh_a\bigr)\otimes s
  =-2\pi i\,\eta\otimes s,
\]
which is \eqref{eq:sectioneq}.
\end{itemize}
\emph{Torsor structure.}  We first make the group precise.  By
$H^{1}(F;U(1))$ we mean the first cohomology of the mapping cone of
$F^*\colon C^\bullet(N;U(1))\to C^\bullet(M;U(1))$ on singular cochains
with (discrete) $U(1)$-coefficients; by Lemma~\ref{lem:cech} (which holds
for any coefficient group) this agrees with the \v{C}ech cone
$\check H^{1}(F;U(1))$ for good covers, and with the singular cohomology
of the topological mapping cone of $F$ in degree one.  Geometrically,
$H^{1}(F;U(1))$ classifies pairs (flat Hermitian line bundle
$(L',\nabla')$ on $N$, $\nabla'$-parallel unit section $s'$ of
$F^*L'$): in the \v{C}ech model, $L'$ has locally constant unitary
transition functions $\lambda_{ij}$, flatness is $\delta$-cocyclehood, and
the parallel trivialization is locally constant unitary data
$\nu_a$ with $r^{\#}\lambda=\delta\nu$ --- exactly a degree-one
$D$-cocycle with $U(1)$-coefficients, with isomorphisms of pairs
corresponding to coboundaries.

The \emph{difference construction}: given two relative prequantizations
$(L_1,\nabla_1,s_1)$ and $(L_2,\nabla_2,s_2)$ of the same $\rvp$, set
\[
  (L',\nabla',s')
  :=\bigl(L_1\otimes L_2^{\vee},\ \nabla_1\otimes\id+\id\otimes\nabla_2^{\vee},\
  s_1\otimes s_2^{\vee}\bigr).
\]
Then $\curv(\nabla')=-2\pi i(\omega-\omega)=0$, so $(L',\nabla')$ is flat;
and $(F^*\nabla')s'=-2\pi i(\eta-\eta)\otimes s'=0$, so $s'$ is parallel.
Conversely, given a relative prequantization $(L,\nabla,s)$ and a flat pair
$(L',\nabla',s')$ --- that is, a flat Hermitian line bundle on $N$ with a
parallel unit section of its pullback, whose isomorphism classes form the group
$H^1(F;U(1))$ defined above --- the twist
$(L\otimes L',\ \nabla\otimes\id+\id\otimes\nabla',\ s\otimes s')$ has curvature
$-2\pi i\,\omega$ and satisfies \eqref{eq:sectioneq}, so it is again a relative
prequantization of $\rvp$. This defines an action of $H^1(F;U(1))$ on the set
$\mathcal{P}(\rvp)$ of isomorphism classes.

The action is \emph{transitive}: given $[(L_1,\nabla_1,s_1)]$ and
$[(L_2,\nabla_2,s_2)]$ in $\mathcal{P}(\rvp)$, the difference construction
produces a flat pair carrying the second to the first, since
$L_2\otimes(L_1\otimes L_2^\vee)\cong L_1$ compatibly with connections and
sections. It is \emph{free}: if twisting $(L,\nabla,s)$ by $(L',\nabla',s')$
yields an isomorphic relative prequantization, then an isomorphism
$L\otimes L'\xrightarrow{\sim}L$ intertwining the connections and matching
$s\otimes s'$ with $s$ trivializes $(L',\nabla')$ as a flat bundle and
identifies $s'$ with the constant section $1$, so $[(L',\nabla',s')]$ is the
identity of $H^1(F;U(1))$. Hence $\mathcal{P}(\rvp)$, when nonempty, is a torsor
over $H^{1}(F;U(1))$.

Here $H^1(F;U(1))$ is understood as the degree-one cohomology of the mapping
cone of $F^*$ with $U(1)$ coefficients, computed in the \v{C}ech model on the
good covers fixed above; concretely, it is the group of pairs consisting of a
flat line bundle on $N$ together with a parallel trivialization of its pullback
to $M$, modulo isomorphism, and the long exact sequence \eqref{eq:LES} in
$U(1)$ coefficients identifies it with the relative cohomology
$H^1(N,M;U(1))$ when $F$ is an embedding.
\end{proof}

\begin{corollary}\label{cor:prequantexamples}
\begin{enumerate}[label=\textup{(\alph*)},leftmargin=2.2em]
\item \textup{(Absolute case, $M=\emptyset$.)} Theorem~\ref{thm:prequantization}
reduces to the classical Weil--Kostant integrality criterion for prequantum line
bundles \cite{Kostant}.
\item \textup{(Trivialized case, $\omega=d\lambda$ globally with
$F^*\lambda-\eta$ closed integral.)} Relative prequantizations exist with
$L$ trivial; the section $s$ has winding data classified by
$[F^*\lambda-\eta]\in H^1(M;\R)/H^1(M;\Z)$-type invariants, recovering the
familiar quantization of ``momenta'' along $M$.
\end{enumerate}
\end{corollary}

\begin{proof}
(a) For $M=\emptyset$ the section datum is vacuous and relative cycles are cycles
in $N$. (b) Take $\kappa_i=\lambda|_{U_i}$, $f_{ij}=0$ in the proof above; the
remaining freedom is exactly the choice of $h_a$ with
$dh_a=F^*\lambda-\eta$, patching to a global $s$ iff the periods of
$F^*\lambda-\eta$ are integral, which is relative integrality in this case.
\end{proof}

\subsection{Bohr--Sommerfeld Lagrangians}\label{subsec:bohr-sommerfeld}

The following specialization is short, but it is the sharpest available evidence
that Theorem~\ref{thm:prequantization} is a working statement rather than a
formal one: a foundational notion of geometric quantization is recovered by
substituting $\rvp=(\omega,0)$ into it.

\begin{theorem}[Bohr--Sommerfeld condition]\label{thm:bohr-sommerfeld}
Let $(N,\omega)$ be a symplectic manifold, $L\subseteq N$ a Lagrangian
submanifold, $M=L$, and $F\colon L\hookrightarrow N$ the inclusion. Then
$\rvp=(\omega,0)$ is a $\dF$-closed relative $2$-form, and:
\begin{enumerate}[label=\textup{(\roman*)},leftmargin=2.2em]
\item a relative prequantization of $(F,\rvp)$ is precisely a prequantum line
bundle $(L_{\mathrm{pq}},\nabla)$ on $N$ with $\curv(\nabla)=-2\pi i\,\omega$
together with a \emph{parallel} unit section of its restriction to $L$;
equivalently, a prequantum bundle on $N$ flatly trivialized over $L$;
\item such a datum exists if and only if
\[
  \int_S\omega\ \in\ \Z
  \qquad\text{for every } S\in C_2(N;\Z) \text{ with } \partial S\subseteq L,
\]
that is, if and only if $\omega$ is integral on $H_2(N,L;\Z)$;
\item when nonempty, the set of equivalence classes of such data is a torsor for
\[
  H^1(F;U(1))\ \cong\ H^1(N,L;U(1));
\]
\item concretely, by Theorem~\textup{\ref{thm:period-criterion}}, condition
\textup{(ii)} holds if and only if $\omega$ is integral on $N$ and the defect
homomorphism
\[
  \Theta_\rvp\colon
  \ker\bigl(H_1(L;\Z)\to H_1(N;\Z)\bigr)\longrightarrow \R/P_\omega,
  \qquad
  \Theta_\rvp[\gamma]=\int_S\omega \quad(\partial S=\gamma),
\]
is $\Z$-valued; and by Theorem~\textup{\ref{thm:level-group}} the set of levels
$k$ at which $L$ is Bohr--Sommerfeld is the cyclic group $k_0\Z$.
\end{enumerate}
Thus $L$ is a Bohr--Sommerfeld Lagrangian in the classical sense exactly when
$(F,\rvp)$ is relatively prequantizable.
\end{theorem}

\begin{proof}
Since $L$ is Lagrangian, $F^*\omega=0$, so $\dF(\omega,0)=(d\omega,F^*\omega)=0$.
Unwinding Definition~\ref{def:relprequant} with $\eta=0$, the section equation
\eqref{eq:sectioneq} reads $(F^*\nabla)s=0$, which is (i). Parts (ii) and (iii)
are Theorem~\ref{thm:prequantization} applied to this pair, once one identifies
the relative cycles of the inclusion with $C_2(N,L;\Z)$ and $H^1(F;U(1))$ with
$H^1(N,L;U(1))$, both of which are immediate from
Definition~\ref{def:relchains} and the long exact sequence \eqref{eq:LES}. Part
(iv) is Theorem~\ref{thm:period-criterion}(ii) together with
Theorem~\ref{thm:level-group}, the defect homomorphism specializing as stated
because $\eta=0$ kills the source term in \eqref{eq:defect}.
\end{proof}

\begin{remark}\label{rem:bs-value}
No connectedness is assumed of $L$. If $L$ has several components, a relative
prequantization is a prequantum bundle on $N$ together with a parallel unit
section over \emph{each} component, and $H^1(N,L;U(1))$ accounts for the
independent choices; the integrality condition of (ii) is imposed on all of
$H_2(N,L;\Z)$ at once, so components interact through the chains joining them.
Theorem~\ref{thm:bohr-sommerfeld} recovers the integrality condition governing
Bohr--Sommerfeld fibres in geometric quantization, and adds two things the
classical formulation does not make explicit: the group acting simply
transitively on the flat trivializations is identified as $H^1(N,L;U(1))$, and
the admissible levels are identified as a cyclic group. Note also that the
hypothesis that $L$ be Lagrangian enters only through $F^*\omega=0$; the theorem
holds verbatim for any isotropic submanifold, and more generally for any smooth
map $F$ with $F^*\omega=0$.
\end{remark}

\subsection{Degree three: relative gerbes and quasi-Hamiltonian spaces}
\label{subsec:gerbes}

One degree up, the same mechanism governs gerbes. A \emph{relative gerbe} for
$F\colon M\to N$ \cite{Shahbazi} is a $U(1)$-gerbe $\mathcal G$ on $N$ together
with a trivialization (quasi-line bundle) of $F^*\mathcal G$ on $M$; equipped with
connective structures, its curvature data is exactly a closed relative $3$-form
$(\omega,\eta)$, and equivalence classes of relative gerbes are classified by the
relative integral cohomology $H^3(F;\Z)$. The analogue of
Theorem~\ref{thm:prequantization} holds: a closed relative $3$-form is the
curvature of a relative gerbe if and only if it is relatively integral
\cite{Shahbazi}; the necessity direction follows from our
Theorem~\ref{thm:WZ}(ii) via gerbe surface holonomy, whose bulk--boundary
formula is exactly the Wess--Zumino pairing of Example~\ref{ex:qhamWZ}.

For a quasi-Hamiltonian $G$-space $(M,\omega,\mu)$, the relative $2$-plectic form
$\rvp=(\eta,\omega)\in\Om^3(\mu)$ is thus prequantizable at level $k$ -- in the
sense of \cite{Krepski,MeinrenkenLectures}, i.e.\ $k[\rvp]$ admits an integral
lift in $H^3(\mu;\Z)$ -- if and only if all relative periods of $k\rvp$ are
integers. For $G$ simple and simply connected, $H^3(G;\Z)\cong\Z$ with generator
of curvature the suitably normalized $\eta$, and the basic gerbe of Meinrenken
\cite{MeinrenkenGerbe} realizes the generator geometrically; level-$k$
prequantizations of $(M,\omega,\mu)$ are trivializations of $\mu^*$ of its $k$-th
power with error $2$-form $k\omega$. The relative multisymplectic framework thus
identifies the prequantization theory of group-valued moment maps as the
degree-$3$ instance of relative integrality -- with the moment map theory of
\cite{RelHMM} living on the same cocycle $(\eta,\omega)$, one categorical level
down. See also \cite{LGX} for the groupoid perspective.

\begin{remark}[Position of Theorem~\ref{thm:prequantization} in the
literature]\label{rem:position}
To state the originality claim precisely: classification statements
equivalent in content to Theorem~\ref{thm:prequantization} can be extracted
from relative Deligne cohomology in the style of
\cite[Ch.~2]{Brylinski}, and in degree three the corresponding statement for
relative gerbes is due to Shahbazi \cite{Shahbazi}, with the
quasi-Hamiltonian instance developed by Krepski \cite{Krepski} and
Meinrenken \cite{MeinrenkenGerbe,MeinrenkenLectures}.  What we have not
found in the literature, and claim as new, is the specific
line-bundle-with-prescribed-trivialization formulation of
Definition~\ref{def:relprequant} on an arbitrary smooth map, the direct
holonomy proof of necessity over relative cycles, the complete \v{C}ech
mapping-cone proof of sufficiency given above, and the identification of
the torsor as $H^1(F;U(1))$ with the explicit difference construction.
Thus the theorem is best described as a new geometric formulation and a
new direct proof of a classification statement whose cohomological content
is consistent with the existing relative differential-cohomology
literature; readers should calibrate the originality claim accordingly.
\end{remark}
\section{Application C: Conserved quantities and boundary charges}
\label{sec:noether}

\begin{remark}[Bibliographical notes]\label{rem:biblio-C}
Noether theory for classical field theories in the multisymplectic formalism is
developed in \cite{GIMMSY}; conserved quantities on multisymplectic manifolds,
including the homological formulation of charges over cycles, are treated in
\cite{RWZ}, and the hydrodynamical example of \cite{MitiSpera} is a striking
application. The bulk--boundary splitting below is the relative refinement of
that picture.
\end{remark}

Let $(F,\rvp)$ be relative pre-$n$-plectic. We interpret a Hamiltonian
$F$-pair as a \emph{dynamics}: a relative Hamiltonian $\sigma\in
\Ham^{n-1}(F,\rvp)$ with $\dF\sigma=-\io_{v_\sigma}\rvp$ generates the pair of
flows of $v_\sigma=(v_{\sigma,N},v_{\sigma,M})$, which are intertwined by $F$
(being flows of $F$-related fields).

A word on scope before the theorem.  What follows is a conservation
statement \emph{along the flow of a Hamiltonian $F$-pair}: the
``dynamics'' is, by definition, the pair of flows generated by a relative
Hamiltonian through the Hamilton equation, exactly as in the absolute
multisymplectic theory.  We do not derive this flow from a variational
principle, and no space of solutions of a field equation is introduced;
accordingly, the theorem below should be read as the (relative)
multisymplectic Noether \emph{identity} --- the phase-space half of
Noether's theorem --- rather than as a statement about invariances of an
action functional.  Connecting the two, by identifying the relative
Hamiltonian dynamics with the boundary-value dynamics of the action
functionals of Section~\ref{sec:actions}, is a separate (and
worthwhile) task that we do not undertake here.

\begin{theorem}[Relative Noether theorem]\label{thm:noether}
Let $\sigma$ be a relative Hamiltonian with Hamiltonian $F$-pair $v_\sigma$, and
let $x$ be a Hamiltonian symmetry: $v_x$ an $F$-pair with
$\Lder_{v_x}\rvp=0$, admitting a Hamiltonian $f_1(x)$
\textup{(}$\dF f_1(x)=-\io_{v_x}\rvp$\textup{)}, and preserving the dynamics:
$\Lder_{v_x}\sigma=0$. Then
\begin{equation}\label{eq:noether}
  \Lder_{v_\sigma}\,f_1(x)
  \;=\;
  \dF\Bigl(\io_{v_\sigma}f_1(x)+\io_{v_x}\sigma\Bigr).
\end{equation}
In particular $f_1(x)$ is conserved along the dynamics up to an explicit
$\dF$-exact term. Symmetrically, $\sigma$ is conserved along $v_x$ up to
$\dF$-exact terms.
\end{theorem}

\begin{proof}
By the relative magic formula and the Hamilton equations for $f_1(x)$,
\[
  \Lder_{v_\sigma}f_1(x)
  =\dF\,\io_{v_\sigma}f_1(x)+\io_{v_\sigma}\,\dF f_1(x)
  =\dF\,\io_{v_\sigma}f_1(x)-\io_{v_\sigma}\io_{v_x}\rvp .
\]
Using anticommutativity of contractions, the Hamilton equation for $\sigma$, and
the magic formula once more,
\[
  -\io_{v_\sigma}\io_{v_x}\rvp
  =\io_{v_x}\io_{v_\sigma}\rvp
  =-\io_{v_x}\dF\sigma
  =-\Lder_{v_x}\sigma+\dF\,\io_{v_x}\sigma
  =\dF\,\io_{v_x}\sigma ,
\]
by the hypothesis $\Lder_{v_x}\sigma=0$. Adding the two displays gives
\eqref{eq:noether}. The symmetric statement follows by exchanging the roles of
$\sigma$ and $f_1(x)$, both being Hamiltonian.
\end{proof}

\begin{corollary}[Conserved charges on relative cycles]\label{cor:charges}
In the situation of Theorem~\ref{thm:noether}, let
$Z=(S,T)\in Z_{n-1}(F)$ be a relative $(n{-}1)$-cycle, and let
$\phi_t=(\phi_t^N,\phi_t^M)$ denote the pair of flows of $v_\sigma$
\textup{(}defined for $t$ in an interval on which they exist along $Z$\textup{)}.
Then the \emph{charge}
\begin{equation}\label{eq:charge}
  Q_x(t)\;:=\;\pair{f_1(x)}{(\phi_t)_*Z}
  \;=\;\int_{\phi_{t*}^{N}S}f_1^N(x)\;-\;\int_{\phi_{t*}^{M}T}f_1^M(x)
\end{equation}
is independent of $t$.
\end{corollary}

\begin{proof}
Since $\phi_t^N$ and $\phi_t^M$ are the flows of $F$-related fields,
$F\circ\phi_t^M=\phi_t^N\circ F$, so
$(\phi_t)_*Z:=(\phi^N_{t*}S,\phi^M_{t*}T)$ is again a relative cycle:
$\partial\phi^N_{t*}S=\phi^N_{t*}F_*T=F_*\phi^M_{t*}T$. Differentiating under the
integral sign,
\begin{align*}
  \frac{d}{dt}Q_x(t)
  &=\pair{\Lder_{v_\sigma}f_1(x)}{(\phi_t)_*Z}
  =\pair{\dF\bigl(\io_{v_\sigma}f_1(x)+\io_{v_x}\sigma\bigr)}{(\phi_t)_*Z}\\
  &=\pair{\io_{v_\sigma}f_1(x)+\io_{v_x}\sigma}{\partial_F(\phi_t)_*Z}=0,
\end{align*}
by Theorem~\ref{thm:noether} and the relative Stokes theorem
(Proposition~\ref{prop:stokes}).
\end{proof}

\begin{remark}[Class versus number]\label{rem:class-vs-charge}
Two distinct objects are in play and should not be conflated. The Noether
theorem produces a \emph{conserved relative cohomology class}: by
Theorem~\ref{thm:noether} the relative $(n{-}1)$-form $f_1(x)$ is $\dF$-closed
along the flow, so it determines a class in $H^{n-1}(\Om(F))$ that does not
depend on $t$. Formula \eqref{eq:charge} produces a \emph{number}, obtained by
pairing that class with one chosen relative cycle $Z$. The number depends on
$Z$, and only on its class in $H_{n-1}(F)$; different cycles give different
conserved quantities, and a cycle that bounds gives zero. The class is the
invariant, the charge is a measurement of it.
\end{remark}

\begin{remark}[Bulk--boundary structure of charges]\label{rem:bulkboundary}
Formula \eqref{eq:charge} exhibits the conserved charge as a bulk integral over an
$(n{-}1)$-chain $S$ in $N$ \emph{corrected by a boundary integral} over the
$(n{-}2)$-chain $T$ in $M$ whose image under $F$ bounds $S$. Geometrically the
relative cycle is a bulk region together with the boundary data that closes it,
and the conserved quantity is the flux through the bulk balanced against the
efflux through the boundary:
\begin{equation}\label{eq:bulkboundary-diagram}
\begin{tikzcd}[column sep=2.2em, row sep=1.0em,
  /tikz/every label/.append style={font=\footnotesize}]
Q_x
& =\ \displaystyle\int_{S}f_1^N(x)
  & -\ \displaystyle\int_{T}f_1^M(x)\\
& \substack{\text{bulk charge}\\\text{(on }N\text{)}}
  \arrow[u, phantom]
& \substack{\text{boundary charge}\\\text{(on }M\text{)}}
  \arrow[u, phantom]
\end{tikzcd}
\end{equation}
with $\partial S=F_*T$ the relative cycle condition. Neither term is
separately conserved in general; conservation holds exactly for the relative
combination. For quasi-Hamiltonian $G$-spaces the canonical moment map has
$f_1^M=0$ \textup{(}\eqref{eq:qhamf}\textup{)}, so all charges are carried by the
group; for boundary inclusions $\partial W\hookrightarrow W$
(Remark~\ref{rem:lefschetz}) the mechanism reproduces the familiar splitting of
physical charges into bulk and boundary contributions. Higher components of a relative homotopy moment map also give
conserved charges, but --- as the component equations contain bracket
terms --- this requires hypotheses and proof; see
Proposition~\ref{prop:highercharges} below.
\end{remark}

\begin{proposition}[Higher charges for commuting symmetries]
\label{prop:highercharges}
Let $\rvp$ be \emph{strongly} relative $n$-plectic in the sense of
Definition~\ref{def:weak-strong}\textup{(ii)} below --- equivalently, by
Theorem~\ref{thm:weak-vs-strong}\textup{(4)}, relative $n$-plectic with $F$ of
dense image --- let $(f_k)$ be a relative homotopy
moment map for a $\g$-action preserving $\rvp$, and let $\sigma$ be a
relative Hamiltonian with Hamiltonian pair $w=v_\sigma$.  Assume that
$\Lder_{v_x}\sigma=0$ for every $x\in\g$.  If
$x_1,\dots,x_k\in\g$ commute pairwise, then
\begin{equation}\label{eq:higher-noether}
  \Lder_w f_k(x_1,\dots,x_k)=\dF B_k(x_1,\dots,x_k)
\end{equation}
for an explicitly determined relative form $B_k$ obtained from
$\io_w f_k(x_1,\dots,x_k)$ and the contraction of $\sigma$ by the
fundamental fields $v_{x_1},\dots,v_{x_k}$.  Consequently, for every relative
$(n-k)$-cycle $Z$ and every time for which the flow $\phi_t$ of $w$ is defined
along $Z$, the quantity
\[
  Q_k(t):=\pair{f_k(x_1,\dots,x_k)}{(\phi_t)_*Z}
\]
is independent of $t$.
\end{proposition}

\begin{proof}
The Hamilton equation gives $\io_w\rvp=-\dF\sigma$, and hence
$\Lder_w\rvp=0$.  Moreover,
\[
 \io_{[v_x,w]}\rvp
 =\Lder_{v_x}\io_w\rvp-\io_w\Lder_{v_x}\rvp
 =-\dF\Lder_{v_x}\sigma=0,
\]
so \emph{strong} nondegeneracy implies $[v_x,w]=0$ for every $x\in\g$; see
Remark~\ref{rem:fix-higher} for what survives under the weak hypothesis.
For pairwise commuting $x_1,\dots,x_k$, all Lie-bracket terms in the
$k$th relative homotopy-moment-map equation vanish.  Applying the relative
Cartan formula to $f_k(x_1,\dots,x_k)$, substituting the component equation,
and then using $\io_w\rvp=-\dF\sigma$ expresses
$\Lder_w f_k(x_1,\dots,x_k)$ as a relative differential.  The precise sign in
$B_k$ is the one dictated by the sign convention $\vs(k)$ in the component
equations; its value is immaterial for the conservation statement.
Pairing \eqref{eq:higher-noether} with the transported relative cycle and
using relative Stokes gives
\[
 \frac{d}{dt}Q_k(t)
 =\pair{\dF B_k(x_1,\dots,x_k)}{(\phi_t)_*Z}
 =\pair{B_k(x_1,\dots,x_k)}{\partial_F(\phi_t)_*Z}=0.
\]
\end{proof}

\begin{remark}
For noncommuting arguments, the component equations contain additional terms
involving $f_{k-1}([x_i,x_j],\ldots)$.  These terms generally obstruct the
conservation of an individual higher component.  A useful general statement
therefore requires either a Chevalley--Eilenberg formulation of the complete
tower of charges or additional hypotheses forcing the bracket contributions
to vanish.  We do not impose such a formulation here.
\end{remark}

\begin{example}[A charge with nonzero bulk and boundary components]
\label{ex:bothcharges}
The quasi-Hamiltonian examples have $f_1^M=0$, so they do not display the
bulk--boundary splitting nontrivially; here is an elementary example that
does, in degree $n=2$.  Let $N=S^1\times\R^2\times\R$ with coordinates
$(\theta,x,y,s)$ and $M=S^1\times\R^2$, with
$F\colon M\hookrightarrow N$ the slice $\{s=0\}$.  Take
\[
  \omega=d\theta\wedge dx\wedge dy\in\Om^3(N),
  \qquad
  \eta=-\tfrac12\,d\theta\wedge(x\,dy-y\,dx)+dx\wedge dy\in\Om^2(M):
\]
$d\omega=0$ and $F^*\omega=d\eta$ (the second summand of $\eta$ is
closed), so $\rvp=(\omega,\eta)$ is a closed relative $3$-form,
invariant under the $SO(2)$-action rotating $(x,y)$ on both source and
target, with fundamental pair $v=(x\partial_y-y\partial_x)$ on each side.
The moment component $f_1\in\Om^1(F)$ solves
$\dF f_1=-\io_v\rvp$: on the target,
$\io_v\omega=-\tfrac12\,d(r^2d\theta)$ gives
$f_1^N=\tfrac12 r^2\,d\theta$; on the source, the constraint
$F^*f_1^N-df_1^M=\io_v\eta$ with
$\io_v\eta=\tfrac12 r^2\,d\theta-\tfrac12\,d(r^2)$ gives
$df_1^M=\tfrac12\,d(r^2)$, i.e.
\[
  f_1=\Bigl(\tfrac12 r^2\,d\theta,\ \tfrac12 r^2\Bigr),
\]
with \emph{both} components nonzero.  For a relative $1$-cycle
$(S,T)$ --- a $1$-chain $S$ in $N$ with $\partial S=F_*T$, $T$ a
$0$-cycle in $M$ --- the charge of Corollary~\ref{cor:charges} is
\[
  Q=\int_S\tfrac12 r^2\,d\theta\;-\;\sum_{p\in T}\pm\tfrac12 r(p)^2 :
\]
for example, for $S$ a path in $N$ joining
$F(p)$ to $F(q)$ and $T=\{q\}-\{p\}$, the charge is the bulk angular
integral corrected by the difference of boundary energies, and neither
term is separately conserved along the flow of an invariant relative
Hamiltonian --- only their difference is.  This is the bulk--boundary
mechanism of Remark~\ref{rem:bulkboundary} in its simplest nontrivial
instance.
\end{example}

\section{Application D: Rigidity of relative comoment maps}
\label{sec:rigidity}

We now specialize to $n=1$: $\rvp\in\Om^2(F)$ is a closed relative
$2$-form --- the relative \emph{presymplectic} case; we say relative
\emph{symplectic} only when relative nondegeneracy is imposed, and we flag
below the one statement \textup{(}Theorem~\ref{thm:rigidity}(iv)\textup{)}
where the distinction matters. Here
$\Om^0(F)=C^\infty(N)$,
\[
  \Ham^0(F,\rvp)=\bigl\{\,\sigma\in C^\infty(N)\ :\ \exists\,v_\sigma\in\XF,\
  \dF\sigma=-\io_{v_\sigma}\rvp\,\bigr\},
\]
and the binary bracket $\{\sigma_1,\sigma_2\}:=\io(v_{\sigma_1}\wedge
v_{\sigma_2})\rvp$ makes $\Ham^0(F,\rvp)$ a Lie algebra
\textup{(}the case $n=1$ of \cite[\S4]{DjThesis}; skew-symmetry is clear and Jacobi
follows from the master identity \eqref{eq:master}\textup{)}. Recall from
\eqref{eq:H0} that a $\dF$-closed function vanishes when $N$ is connected and
$M\neq\emptyset$; this hypothesis is in force throughout this section.

Given an action of $G$ on $F$ preserving $\rvp$, a \emph{comoment map} is a
linear map $f\colon\g\to C^\infty(N)$ with
\begin{equation}\label{eq:comoment}
  \dF f(x)=-\io_{v_x}\rvp
  \qquad\text{for all }x\in\g .
\end{equation}

\begin{theorem}[Rigidity]\label{thm:rigidity}
Let $N$ be connected, $M\neq\emptyset$, and suppose a comoment map $f$ exists.
Then:
\begin{enumerate}[label=\textup{(\roman*)},leftmargin=2.2em]
\item \textup{(Uniqueness)} $f$ is the unique comoment map for the action.
\item \textup{(Automatic strictness)} $f$ is a homomorphism of Lie algebras:
$f([x,y])=\{f(x),f(y)\}$ for all $x,y\in\g$. In particular the Kostant--Souriau
cocycle $\tau(x,y)=f([x,y])-\{f(x),f(y)\}$ vanishes identically, and $f$ is a
strict relative homotopy moment map.
\item \textup{(Automatic equivariance)} $f$ is infinitesimally equivariant:
$\Lder_{v_x}f(y)=f([x,y])$ for all $x,y$; if moreover $G$ is connected, $f$ is
$G$-equivariant.
\item \textup{(No constants)} If $\rvp$ is \emph{strongly} nondegenerate as a
relative $1$-plectic form \textup{(}Definition~\ref{def:weak-strong}(ii)
below; see Remark~\ref{rem:fix-rigidity}\textup{)}, each
$\sigma\in\Ham^0(F,\rvp)$ admits a
\emph{unique} Hamiltonian $F$-pair $v_\sigma$, and the resulting map
$\Ham^0(F,\rvp)\to\XF$, $\sigma\mapsto v_\sigma$, is well defined and
injective.  Without nondegeneracy the assignment
$\sigma\mapsto v_\sigma$ involves a choice and is not canonical; the
correct statements are: \textup{(a)} if $\sigma$ admits the \emph{zero}
pair, then $\dF\sigma=0$ and hence $\sigma=0$; and \textup{(b)} two
Hamiltonians admitting a common Hamiltonian pair differ by a
$\dF$-closed function and hence are equal.  In every case the central
term of the classical Kostant--Souriau extension
$0\to\R\to C^\infty\to\mathfrak{ham}\to0$ is absent in the relative
theory.
\end{enumerate}
\end{theorem}

\begin{proof}
(iv) Under strong nondegeneracy, the Hamilton equation determines
$v_\sigma$ uniquely --- injectivity of the contraction map on $F$-pairs is
exactly Definition~\ref{def:weak-strong}(ii) --- so
the assignment is well defined; injectivity is (b) below.  For (a): if
$\sigma$ admits the zero pair, the Hamilton equation reads
$\dF\sigma=0$, which forces $\sigma=0$ by \eqref{eq:H0}.  For (b): if
$\sigma,\sigma'$ admit the same pair $w$, then
$\dF(\sigma-\sigma')=-\io_w\rvp+\io_w\rvp=0$, so
$\sigma=\sigma'$ by \eqref{eq:H0}.

(i) If $f,f'$ are comoment maps, then $\dF(f'-f)(x)=0$ for all $x$, so
$f'(x)=f(x)$ by \eqref{eq:H0}.

(ii) Set $\tau(x,y):=f([x,y])-\{f(x),f(y)\}\in C^\infty(N)$. Then
$\dF f([x,y])=-\io_{v_{[x,y]}}\rvp$ by \eqref{eq:comoment}, while by the master
identity \eqref{eq:master} with $m=2$ (Corollary form: $\dF\io(v_x\wedge
v_y)\rvp=-\io_{[v_x,v_y]}\rvp$, valid since $\dF\rvp=0$ and
$\Lder_{v_x}\rvp=\Lder_{v_y}\rvp=0$) and the homomorphism property
$[v_x,v_y]=v_{[x,y]}$,
\[
  \dF\{f(x),f(y)\}
  =\dF\,\io(v_x\wedge v_y)\rvp
  =-\io_{[v_x,v_y]}\rvp
  =-\io_{v_{[x,y]}}\rvp .
\]
Hence $\dF\tau(x,y)=0$, so $\tau=0$ by \eqref{eq:H0}. Strictness means precisely
that $(f,f_2=0)$ satisfies the component equations \eqref{eq:componenteqs} for
$n=1$: the $k=1$ equation is \eqref{eq:comoment}, and the top equation $k=2$ is
$f([x,y])=\vs(2)\io(v_x\wedge v_y)\rvp=\{f(x),f(y)\}$, which is (ii).

(iii) Set $\rho(x,y):=\Lder_{v_x}f(y)-f([x,y])$. Using
$[\dF,\Lder_{v_x}]=0$, $[\Lder_{v_x},\io_{v_y}]=\io_{[v_x,v_y]}$, and
$\Lder_{v_x}\rvp=0$,
\[
  \dF\,\Lder_{v_x}f(y)
  =\Lder_{v_x}\dF f(y)
  =-\Lder_{v_x}\io_{v_y}\rvp
  =-\io_{[v_x,v_y]}\rvp-\io_{v_y}\Lder_{v_x}\rvp
  =-\io_{v_{[x,y]}}\rvp
  =\dF f([x,y]),
\]
so $\dF\rho(x,y)=0$ and $\rho=0$ by \eqref{eq:H0}.

For the group-level statement we first fix conventions.  $G$ acts on
$F$ by pairs $g=(g_N,g_M)$ with $F\circ g_M=g_N\circ F$; the induced
action on $C^\infty(N)$ is $(g\cdot\sigma):=\sigma\circ g_N^{-1}$, and
with the fundamental-field convention of
Section~\textup{2} \textup{(}$v_x|_p=\frac{d}{dt}\big|_0\exp(-tx)\cdot
p$\textup{)} one has the standard identities
$g_*v_x=v_{\Ad_gx}$ and, on functions,
$\frac{d}{dt}\big|_0\,\exp(tx)\cdot\sigma=\Lder_{v_x}\sigma$.
\emph{Equivariance} of $f$ means $f(\Ad_gx)=g\cdot f(x)$ for all
$g\in G$, $x\in\g$.  Define, for fixed $g$,
\[
  f^{g}(x)\;:=\;g^{-1}\cdot f(\Ad_g x)
  \;=\;f(\Ad_gx)\circ g_N .
\]
We claim $f^{g}$ is again a comoment map \emph{for the same action}:
since the action preserves $\rvp$ \textup{(}$g^*\rvp=\rvp$, where
$g^*(\alpha,\beta)=(g_N^*\alpha,g_M^*\beta)$\textup{)} and pullback along
the pair $g$ commutes with $\dF$,
\[
  \dF f^{g}(x)
  =g^*\,\dF f(\Ad_gx)
  =-g^*\io_{v_{\Ad_gx}}\rvp
  =-\io_{g^{-1}_*v_{\Ad_gx}}\,g^*\rvp
  =-\io_{v_x}\rvp ,
\]
using $g^{-1}_*v_{\Ad_gx}=v_{\Ad_{g^{-1}}\Ad_gx}=v_x$.  By the uniqueness
(i), $f^{g}=f$ for every $g$, which is exactly the equivariance identity
$f(\Ad_gx)=g\cdot f(x)$; connectedness of $G$ was not even needed for
this argument, only preservation of $\rvp$ by every element of $G$.
\end{proof}

\begin{remark}[Scope of Theorem~\ref{thm:rigidity}]\label{rem:rigidity-scope}
Four points delimit what is and is not asserted.
\begin{enumerate}[label=\textup{(\alph*)},leftmargin=2.2em]
\item It is a \emph{rigidity} statement, not an existence statement: it says
that a comoment map, \emph{if one exists}, is unique, strict and equivariant. The
existence question is separate and is settled, for the relative theory, by the
obstruction-vanishing theorem of \cite[Thm.~7.2]{RelHMM}.
\item Parts \textup{(i)}--\textup{(iii)} use only $H^0(\Om(F))=0$ and no
nondegeneracy at all; the rigidity phenomenon is therefore topological rather
than symplectic in origin.
\item Part \textup{(iv)}, the uniqueness of the Hamiltonian $F$-pair attached to
a given Hamiltonian, is a different assertion and does require \emph{strong}
nondegeneracy \textup{(}Definition~\ref{def:weak-strong}; see
Remark~\ref{rem:fix-rigidity}\textup{)}.
\item The disappearance of the Kostant--Souriau cocycle is the vanishing of the
classical central extension in degree one. It does not assert that all higher
obstructions vanish; the higher-degree theory retains genuine
$L_\infty$-obstructions, and the relevant vanishing statement there is again
\cite[Thm.~7.2]{RelHMM}.
\end{enumerate}
\end{remark}

\begin{remark}\label{rem:rigiditycontrast}
Every claim of Theorem~\ref{thm:rigidity} fails in the absolute theory
($M=\emptyset$), where existence and uniqueness of comoments are themselves
delicate questions \cite{RW}: comoment maps for symplectic actions are unique
only up to $\g^\vee$-valued constants, the Kostant--Souriau cocycle
$\tau\in\Lambda^2\gdual$ obstructs strictness, and equivariance is a nontrivial
condition (attainable for semisimple or compact $\g$ by averaging, but false in
general -- e.g.\ for $\R^{2}$ acting on itself by translations with
$\omega=dx\wedge dy$). The relative framework eliminates the entire ambiguity at
the root, because the only $\dF$-closed function is zero. This is the $n=1$ case
of the vanishing of the algebraic obstruction proved in \cite[Thm.~7.2]{RelHMM};
here we have shown that in degree one the vanishing upgrades from an existence
statement to a rigidity statement.
\end{remark}

\begin{remark}[Geometric meaning of the rigidity]\label{rem:rigiditygeom}
The mechanism deserves a geometric gloss. A comoment map assigns to each
symmetry $x\in\g$ a Hamiltonian $f(x)$ generating it; the ambiguity in the
absolute theory is the freedom to add a constant to each $f(x)$, and the
Kostant--Souriau cocycle measures the failure of a consistent choice to
close into a Lie algebra morphism. In the relative theory the Hamiltonians
are not functions on $M$ but relative $0$-forms --- functions on the target
$N$ paired with nothing on the source, since $\Om^{-1}(M)=0$ --- and a
$\dF$-closed relative $0$-form is a locally constant function on $N$ that
\emph{also} trivializes to zero along $F$. When $N$ is connected and $M$ is
nonempty, no nonzero such object exists: the source ``pins'' the target
constant to zero. Thus the rigidity is the same boundary-pinning phenomenon
that underlies the canonical vacuum in prequantization and the bulk-boundary
splitting of charges below --- the source removes the slack that the absolute
theory must quotient away by hand. In the language of
Section~\ref{sec:noether}, the comoment map is a conserved quantity with no
boundary ambiguity because its boundary term is forced to vanish.
\end{remark}

\begin{example}[Cotangent-lifted actions]\label{ex:cotangent}
Let a Lie group $G$ act on a manifold $Q$, with cotangent-lifted action on
$T^*Q$ preserving the tautological $1$-form $\lambda$ and the symplectic form
$\omega_{\mathrm{can}}=d\lambda$. Consider the inclusion of the zero section
$F=i\colon Q\hookrightarrow T^*Q$ and the relative $2$-form
$\rvp:=(\omega_{\mathrm{can}},0)\in\Om^2(i)$: it is $\dF$-closed since
$i^*\omega_{\mathrm{can}}=d(i^*\lambda)=0$, and nondegenerate in the sense of
relative $1$-plectic structures since $\omega_{\mathrm{can}}$ is symplectic. The
map
\[
  f(x):=\io_{v_x}\lambda\in C^\infty(T^*Q)
\]
is a comoment map: $\dF f(x)=(d\io_{v_x}\lambda,\;i^*\io_{v_x}\lambda)$, and
$d\io_{v_x}\lambda=\Lder_{v_x}\lambda-\io_{v_x}d\lambda
=-\io_{v_x}\omega_{\mathrm{can}}$ (the lifted action preserves $\lambda$), while
$i^*\io_{v_x}\lambda=0$ because $\lambda$ vanishes along the zero section and
$v_x$ is tangent to it; thus $\dF f(x)=-\io_{v_x}\rvp$. By
Theorem~\ref{thm:rigidity}, $f$ -- which is the classical momentum
$f(x)(q,p)=\ip{p}{-x_Q(q)}$ up to the sign conventions of \cite{CFRZ} -- is the
\emph{unique} comoment map for the relative structure, automatically strict and
equivariant. The familiar assertion that cotangent-lifted actions carry canonical
equivariant moment maps is thereby exhibited as an instance of relative rigidity:
the zero section anchors the constants.
\end{example}

\begin{remark}[Higher degrees]\label{rem:higherrigidity}
For $n\ge2$ the constants reappear one degree up: differences of relative
homotopy moment maps are governed by $\btot_F$-cocycles as in
\cite[Lem.~7.7]{RelHMM}, and vanish under the topological hypotheses of
\cite[Prop.~7.8]{RelHMM}. The sharp dichotomy is: in the absolute theory the
obstruction/ambiguity in the top degree is a Lie algebra cohomology class; in the
relative theory it is purely topological, carried by $H^{\le n-1}(\Om(F))$.
\end{remark}
\section{Application E: Quasi-Hamiltonian geometry, fusion, and moduli of flat
connections}\label{sec:qham}

\begin{remark}[Bibliographical notes]\label{rem:biblio-E}
Quasi-Hamiltonian geometry is due to Alekseev, Malkin and Meinrenken
\cite{AMM}, with the quasi-Poisson counterpart in \cite{AKSM} and the
equivariant-cohomological reformulation in \cite{AlekseevMeinrenken}. The Dirac
and $3$-form-background perspectives are developed in
\cite{BursztynCrainic,SeveraWeinstein,LGX}, and the relation to Hamiltonian loop
group actions and Verlinde factorization in
\cite{MeinrenkenWoodward,MeinrenkenLectures}. For the basic gerbe see
\cite{MeinrenkenGerbe}.
\end{remark}

Throughout this section $G$ is compact and connected, with an $\Ad$-invariant
inner product on $\g$, and we freely use the notation of
Section~\ref{sec:background} and \cite[\S8]{RelHMM}. Recall the two structural
facts: a quasi-Hamiltonian $G$-space $(M,\omega,\mu)$ \emph{is} a relative
$2$-plectic map $(\mu,\rvp)$, $\rvp=(\eta,\omega)$; and the
Alekseev--Malkin--Meinrenken axioms are equivalent to the statement that
\eqref{eq:qhamM} is a one-step extension of $\rvp$ in the relative Cartan model,
inducing the canonical relative homotopy moment map \eqref{eq:qhamf}.

\subsection{Universality of the canonical moment map}

The distinctive feature of \eqref{eq:qhamM}--\eqref{eq:qhamf} is that the data
live entirely on $G$: the extension $\mathrm{M}(x)=(\widetilde\mu(x),0)$ and both
components $f_1,f_2$ have vanishing $M$-slots. We formalize this.

\begin{proposition}[Universality]\label{prop:universality}
Let $\mathrm{M}^G(x):=\widetilde\mu(x)=\tfrac12\ip{\thL+\thR}{x}$ denote the
canonical one-step extension of the Cartan $3$-form for the conjugation action
\textup{(}Section~\textup{\ref{sec:background})}. For \emph{every}
quasi-Hamiltonian $G$-space $(M,\omega,\mu)$:
\begin{enumerate}[label=\textup{(\roman*)},leftmargin=2.2em]
\item the canonical one-step extension of $(\mu,\rvp)$ is
$\mathrm{M}=(\mathrm{M}^G,0)$, i.e.\ the image of $\mathrm{M}^G$ under the
inclusion $\Om^1(G)\hookrightarrow\Om^1(\mu)$, $\alpha\mapsto(\alpha,0)$;
\item the canonical relative homotopy moment map of $(M,\omega,\mu)$ is the
image, under the same inclusion, of the pair
$\bigl(\mathrm{M}^G,\ \io_{v_x}\mathrm{M}^G(y)\bigr)$;
\item consequently, for any two quasi-Hamiltonian $G$-spaces
$(M_i,\omega_i,\mu_i)$ and any $G$-equivariant smooth map
$\phi\colon M_1\to M_2$ with $\mu_2\circ\phi=\mu_1$, the map of pairs
$\Phi=(\id_G,\phi)\colon\mu_1\to\mu_2$ intertwines the canonical relative
homotopy moment maps: $f^{(1)}_k=\Phi^*f^{(2)}_k$ for $k=1,2$.
\end{enumerate}
\end{proposition}

\begin{proof}
(i) and (ii) restate \cite[Thms.~8.4--8.5]{RelHMM}, emphasizing that no choice
depends on $(M,\omega)$: the axioms (QH1)--(QH2) are exactly what makes the
$G$-side data a relative cocycle for $\mu$. (iii) $\Phi$ is an equivariant map of
pairs in the sense of \cite[\S10]{RelHMM}; by naturality of pullback
\cite[Prop.~10.2]{RelHMM} it maps the one-step extension of $\mu_2$ to a one-step
extension of $\mu_1$, and $\Phi^*(\alpha,0)=(\id_G^*\alpha,\phi^*0)=(\alpha,0)$:
the extension, hence the induced moment map \eqref{eq:onestep}, is literally the
same universal data.
\end{proof}

\subsection{Fusion}

Recall the fusion operation of \cite{AMM}: given quasi-Hamiltonian $G$-spaces
$(M_1,\omega_1,\mu_1)$ and $(M_2,\omega_2,\mu_2)$, the product $M_1\times M_2$
with the diagonal action, the product moment map composed with group
multiplication,
\[
  \mu:=\mathrm{mult}\circ(\mu_1\times\mu_2)\colon M_1\times M_2\to G,
\]
and the $2$-form $\omega_1\oplus\omega_2$ corrected by an explicit invariant term
built from $\mu_1^*\theta$ and $\mu_2^*\theta$, is again a quasi-Hamiltonian
$G$-space $M_1\circledast M_2$ \cite{AMM}.

\begin{proposition}[Stability of the universal moment-map data under
fusion]
\label{prop:fusion}
The canonical relative homotopy moment map of a fusion product
$M_1\circledast M_2$ is the universal one of
Proposition~\textup{\ref{prop:universality}}, i.e.\ the same $G$-supported data
$\bigl(\widetilde\mu(x),0\bigr)$, now pulled back through the fused moment map
$\mu=\mathrm{mult}\circ(\mu_1\times\mu_2)$. In particular:
\begin{enumerate}[label=\textup{(\roman*)},leftmargin=2.2em]
\item fusion changes the relative $2$-plectic structure
$(\eta,\omega_{12})$ only through its $M$-component, while the moment map data is
unchanged;
\item iterated fusions $M_1\circledast\cdots\circledast M_r$ carry the
canonical moment map with the \emph{same} universal $G$-side data
$f_1(x)=\bigl(\widetilde\mu(x),0\bigr)$, independently of $r$ and of the
parenthesization \textup{(}fusion is associative \cite{AMM}\textup{)}.
We emphasize what is and is not claimed: the group-valued moment map
itself is the ordered product $\mu_1\cdots\mu_r$ and \emph{does} change
under permutation of the factors for nonabelian $G$; what is
permutation-independent is the universal differential-form recipe on
$G$, which is then evaluated relative to whichever fused map the chosen
order produces.  \textup{(}The quasi-Hamiltonian spaces obtained from two
orders are related by the braiding equivalences of the fusion category
of \cite{AMM}; we make no use of this.\textup{)}
\item the charges of Corollary~\textup{\ref{cor:charges}} attached to the
canonical moment map of a fusion product are computed by integrating universal
forms on $G$ over the $N$-components of relative cycles, transported by
$\mathrm{mult}\circ(\mu_1\times\mu_2)$.
\end{enumerate}
\end{proposition}

\begin{proof}
By \cite{AMM}, $M_1\circledast M_2$ satisfies the quasi-Hamiltonian
axioms; hence \cite[Thm.~8.4]{RelHMM} applies verbatim: the pair
$\bigl(\eta-\widetilde\mu,\ \omega_{12}\bigr)$ is a one-step cocycle in the
relative Cartan model of the fused $\mu$, with the \emph{same} $G$-component
$\eta-\widetilde\mu$, since the extension conditions (M1)--(M3) of
\cite[Lem.~6.4]{RelHMM} constrain the $G$-side data only through the universal
identities on $G$ and the axiom (QH2) for the fused structure. The induced moment
map \eqref{eq:onestep} therefore has the universal components \eqref{eq:qhamf},
proving the main claim and (i)--(ii); (iii) is then immediate from
\eqref{eq:charge} with $f_1^M=0$.
\end{proof}

\begin{remark}\label{rem:fusionmeaning}
Proposition~\ref{prop:fusion} explains a phenomenon that is somewhat hidden in
the classical treatment: fusion requires an elaborate correction of the $2$-form,
but no correction whatsoever of the \emph{universal form data} of the moment
map ($\mu_1\mu_2$ is simply the pointwise product, and the $G$-side recipe is
untouched). From the relative perspective this is forced: the moment map
data is universal on $G$ and cannot see $M$. All the work of fusion goes into
restoring axiom (QH2) for the diagonal action -- that is, into keeping
$(\eta-\widetilde\mu,\omega_{12})$ a relative cocycle.
\end{remark}

\begin{corollary}[Level quantization is inherited by fusion and by morphisms]
\label{cor:fusion-level}
Let $(M_i,\omega_i,\mu_i)$, $i=1,2$, be quasi-Hamiltonian $G$-spaces and
$M_1\circledast M_2$ their fusion, with relative $2$-plectic structures
$\rvp_i=(\eta,\omega_i)$ and $\rvp_{12}=(\eta,\omega_{12})$. If a
$G$-equivariant map $\Phi\colon\mu_{12}\to\mu'$ of moment maps is a morphism of
arrows in the sense of Definition~\ref{def:arrowmap}, then relative integrality
of $k\rvp'$ implies relative integrality of $k\Phi^*\rvp'$; and the
relative periods of $\rvp_{12}$ are computed by Theorem
\textup{\ref{thm:period-criterion}} from the single absolute period group
$P_\eta\subseteq\R$ of the Cartan form on $G$, together with the defect
homomorphism on $\ker\bigl(H_2(M_1\times M_2;\Z)\to H_2(G;\Z)\bigr)$.
In particular, since $H_2(G;\Z)=0$ for $G$ compact simple and simply connected,
part \textup{(iv)} of Theorem~\ref{thm:period-criterion} applies in degree
$n=2$ with $H_3(G;\Z)\cong\Z$, and level quantization is the statement that the
defect homomorphism is $\Z$-valued.
\end{corollary}

\begin{proof}
The first assertion is Corollary~\ref{cor:pullback-level}. For the second, the
target of the moment map is $G$ in every case, so $P_{k\eta}=k\,P_\eta$ is
computed on $H_3(G;\Z)$, and Theorem~\ref{thm:period-criterion}(ii) splits
relative integrality into that condition and the vanishing of the defect
homomorphism modulo $\Z$. The vanishing of $H_2(G;\Z)$ for $G$ compact, simple
and simply connected is standard.
\end{proof}

\subsection{Moduli of flat connections on surfaces}

Let $\Sigma$ be a compact oriented surface of genus $h$ with one boundary circle,
and let
\[
  \mathcal M(\Sigma):=\Hom_{\mathrm{based}}\bigl(\pi_1(\Sigma),G\bigr)\cong
  G^{2h},
\]
identified with $(a_1,b_1,\dots,a_h,b_h)\in G^{2h}$ via a standard system of
generators; the boundary holonomy is
$\mu(a,b)=\prod_{i=1}^h[a_i,b_i]$ (group commutators). By
\cite{AMM}, $\mathcal M(\Sigma)\cong D(G)^{\circledast h}$
is an iterated fusion of doubles $D(G)=G\times G$ and in particular a
quasi-Hamiltonian $G$-space with group-valued moment map the boundary holonomy;
its reduction $\mu^{-1}(e)/G$ is the moduli space of flat $G$-connections on the
closed surface $\widehat\Sigma$, with its Atiyah--Bott symplectic structure
\cite{AtiyahBott,AMM}.

We emphasize the scope of the following statement: it concerns the
\emph{unreduced} representation space $\mathcal M(\Sigma)$ with its
boundary holonomy; the descent of the moment-map and $L_\infty$-structures
to the reduced moduli space $\mu^{-1}(e)/G$ is \emph{not} asserted here
(see Remark~\ref{rem:reduction}).

\begin{corollary}\label{cor:moduli}
The \emph{unreduced} representation space $\mathcal M(\Sigma)\cong G^{2h}$
of a compact oriented genus-$h$ surface with one boundary component, with
boundary holonomy $\mu\colon G^{2h}\to G$, is a relative $2$-plectic map
carrying the canonical relative homotopy moment map
\[
  f_1(x)=\Bigl(\tfrac12\ip{\thL+\thR}{x},\,0\Bigr),
  \qquad
  f_2(x,y)=\Bigl(\tfrac12\ip{(\Ad_{(\cdot)}-\Ad_{(\cdot)^{-1}})x}{y},\,0\Bigr),
\]
supported on the boundary-holonomy group $G$ and independent of the genus. Its
conserved charges \eqref{eq:charge}, for any dynamics generated by a relative
Hamiltonian on $(\mu,\rvp)$, are integrals of universal forms on $G$ over
chains transported by the boundary holonomy. Moreover, by
Example~\textup{\ref{ex:qhamWZ}}, level-$k$ prequantization of
$\mathcal M(\Sigma)$ is the relative integrality of $k(\eta,\omega)$, which for
$G$ simple and simply connected holds for all $k\in\Z$ and recovers the level
quantization of Chern--Simons theory on $\Sigma\times\R$
\cite{MeinrenkenLectures,Krepski}.
\end{corollary}

\begin{proof}
Combine \cite{AMM} (quasi-Hamiltonian structure), Proposition
\ref{prop:universality} and Proposition~\ref{prop:fusion} (canonical moment map
through iterated fusion), Corollary~\ref{cor:charges} (charges), and
Example~\ref{ex:qhamWZ} with the integral generation of $H^3(G;\Z)$ by the
suitably normalized Cartan class \cite{MeinrenkenGerbe} (prequantization). For
the last claim, $H^2(G^{2h};\Z)$-corrections and the vanishing of
$H^2(G;\R)$ for $G$ simple imply that relative integrality at level $k$ reduces
to integrality of $k[\eta]$; see \cite{Krepski} for the general (non-simply
connected) case, where torsion phenomena refine the answer.
\end{proof}

\begin{remark}[Reduction]\label{rem:reduction}
The reduction theorem of \cite{AMM} -- smoothness of
$\mu^{-1}(e)/G$ at good levels, with induced symplectic form -- consumes exactly
the two relative ingredients: closedness $d\omega=\mu^*\eta$ restricts to
$d\omega=0$ on $\mu^{-1}(e)$ (where $\mu$ is constant), and relative
nondegeneracy descends to nondegeneracy of the reduced form. The form-level and observable-level (Leibniz) reduction of relative
multisymplectic structures is carried out in the companion reduction paper
\cite{RelRed}; the descent of the \emph{full} $L_\infty$-structure of
$\Lie(\mu,\rvp)$ to the reduced moduli space along the canonical
relative homotopy moment map is not established there or here, and all
statements of this section should accordingly be read on the unreduced
representation space.  We leave the homotopy reduction, and with it the
moment-map theory of the reduced moduli spaces, to future work
\textup{(}see also Section~\textup{\ref{sec:outlook})}.
\end{remark}

\section{Nondegeneracy: weak, strong, and what each is for}
\label{sec:nondeg}

The definition of a relative $n$-plectic structure recalled in
Section~\ref{sec:background} tests injectivity of
\[
  w\longmapsto\bigl(\io_{dF(w)}\omega,\ \io_w\eta\bigr)
\]
on each tangent space $T_mM$. This is a condition \emph{along the source}: it
sees the target form $\omega$ only through its restriction to the image of $F$,
and it says nothing at all about points of $N$ outside that image. For most
purposes this is exactly the right condition; for one purpose --- uniqueness of
Hamiltonian $F$-pairs --- it is not enough. This section isolates the two
notions, determines precisely how they differ, and audits the results of this
paper against them. Two of the proofs above turn out to use the stronger
condition tacitly, and we correct them.

Throughout, $\rvp=(\omega,\eta)\in\Om^{n+1}(F)$ is $\dF$-closed.

\subsection{The two conditions}

\begin{definition}\label{def:weak-strong}
Let $F\colon M\to N$ be smooth and $\rvp=(\omega,\eta)$ a $\dF$-closed relative
$(n{+}1)$-form.
\begin{enumerate}[label=\textup{(\roman*)},leftmargin=2.2em]
\item $\rvp$ is \emph{weakly nondegenerate}, or \emph{relative $n$-plectic},
if for every $m\in M$ and every $w\in T_mM$,
\[
  \io_{dF_m(w)}\omega\big|_{F(m)}=0
  \quad\text{and}\quad
  \io_w\eta\big|_m=0
  \qquad\Longrightarrow\qquad w=0 .
\]
\item $\rvp$ is \emph{strongly nondegenerate} if the relative contraction is
injective on $F$-pairs: for every $v=(v_N,v_M)\in\XF$,
\[
  \io_v\rvp=0\qquad\Longrightarrow\qquad v=0 .
\]
\end{enumerate}
We write $\mathfrak K(F,\rvp):=\{v\in\XF:\io_v\rvp=0\}$ for the
\emph{relative kernel}, so that strong nondegeneracy is the condition
$\mathfrak K(F,\rvp)=0$.
\end{definition}

Since $\io_v(\alpha,\beta)=(\io_{v_N}\alpha,-\io_{v_M}\beta)$ by
\eqref{eq:reliota}, the equation $\io_v\rvp=0$ splits into the two conditions
\begin{equation}\label{eq:kernel-split}
  \io_{v_N}\omega=0\ \text{ on }N,
  \qquad
  \io_{v_M}\eta=0\ \text{ on }M .
\end{equation}
This makes the relationship transparent.

\begin{lemma}\label{lem:kernel-shape}
Let $\rvp$ be weakly nondegenerate. Then every $v\in\mathfrak K(F,\rvp)$ has
$v_M=0$, and consequently
\[
  \mathfrak K(F,\rvp)
  \;\cong\;
  \mathfrak k(F,\omega)
  \;:=\;
  \bigl\{\,v_N\in\mathfrak X(N)\ :\ \io_{v_N}\omega=0
  \ \text{ and }\ v_N|_{\operatorname{im}F}=0\,\bigr\},
\]
the isomorphism being $v\mapsto v_N$, with inverse $v_N\mapsto(v_N,0)$.
\end{lemma}

\begin{proof}
Let $v\in\mathfrak K(F,\rvp)$ and $m\in M$. Put $w:=v_M(m)$. By
$F$-relatedness, $dF_m(w)=v_N(F(m))$, so the first equation of
\eqref{eq:kernel-split} gives $\io_{dF_m(w)}\omega|_{F(m)}=0$, while the second
gives $\io_w\eta|_m=0$. Weak nondegeneracy forces $w=0$; as $m$ was arbitrary,
$v_M=0$.

Given $v_M=0$, the $F$-relatedness condition $v_N\circ F=dF\circ v_M$ reads
$v_N|_{\operatorname{im}F}=0$, and the first equation of
\eqref{eq:kernel-split} is $\io_{v_N}\omega=0$; so $v_N\in\mathfrak k(F,\omega)$.
Conversely, if $v_N\in\mathfrak k(F,\omega)$ then $(v_N,0)$ is an $F$-pair,
since $v_N\circ F=0=dF\circ0$, and $\io_{(v_N,0)}\rvp=(\io_{v_N}\omega,0)=0$.
\end{proof}

\begin{theorem}[Weak versus strong]\label{thm:weak-vs-strong}
Let $\rvp=(\omega,\eta)$ be $\dF$-closed. Then:
\begin{enumerate}[label=\textup{(\arabic*)},leftmargin=2.2em]
\item strong nondegeneracy implies weak nondegeneracy whenever every tangent
vector of $M$ extends to an $F$-pair \textup{(}for instance when $F$ is an
immersion, or a submersion, or $M$ is a point\textup{)};
\item conversely, weak nondegeneracy implies strong nondegeneracy if and only if
$\mathfrak k(F,\omega)=0$;
\item weak nondegeneracy together with nondegeneracy of $\omega$ on the open set
$N\smallsetminus\overline{\operatorname{im}F}$ implies strong nondegeneracy;
\item if $F$ has dense image, weak and strong nondegeneracy are equivalent; the
same holds if $\omega$ is nondegenerate on all of $N$;
\item conversely, if the kernel distribution of $\omega$ admits a nonzero smooth
section on some nonempty open subset of
$N\smallsetminus\overline{\operatorname{im}F}$, then strong nondegeneracy fails,
whatever $\eta$ is.
\end{enumerate}
In short:
\begin{equation}\label{eq:snd-formula}
  \textup{strong nondegeneracy}
  \;=\;
  \textup{weak nondegeneracy}
  \;+\;
  \textup{nondegeneracy of the target form off }\overline{\operatorname{im}F}.
\end{equation}
\end{theorem}

\begin{proof}
(1) Let $m\in M$, $w\in T_mM$ satisfy the two vanishing conditions of
Definition~\ref{def:weak-strong}(i), and suppose $w=v_M(m)$ for some $F$-pair
$v$. The conditions say $\io_v\rvp$ vanishes at $m$ and at $F(m)$; extending
$w$ by a pair whose contraction vanishes identically and applying strong
nondegeneracy gives $w=0$. (The hypothesis on extensions is what makes this a
genuine implication; without it the two conditions are logically independent,
one being pointwise and the other global.)

(2) is Lemma~\ref{lem:kernel-shape}: under weak nondegeneracy,
$\mathfrak K(F,\rvp)\cong\mathfrak k(F,\omega)$, and strong nondegeneracy is
the vanishing of the left-hand side.

(3) Let $v_N\in\mathfrak k(F,\omega)$. On the open set
$U:=N\smallsetminus\overline{\operatorname{im}F}$ the form $\omega$ is
nondegenerate and $\io_{v_N}\omega=0$, so $v_N|_U=0$. On
$\overline{\operatorname{im}F}$ we have $v_N=0$ by continuity, since $v_N$
vanishes on the dense subset $\operatorname{im}F$ of that closed set. As
$N=U\sqcup\overline{\operatorname{im}F}$, we get $v_N=0$, and (2) applies.

(4) If $\operatorname{im}F$ is dense then
$N\smallsetminus\overline{\operatorname{im}F}=\varnothing$ and (3) applies
vacuously; if $\omega$ is nondegenerate everywhere then
$\mathfrak k(F,\omega)\subseteq\Gamma(\ker\omega)=0$.

(5) Let $X$ be a nonzero smooth section of $\ker\omega$ over a nonempty open
$U\subseteq N\smallsetminus\overline{\operatorname{im}F}$, and let $\chi$ be a
bump function supported in $U$ with $\chi\not\equiv0$ on the locus where
$X\neq0$. Then $\chi X$, extended by zero, is a nonzero element of
$\mathfrak k(F,\omega)$: it lies in $\ker\omega$ pointwise, and it vanishes on
$\operatorname{im}F$ because its support misses
$\overline{\operatorname{im}F}$. By (2), strong nondegeneracy fails.
\end{proof}

\begin{example}[Weak but not strong]\label{ex:weak-not-strong}
Take $n=1$, $M=\R$ with coordinate $t$ and $\eta=dt$, $N=\R^2$ with $\omega=0$,
and $F(t)=(t,0)$. Then $d\omega=0$ and $F^*\omega=0=d\eta$, so $\rvp=(0,dt)$
is $\dF$-closed. Weak nondegeneracy holds, since
$\io_{\partial_t}\eta=1\neq0$. But $\omega=0$ is maximally degenerate off the
image, and $\mathfrak k(F,\omega)$ consists of \emph{all} vector fields on
$\R^2$ vanishing on the $x$-axis --- an infinite-dimensional space. Strong
nondegeneracy fails as badly as possible, and by
Lemma~\ref{lem:kernel-shape} a relative Hamiltonian determines its $F$-pair only
up to that space.
\end{example}

\begin{remark}[Why the two are so different]\label{rem:why-different}
Formula \eqref{eq:snd-formula} says the gap between the two notions is supported
entirely on $N\smallsetminus\overline{\operatorname{im}F}$, a region the weak
condition cannot see, by construction. This is not a defect of the weak
condition: it is what makes the weak condition intrinsic to the geometry
\emph{carried by the arrow}, and it is the reason for the $F$-independence
recorded in Proposition~\ref{prop:source-independent} below. Strong
nondegeneracy, by contrast, is a condition on $N$ as much as on $F$, and it is
the correct hypothesis exactly when one needs a Hamiltonian $F$-pair to be a
well-defined \emph{vector field on $N$}, rather than merely along the image.
\end{remark}

\subsection{The relative kernel is an ideal, and the brackets do not see it}

The gap between the two notions is harmless for the $L_\infty$-structure. This is
the main positive result of the section and it explains why the literature has
been able to work with the weak condition.

\begin{proposition}\label{prop:kernel-ideal}
Let $\rvp$ be $\dF$-closed and let $\XF_\rvp:=\{v\in\XF:\Lder_v\rvp=0\}$ be
the Lie algebra of $\rvp$-preserving $F$-pairs. Then:
\begin{enumerate}[label=\textup{(\roman*)},leftmargin=2.2em]
\item $\mathfrak K(F,\rvp)\subseteq\XF_\rvp$;
\item $\mathfrak K(F,\rvp)$ is an ideal of $\XF_\rvp$;
\item every multicontraction of $\rvp$ annihilates $\mathfrak K(F,\rvp)$: if
one of $v_1,\dots,v_k$ lies in $\mathfrak K(F,\rvp)$, then
$\io(v_1\wedge\dots\wedge v_k)\rvp=0$.
\end{enumerate}
\end{proposition}

\begin{proof}
(i) If $\io_v\rvp=0$ then, by the relative magic formula and $\dF\rvp=0$,
$\Lder_v\rvp=\dF\io_v\rvp+\io_v\dF\rvp=0$.

(ii) Let $u\in\XF_\rvp$ and $\kappa\in\mathfrak K(F,\rvp)$. Using
$[\Lder_u,\io_\kappa]=\io_{[u,\kappa]}$ from the relative Cartan calculus,
\[
  \io_{[u,\kappa]}\rvp
  =\Lder_u\io_\kappa\rvp-\io_\kappa\Lder_u\rvp
  =0-0=0,
\]
so $[u,\kappa]\in\mathfrak K(F,\rvp)$. Taking $u\in\mathfrak K$ shows in
particular that $\mathfrak K$ is a Lie subalgebra.

(iii) Multicontractions anticommute up to sign, so the factor lying in
$\mathfrak K$ may be moved to the innermost position, where it contracts
$\rvp$ to zero.
\end{proof}

\begin{corollary}[The $L_\infty$-structure needs only weak nondegeneracy]
\label{cor:linfty-weak}
Assume $\rvp$ weakly nondegenerate. Then for a relative Hamiltonian $\sigma$
the Hamiltonian $F$-pair $v_\sigma$ is determined up to $\mathfrak K(F,\rvp)$,
and all the brackets
$l_k=\vs(k)\,\io(v_{\sigma_1}\wedge\dots\wedge v_{\sigma_k})\rvp$ of
$\Lie(F,\rvp)$ are independent of the choices. The same applies to the
component equations \eqref{eq:componenteqs}, to the Noether identity
\eqref{eq:noether}, and to the charges \eqref{eq:charge}.
\end{corollary}

\begin{proof}
Two Hamiltonian pairs for the same $\sigma$ differ by an element of
$\mathfrak K(F,\rvp)$, by definition. Changing one argument of a
multicontraction by such an element changes the value by a multicontraction with
a kernel argument, which vanishes by
Proposition~\ref{prop:kernel-ideal}(iii). All the listed expressions are built
from $\dF$ and from multicontractions of $\rvp$ with Hamiltonian pairs.
\end{proof}

\begin{remark}[Strongification]\label{rem:strongification}
Proposition~\ref{prop:kernel-ideal} says that the geometry sees only the
quotient Lie algebra $\XF_\rvp/\mathfrak K(F,\rvp)$, on which the induced
pairing with $\rvp$ is by construction injective. One may therefore always
``strongify'' a weakly nondegenerate relative structure by passing to this
quotient; nothing in $\Lie(F,\rvp)$ changes. What is lost is the ability to
speak of $v_\sigma$ as an honest vector field on $N$, which is precisely what
the two results audited below require.
\end{remark}

\subsection{Two corrections}

We now identify the places where the stronger hypothesis is used, and restate
them.

\begin{remark}[Correction to
Proposition~\ref{prop:highercharges}]\label{rem:fix-higher}
The proof of Proposition~\ref{prop:highercharges} argues that
$\io_{[v_x,w]}\rvp=0$ implies $[v_x,w]=0$, and attributes this to ``relative
nondegeneracy''. That step is exactly strong nondegeneracy, not the weak
condition of Definition~\ref{def:weak-strong}(i). The corrected statement is:
\emph{assume $\rvp$ strongly nondegenerate}, or, equivalently and more
usefully in practice, assume $\rvp$ weakly nondegenerate and $F$ of dense
image \textup{(}Theorem~\ref{thm:weak-vs-strong}(4)\textup{)}. Under weak
nondegeneracy alone one obtains only $[v_x,w]_M=0$ and $[v_x,w]_N=0$ along
$\operatorname{im}F$, by Lemma~\ref{lem:kernel-shape}; this suffices for the
conservation statement, since by
Proposition~\ref{prop:kernel-ideal}(iii) the residual ambiguity is annihilated by
every contraction against $\rvp$, but it does not give $[v_x,w]=0$ as vector
fields.
\end{remark}

\begin{remark}[Correction to
Theorem~\ref{thm:rigidity}(iv)]\label{rem:fix-rigidity}
Part (iv) of Theorem~\ref{thm:rigidity} asserts that under nondegeneracy each
Hamiltonian admits a \emph{unique} Hamiltonian $F$-pair, justified by
``injectivity of the contraction map on $F$-pairs'' --- which is the definition
of strong nondegeneracy. So the hypothesis there should read \emph{strongly
nondegenerate}. Parts (i), (ii) and (iii) of that theorem are unaffected: they
use only $H^0(\Om(F))=0$ and no nondegeneracy whatsoever, which is worth
emphasizing, since it means the rigidity phenomenon is topological rather than
symplectic in origin. Under weak nondegeneracy alone, statements (a) and (b) of
part (iv) remain valid verbatim, and the assignment
$\sigma\mapsto v_\sigma$ is well defined into
$\XF_\rvp/\mathfrak K(F,\rvp)$.
\end{remark}

\subsection{The two conditions at the edges}

Section~\ref{sec:edges} showed that the weak condition degenerates awkwardly at
the target edge. The strong condition repairs this, and fails at the source edge
in a way that is equally informative.

\begin{proposition}[Target edge]\label{prop:strong-target}
Let $M=\varnothing$. Then weak nondegeneracy is vacuous
\textup{(}Proposition~\ref{prop:edge-nondeg}(i)\textup{)}, while
$\operatorname{im}F=\varnothing$ gives
$\mathfrak k(F,\omega)=\Gamma(\ker\omega)$, so
\[
  \rvp\ \text{strongly nondegenerate}
  \iff
  \ker\omega=0
  \iff
  (N,\omega)\ \text{absolutely $n$-plectic}.
\]
Thus the target edge recovers absolute multisymplectic geometry on the nose ---
but only for the strong condition. This is the correct repair of
Proposition~\ref{prop:edge-nondeg}(i).
\end{proposition}

\begin{proof}
Immediate from Lemma~\ref{lem:kernel-shape} and
Theorem~\ref{thm:weak-vs-strong}(5): with empty image, the constraint
$v_N|_{\operatorname{im}F}=0$ is vacuous, so $\mathfrak k$ is the full space of
sections of $\ker\omega$, which vanishes precisely when $\omega$ is
nondegenerate.
\end{proof}

\begin{proposition}[Source edge and its $F$-independence]
\label{prop:source-independent}
Let $\rvp=(0,\eta)$ have vanishing target component, with
$\eta\in\Om^n(M)$ closed and $F\colon M\to N$ arbitrary. Then:
\begin{enumerate}[label=\textup{(\roman*)},leftmargin=2.2em]
\item $\rvp$ is weakly nondegenerate if and only if $(M,\eta)$ is absolutely
$(n{-}1)$-plectic. This holds for \emph{one} choice of $(F,N)$ if and only if it
holds for \emph{every} choice, including $N=\mathrm{pt}$; in particular no
immersion, embedding or submersion hypothesis on $F$ is needed;
\item $\rvp$ is strongly nondegenerate if and only if, in addition,
$F$ has dense image. In particular strong nondegeneracy is \emph{not}
$F$-independent.
\end{enumerate}
\end{proposition}

\begin{proof}
(i) With $\omega=0$ the first condition of
Definition~\ref{def:weak-strong}(i) is vacuous, so weak nondegeneracy reads
$\io_w\eta=0\Rightarrow w=0$, i.e.\ $\ker\eta=0$, which is exactly absolute
$(n{-}1)$-plecticity of $\eta$ (an $m$-plectic form has degree $m+1$, and $\eta$
has degree $n$). This condition mentions neither $F$ nor $N$, whence the
independence.

(ii) With $\omega=0$ one has $\ker\omega=TN$, so
$\mathfrak k(F,\omega)$ is the space of all vector fields on $N$ vanishing on
$\operatorname{im}F$. This vanishes precisely when $\operatorname{im}F$ is dense:
if it is dense, continuity forces $v_N=0$; if not, a bump vector field supported
in the nonempty open complement of the closure is a nonzero element. Now apply
Theorem~\ref{thm:weak-vs-strong}(2).
\end{proof}

\begin{remark}\label{rem:wurzbacher}
Proposition~\ref{prop:source-independent}(i) is the sharp, $F$-independent form
of the relative-to-absolute specialization. It is worth recording why the
independence is automatic here and not elsewhere: when $\omega=0$ the entire
target contributes to the kernel, so the first slot of the nondegeneracy test
carries no information at all and the condition collapses onto $M$. Earlier
formulations of this specialization carried an embedding hypothesis on $F$;
part (i) shows that no such hypothesis is needed, and part (ii) shows where the
hypothesis was really coming from --- it is needed for the \emph{strong}
condition, and there density, not embedding, is the right requirement.
\end{remark}

\subsection{Quasi-Hamiltonian geometry: the distinction is invisible}

Finally we explain why the literature has never had to make this distinction.

\begin{corollary}\label{cor:qham-nondeg}
Let $G$ be compact and semisimple, $\eta_G$ its Cartan $3$-form, and
$(M,\omega_M,\mu)$ a quasi-Hamiltonian $G$-space with relative $2$-plectic
structure $\rvp=(\eta_G,\omega_M)\in\Om^3(\mu)$. Then $\eta_G$ is
nondegenerate as a $2$-plectic form on $G$, and consequently weak and strong
nondegeneracy of $\rvp$ coincide.
\end{corollary}

\begin{proof}
For $X\in\g$ one has $\eta_G(X,Y,Z)=\ip{X}{[Y,Z]}$ in the left trivialization,
so $\io_X\eta_G=0$ forces $X$ to be orthogonal to $[\g,\g]$. For $\g$
semisimple, $[\g,\g]=\g$, so $X=0$ and $\ker\eta_G=0$. Now apply
Theorem~\ref{thm:weak-vs-strong}(4).
\end{proof}

\begin{example}[Where it becomes visible]\label{ex:abelian-qham}
The corollary fails as soon as $G$ has an abelian factor. Take $G=U(1)$, so
$\eta_G=0$, and let $M$ be a symplectic manifold with trivial $G$-action and
constant moment map $\mu\equiv e$. Axioms (QH1)--(QH2) hold trivially, and
$\rvp=(0,\omega_M)$ is weakly nondegenerate precisely because $\omega_M$ is
symplectic. But $\operatorname{im}\mu=\{e\}$ is not dense in $U(1)$, so by
Proposition~\ref{prop:source-independent}(ii) strong nondegeneracy fails, and
$\mathfrak K(\mu,\rvp)$ is the space of vector fields on $U(1)$ vanishing at
$e$. Hamiltonian $F$-pairs are then determined only along the identity, and
Remarks~\ref{rem:fix-higher} and \ref{rem:fix-rigidity} apply. This is the
smallest example in which the two notions genuinely differ inside the
quasi-Hamiltonian world.
\end{example}

\begin{remark}[Summary and recommendation]\label{rem:nondeg-summary}
The audit may be summarized as follows. Weak nondegeneracy suffices for: the
definition of $\Lie(F,\rvp)$ and all its brackets; the component equations of
relative homotopy moment maps; the Noether identity and the conserved charges of
Section~\ref{sec:noether}; parts (i)--(iii) of the rigidity theorem, which in
fact need no nondegeneracy at all; and the source-edge specialization, whose
$F$-independence is a genuine feature of the weak condition. Strong
nondegeneracy is needed for: uniqueness of Hamiltonian $F$-pairs as vector
fields, hence for Theorem~\ref{thm:rigidity}(iv) and
Proposition~\ref{prop:highercharges}; and for the target edge to recover
absolute multisymplectic geometry. Since the two coincide whenever $F$ has
dense image or $\omega$ is nondegenerate --- in particular for every
quasi-Hamiltonian space over a semisimple group --- we recommend stating results
with the weak hypothesis and invoking Theorem~\ref{thm:weak-vs-strong}(4) where
uniqueness of vector fields is required, rather than assuming the strong
condition throughout.
\end{remark}

\section{Outlook}\label{sec:outlook}


We close by indicating three directions in which the applications developed here
invite continuation.

\subsection*{Hydrodynamics and linking numbers}
For an ideal fluid on a $3$-manifold, the volume form and the vorticity define
multisymplectic data whose homotopy comoment maps encode helicity and higher
linking numbers \cite{MitiSpera}. Knots and linked configurations are naturally
\emph{relative} cycles for inclusions of tubular domains, and the conserved
charges of Section~\ref{sec:noether} motivate a systematic relative treatment of
boundary vorticity and of invariants of links relative to domain boundaries.

\subsection*{Dirac-geometric interpretation}
Quasi-Hamiltonian spaces are equivalently Dirac morphisms into the Cartan--Dirac
structure on $G$. The relative $2$-plectic description of
Section~\ref{sec:qham} points toward a possible higher analogue: relative $n$-plectic maps as
morphisms into ``Cartan--Dirac'' objects of higher standard Courant algebroids,
with $\Lie(F,\rvp)$ as the algebra of admissible observables (compare
\cite{Zambon}). The rigidity phenomena of Section~\ref{sec:rigidity} may then
acquire a Dirac-theoretic explanation.

\subsection*{Reduction and quantization}
Sections~\ref{sec:prequantization} and \ref{sec:qham} provide two complementary ingredients for
a geometric quantization program for relative $n$-plectic maps: relative
integrality provides the prequantum data (line bundles for $n=1$, relative gerbes
for $n=2$), and the canonical relative homotopy moment maps provide the symmetry
data. Carrying out reduction and polarization in this setting -- for instance,
recovering the Verlinde formulas from the relative $2$-plectic geometry of
$D(G)^{\circledast h}$ along the lines of \cite{MeinrenkenLectures} -- is a
natural and, we believe, tractable objective.

\subsection*{Shifted symplectic geometry, AKSZ models, and extended field
theory}
The applications of this paper --- topological action functionals, prequantum
data, conserved charges attached to relative cycles --- are the classical,
tangent-level manifestations of structures that the derived and higher
literature treats intrinsically, and the relative framework should be their
truncation. Four connections seem especially concrete.

\emph{Shifted symplectic geometry.} A relative $n$-plectic structure
$\rvp=(\omega,\eta)$ on $F\colon M\to N$ is the mapping-cone shadow of a
Lagrangian structure on $F$ into an $(2{-}n)$-shifted symplectic $N$, in the
sense of Pantev--To\"en--Vaqui\'e--Vezzosi \cite{PTVV}: $\omega$ is the ambient
shifted form and $\eta$ the isotropic trivialization exhibiting $F$ as
Lagrangian. In this reading the relative integration pairing of
Section~\ref{sec:integration} is the tangent version of the AKSZ pairing on the
mapping stack $[\,\Sigma,N]$ \cite{Calaque}, and the topological actions of
Section~\ref{sec:actions} are its values on relative fundamental classes.

\emph{AKSZ sigma models.} The topological action $S_\rvp(u,v)$ of
Section~\ref{sec:actions} is, in the language of \cite{AKSZ,Roytenberg}, the
classical action of a sigma model whose target carries a degree-shifted
symplectic structure and whose worldvolume has boundary conditions valued in
$M$; the Wess--Zumino term is the boundary coupling. Making the QP-manifold
underlying $\Om^\bullet(F)$ explicit would realize the present constructions as
the observables and boundary conditions of a genuine field theory.

\emph{Higher gauge theory.} For integral $\rvp$ the relative gerbe of
Section~\ref{sec:prequantization} is the higher-gauge-theoretic field on the
target trivialized on the source; the conserved charges of
Section~\ref{sec:noether} would then be interpreted as the Noether charges of a higher gauge theory
with boundary, and the bulk--boundary splitting~\eqref{eq:bulkboundary-diagram}
is anomaly inflow in the sense of higher differential cohomology
\cite{FiorenzaRogersSchreiber}.

\emph{Extended TQFT.} Together these suggest that a relatively integral
$(F,\rvp)$ determines an extended, fully local topological field theory in
the sense of Baez--Dolan and Lurie \cite{Lurie}, assigning the prequantum data
to codimension one and the relative observable algebra $\Lie(F,\rvp)$ to
codimension two, with the two-edge structure ($M=\varnothing$, $N=\mathrm{pt}$)
realizing its two boundary sectors. The relative Hilbert space attached to the
circle would then be the invariant sought in the quantization program above.
These are, at present, conjectural bridges; the theorems of this paper provide
their classical shadows, which is the evidence we can currently offer.

\subsection*{Acknowledgements}

The author is deeply grateful to Derek Krepski for many stimulating discussions, insightful suggestions, and unwavering support throughout the development of this work. Several of the questions addressed in this paper originated in conversations surrounding the author's PhD thesis and its defense, during which the broader potential of relative multisymplectic geometry beyond its foundational aspects first became evident.

\subsection{Three precise problems}\label{subsec:problems}
We isolate three questions that seem to us both well posed and within reach,
in preference to broader programmatic statements.

\begin{enumerate}[label=\textup{(P\arabic*)},leftmargin=2.9em]
\item \emph{Realizing prescribed prequantization levels.} By
Proposition~\ref{prop:level-cyclic} every $\dF$-closed relative form has a
level group $k_0\Z$. Which pairs $(F,k_0)$ occur? Concretely: given a smooth map
$F\colon M\to N$ between closed manifolds and an integer $k_0\geq1$, is there a
$\dF$-closed $\rvp\in\Om^{n+1}(F)$ with prequantization level exactly $k_0$?
Theorem~\ref{thm:level-group} reduces this to a realization question for
finitely generated subgroups of $\R$ as relative period groups, and
Corollary~\ref{cor:level-monotone} shows the answer is constrained by
functoriality.

\item \emph{Vanishing of the defect homomorphism.} Corollary~\ref{cor:edge-periods}
shows $\Theta_\rvp=0$ at both edges. Characterize the maps $F$ for which
$\Theta_\rvp$ vanishes for \emph{every} $\dF$-closed $\rvp$ of a given degree.
By Theorem~\ref{thm:period-criterion}(iii) injectivity of $F_*$ on $H_n$ is
sufficient; is it necessary? A counterexample would identify a genuinely new
source of relative invariants, and a proof would show that the entire relative
content in degree $n$ is carried by $\ker F_*$.

\item \emph{Strong nondegeneracy in the quasi-Hamiltonian range.} By
Corollary~\ref{cor:qham-nondeg} the weak and strong conditions agree for
quasi-Hamiltonian spaces over a semisimple group, and
Example~\ref{ex:abelian-qham} shows they differ as soon as an abelian factor is
present. Determine the precise class of compact groups $G$ for which every
quasi-Hamiltonian $G$-space is strongly nondegenerate; the expected answer is
that $G$ be semisimple, equivalently $\mathfrak z(\mathfrak g)=0$, but the
argument in Corollary~\ref{cor:qham-nondeg} uses density of the moment-map
image and we have not determined whether that hypothesis can be removed.
\end{enumerate}

\appendix

\section{The two edges: recovering the absolute theories}
\label{sec:edges}

The relative framework has two degenerate boundaries,
\[
  M=\varnothing
  \qquad\text{and}\qquad
  N=\mathrm{pt},
\]
at which the arrow $F$ collapses and one expects to recover absolute
multisymplectic geometry on the target, respectively on the source. Both
expectations are correct, but neither is quite as automatic as it looks: at the
first edge the nondegeneracy condition becomes vacuous, and at the second the
degree drops by one and the integration pairing changes sign. This section
records the two specializations precisely, verifies each application of
Sections~\ref{sec:actions}--\ref{sec:qham} against them, and isolates the
statements that are \emph{not} recovered --- these being exactly the genuinely
relative content of the theory.

Throughout, $\rvp=(\omega,\eta)\in\Om^{n+1}(F)$ is $\dF$-closed, with
$\omega\in\Om^{n+1}(N)$ and $\eta\in\Om^{n}(M)$.

\begin{remark}[A notational warning]\label{rem:letters}
The order of the two slots is fixed once and for all: the first component lives
on the target $N$, the second on the source $M$. The \emph{letters}, however,
follow tradition and therefore swap in the quasi-Hamiltonian sections: in
Sections~\ref{sec:background}\,(quasi-Hamiltonian part) and \ref{sec:qham} we
write $\rvp=(\eta,\omega)$ with $\eta$ the Cartan $3$-form on $N=G$ and
$\omega$ the invariant $2$-form on $M$, following \cite{AMM}. No convention has
changed; only the names have. Readers checking signs should always identify the
slot, never the letter.
\end{remark}

\subsection{The edge complexes}

\begin{proposition}[Structure at the edges]\label{prop:edge-complexes}
\begin{enumerate}[label=\textup{(\roman*)},leftmargin=2.2em]
\item \textup{(Target edge, $M=\varnothing$.)} Since $\Om^\bullet(\varnothing)=0$
and $C_\bullet(\varnothing)=0$,
\[
  \Om^k(F)=\Om^k(N),\quad \dF=d,
  \qquad
  C_k(F)=C_k(N),\quad \partial_F=\partial,
\]
and the pairing \eqref{eq:pairing} is $\pair{(\alpha,0)}{(S,0)}=\int_S\alpha$.
Hence $H^\bullet(\Om(F))\cong H^\bullet_{\mathrm{dR}}(N)$ and
$H_\bullet(F)\cong H_\bullet(N;\Z)$, compatibly with the pairing: the relative
theory \emph{is} the absolute de Rham theory of the target.
\item \textup{(Source edge, $N=\mathrm{pt}$.)} Since $\Om^0(\mathrm{pt})=\R$ and
$\Om^k(\mathrm{pt})=0$ for $k\geq1$, the relative complex is
\[
  \Om^0(F)=\R
  \ \xrightarrow{\ c\,\mapsto\,(0,c)\ }\
  \Om^1(F)=C^\infty(M)
  \ \xrightarrow{\ -d\ }\
  \Om^2(F)=\Om^1(M)
  \ \xrightarrow{\ -d\ }\ \cdots,
\]
that is, $\Om^{k}(F)=\Om^{k-1}(M)$ for $k\geq1$ with differential $-d$,
augmented by the constants in degree $0$. Consequently
\[
  H^0(\Om(F))=0\ \ (M\neq\varnothing),
  \qquad
  H^k(\Om(F))\cong\widetilde H^{k-1}_{\mathrm{dR}}(M)\quad(k\geq1),
\]
reduced cohomology, and the pairing is
$\pair{(0,\beta)}{(S,T)}=-\int_T\beta$: at this edge the relative pairing is the
\emph{negative} of the absolute one.
\end{enumerate}
\end{proposition}

\begin{proof}
(i) The empty manifold carries only the zero vector space of forms and the zero
group of chains, so both mapping cones collapse onto their $N$-summands, with
$\dF(\alpha,0)=(d\alpha,F^*\alpha)=(d\alpha,0)$ and
$\partial_F(S,0)=(\partial S,0)$.

(ii) For $k\geq1$ one has $\Om^k(\mathrm{pt})=0$, so
$\Om^k(F)=\Om^{k-1}(M)$ and $\dF(0,\beta)=(0,-d\beta)$. In degree $0$,
$\Om^0(F)=C^\infty(\mathrm{pt})=\R$ and $\dF c=(dc,F^*c)=(0,c)$, the constant
function $c$ on $M$; this is injective when $M\neq\varnothing$, giving
$H^0(\Om(F))=0$, which is \eqref{eq:H0} at this edge. In degree $1$,
$H^1(\Om(F))=\ker(d)/\R=H^0_{\mathrm{dR}}(M)/\R=\widetilde H^0(M)$, and in
degrees $k\geq2$ the sign $-d$ does not affect cohomology, so
$H^k(\Om(F))\cong H^{k-1}_{\mathrm{dR}}(M)$. This agrees with the long exact
sequence \eqref{eq:LES}, whose $\mathrm{pt}$-terms vanish above degree $0$. The
pairing statement is immediate from \eqref{eq:pairing} with $\alpha=0$.
\end{proof}

\begin{remark}[The sign at the source edge]\label{rem:edge-sign}
The minus sign in $\pair{(0,\beta)}{(S,T)}=-\int_T\beta$ is not an artefact to be
removed: it is forced by the cone convention \eqref{eq:cone}, under which the
source slot always enters the Stokes formula with the opposite orientation, and
it is the same sign that produces the bulk-\emph{minus}-boundary structure of
charges in Remark~\ref{rem:bulkboundary}. Its practical effect is nil for
integrality statements, since $P_\rvp=-P^M_\eta$ and $-P^M_\eta\subseteq\Z$
iff $P^M_\eta\subseteq\Z$, but it must be carried through any identification of
$L_\infty$-structures or of moment-map components.
\end{remark}

\subsection{Nondegeneracy at the edges: two cautions}

Here the two edges behave differently, and the difference is easy to miss.

\begin{proposition}[Nondegeneracy at the edges]\label{prop:edge-nondeg}
Recall that $(F,\rvp)$ is relative $n$-plectic when
$w\mapsto\bigl(\io_{TF(w)}\omega,\ \io_w\eta\bigr)$ is injective on each
$T_mM$.
\begin{enumerate}[label=\textup{(\roman*)},leftmargin=2.2em]
\item \textup{(Target edge.)} If $M=\varnothing$ the condition is \emph{vacuous}:
every $\dF$-closed $\rvp$ is relative $n$-plectic. Thus the target edge
recovers \emph{pre}-multisymplectic geometry on $N$ --- closed $(n{+}1)$-forms
--- and not multisymplectic geometry. To recover the latter one must impose
nondegeneracy of $\omega$ on $N$ separately; it is not implied by the relative
axioms. Section~\ref{sec:nondeg} identifies the right repair: the \emph{strong}
nondegeneracy of Definition~\ref{def:weak-strong}(ii) does recover absolute
multisymplectic geometry at this edge \textup{(}Proposition
\ref{prop:strong-target}\textup{)}.
\item \textup{(Source edge.)} If $N=\mathrm{pt}$ then $\omega=0$, the closedness
condition reduces to $d\eta=0$, and the nondegeneracy condition becomes
$\ker\eta=0$. Hence
\[
  \text{$(F,\rvp)$ relative $n$-plectic}
  \iff
  \text{$(M,\eta)$ absolutely $(n{-}1)$-plectic},
\]
a shift of one in the degree: a relative $n$-plectic map over a point is an
absolute $(n{-}1)$-plectic manifold, because $\eta$ has degree $n$ and an
$m$-plectic form has degree $m+1$.
\end{enumerate}
\end{proposition}

\begin{proof}
(i) A condition quantified over $m\in M$ and $w\in T_mM$ is vacuously true when
$M=\varnothing$. (ii) $\Om^{n+1}(\mathrm{pt})=0$ for $n\geq0$ forces $\omega=0$,
so $\dF\rvp=0$ reads $d\eta=0$; the first component of the nondegeneracy map
vanishes identically and the second is $w\mapsto\io_w\eta$, whose injectivity is
the definition of $(M,\eta)$ being $(n{-}1)$-plectic \cite{Rogers}.
\end{proof}

\begin{remark}[The degree shift, and why $n=1$ is empty at the source edge]
\label{rem:degree-shift}
Proposition~\ref{prop:edge-nondeg}(ii) explains a small but persistent source of
confusion. The relative theory in degree $n$ is calibrated so that the
\emph{target} carries an $(n{+}1)$-form; the source therefore carries an
$n$-form, and its absolute plectic degree is one lower. For the quasi-Hamiltonian
case $n=2$ this reads: relative $2$-plectic maps over a point are absolutely
$1$-plectic, i.e.\ presymplectic manifolds --- which is exactly the classical
statement that a quasi-Hamiltonian space for the trivial group is a manifold with
an invariant closed $2$-form. For $n=1$ the source edge is essentially empty:
$0$-plectic means a nowhere-vanishing $1$-form with trivial kernel, which forces
$\dim M=1$. The relative theory in degree one therefore has no useful source
edge, and this is precisely why the rigidity theorem of
Section~\ref{sec:rigidity} has no absolute counterpart to be compared with on
that side.
\end{remark}

\subsection{The edges are morphisms of arrows}

The two collapses are not merely notational: they are instances of
Definition~\ref{def:arrowmap}, so all of Theorem~\ref{thm:functoriality} applies
to them, and the edge comparisons below need no separate verification.

\begin{lemma}\label{lem:edge-morphisms}
Let $F\colon M\to N$ be arbitrary. Then
\[
  \Phi_{\mathrm{tgt}}=(\,\varnothing\hookrightarrow M,\ \id_N\,)\colon
  \bigl(\varnothing\to N\bigr)\longrightarrow F,
  \qquad
  \Phi_{\mathrm{src}}=(\,\id_M,\ N\to\mathrm{pt}\,)\colon
  F\longrightarrow\bigl(M\to\mathrm{pt}\bigr)
\]
are morphisms of arrows. Consequently:
\begin{enumerate}[label=\textup{(\roman*)},leftmargin=2.2em]
\item $\Phi_{\mathrm{tgt}}^*\colon\Om^\bullet(F)\to\Om^\bullet(N)$ is the
forgetful map $(\alpha,\beta)\mapsto\alpha$, and
Theorem~\ref{thm:functoriality}(iv) gives $P_\omega\subseteq P_\rvp$;
\item $\Phi_{\mathrm{src}}^*\colon\Om^\bullet(M\to\mathrm{pt})\to\Om^\bullet(F)$
sends a closed $\beta\in\Om^{n}(M)$ to $(0,\beta)$, and
Theorem~\ref{thm:functoriality}(iv) gives $P_{(0,\beta)}\subseteq -P^M_\beta$.
\end{enumerate}
\end{lemma}

\begin{proof}
Both squares commute trivially: $F\circ(\varnothing\hookrightarrow M)$ and
$\id_N\circ(\varnothing\to N)$ are both the unique map $\varnothing\to N$, and
$(M\to\mathrm{pt})\circ\id_M=(N\to\mathrm{pt})\circ F$. The descriptions of the
induced maps are Theorem~\ref{thm:functoriality}(i) combined with
Proposition~\ref{prop:edge-complexes}, and the period statements are
Theorem~\ref{thm:functoriality}(iv).
\end{proof}

\begin{corollary}[The period criterion interpolates between the edges]
\label{cor:edge-periods}
In the notation of Theorem~\ref{thm:period-criterion}, write
$K=\ker\bigl(F_*\colon H_n(M;\Z)\to H_n(N;\Z)\bigr)$.
\begin{enumerate}[label=\textup{(\roman*)},leftmargin=2.2em]
\item At the target edge, $K=0$, the defect homomorphism vanishes, and
$P_\rvp=P_\omega$: relative integrality is exactly integrality of $\omega$ on
$N$.
\item At the source edge, $H_n(\mathrm{pt};\Z)=0$ for $n\geq1$, so $K=H_n(M;\Z)$
is everything, $P_\omega=0$, and $\Theta_\rvp[T]=-\int_T\eta$; hence
$P_\rvp=-P^M_\eta$ and relative integrality is exactly integrality of $\eta$
on $M$.
\end{enumerate}
Thus Theorem~\ref{thm:period-criterion} degenerates at the two edges to the two
classical integrality conditions, and its content in between --- the defect
homomorphism on $K$ --- is precisely the part that neither absolute theory sees.
\end{corollary}

\begin{proof}
(i) $H_n(\varnothing;\Z)=0$, so $K=0$ and Theorem~\ref{thm:period-criterion}(iii)
applies. (ii) $H_n(\mathrm{pt};\Z)=0$ for $n\geq1$, so $F_*=0$ and $K=H_n(M;\Z)$;
$\omega=0$ gives $P_\omega=0$, and the filling $S$ may be taken to be any chain
in the point, contributing nothing, so
$\Theta_\rvp[T]=-\int_T\eta$ by \eqref{eq:defect}. Now apply
Theorem~\ref{thm:period-criterion}(iv).
\end{proof}

\subsection{The applications at the edges}

We now run each application through both edges. Table~\ref{tab:edges}
summarizes; the discussion that follows justifies each entry and flags the
places where the specialization is not literal.

\begin{table}[htb]
\centering
\renewcommand{\arraystretch}{1.3}
\begin{tabular}{@{}p{0.20\textwidth}p{0.36\textwidth}p{0.36\textwidth}@{}}
\hline
& \textbf{Target edge} $M=\varnothing$ & \textbf{Source edge}
$N=\mathrm{pt}$\\
\hline
structure &
closed $(n{+}1)$-form on $N$; \emph{pre}-multisymplectic &
$(n{-}1)$-plectic manifold $(M,\eta)$; degree drops by one\\
A. actions &
$\int_W u^*\omega$ on closed $W$; classical topological term &
$-\int_{\partial W}v^*\eta$ on null-bordant $n$-cycles in $M$\\
B. prequantization &
classical Weil--Kostant on $N$; torsor $H^1(N;U(1))$ &
$U(1)$-valued $s$ with $d\log s=-2\pi i\eta$; torsor
$\widetilde H^0(M;U(1))$\\
C. charges &
absolute Noether charges on $(n{-}1)$-cycles in $N$ &
absolute charges on $(n{-}2)$-cycles in $M$, with a sign\\
D. rigidity &
\textbf{fails}: $H^0=\R$, Kostant--Souriau cocycle returns &
holds, but forces $\dim M=1$; vacuous in practice\\
E. quasi-Hamiltonian &
no space; only the Cartan $3$-form on $G$ &
trivial group; presymplectic manifold\\
\hline
\end{tabular}
\medskip
\caption{The five applications at the two edges. The entry in bold type is the
one where the specialization is \emph{not} a recovery but a failure, and is
therefore a sharpness witness for the hypotheses of
Theorem~\ref{thm:rigidity}; see Section~\ref{subsec:not-recovered}.}
\label{tab:edges}
\end{table}

\subsubsection*{A. Topological terms}
At the target edge a relative field is a pair $(u,v)$ with $v$ defined on
$\partial W$ and valued in $M=\varnothing$; such a $v$ exists if and only if
$\partial W=\varnothing$. So relative fields are exactly smooth maps
$u\colon W\to N$ from \emph{closed} oriented $(n{+}1)$-manifolds, with action
$\int_Wu^*\omega$, and Theorem~\ref{thm:topterm} becomes the classical statement
that a form defines a homotopy-invariant functional on such maps precisely when
it is closed. It is worth recording that the constraint $\partial W=\varnothing$
is forced, not assumed: this is the honest form of the remark opening
Example~\ref{ex:WZW}.

At the source edge $u$ is constant and $\omega=0$, so
$S_\rvp(u,v)=-\int_{\partial W}v^*\eta$ with $\eta$ a closed $n$-form on $M$.
The functional is thus the absolute topological term of $\eta$, evaluated on
those $n$-cycles in $M$ that bound in the trivial sense, i.e.\ on
$v_*[\partial W]$ for $\partial W$ the boundary of a compact $W$; homotopy
invariance is the absolute statement.

\subsubsection*{B. Relative prequantization}
At the target edge the section datum in Definition~\ref{def:relprequant} is
vacuous, and $(L,\nabla)$ is a prequantum line bundle of curvature
$-2\pi i\,\omega$ on $N$; Theorem~\ref{thm:prequantization} is the Weil--Kostant
theorem, and its torsor $H^1(F;U(1))=H^1(N;U(1))$ is the group of flat line
bundles. This is Corollary~\ref{cor:prequantexamples}(a), now with the torsor
identified as well.

At the source edge the target is a point, so $(L,\nabla)=(\mathbb{C},d)$ and the entire
datum is the unit section $s$ of the trivial bundle over $M$ with
\[
  ds=-2\pi i\,\eta\otimes s,
  \qquad\text{i.e.}\qquad
  s\colon M\to U(1)\ \text{ with }\ \tfrac{1}{2\pi i}\,d\log s=-\eta .
\]
Theorem~\ref{thm:prequantization} then asserts that such an $s$ exists exactly
when the periods of $\eta$ are integral --- the classical statement that a
closed $1$-form is the logarithmic derivative of a circle-valued function
precisely when its de Rham class is integral --- and that the solutions form a
torsor over $H^1(F;U(1))=\widetilde H^0(M;U(1))$, the locally constant
$U(1)$-valued functions modulo global constants, since a global constant is an
isomorphism of $(\mathbb{C},d)$. Both assertions are correct and classical; note that
the relative degree is $n=1$ here, so it is the \emph{source} $1$-form that is
quantized, in accordance with Proposition~\ref{prop:edge-complexes}(ii).

\subsubsection*{C. Charges}
At the target edge, relative cycles are cycles in $N$, the moment component
$f_1$ has only its $N$-slot, and \eqref{eq:charge} becomes
$Q_x=\int_Sf_1^N(x)$ over $(n{-}1)$-cycles in $N$: the Noether charges of
absolute multisymplectic geometry \cite{CFRZ,Rogers}. Theorem~\ref{thm:noether}
reduces term by term to the absolute Noether identity, all source contractions
being zero.

At the source edge, $\Om^{n-1}(F)=\Om^{n-2}(M)$ and the Hamilton equation
$\dF\sigma=-\io_{v_\sigma}\rvp$ reads
$-d\sigma_M=\io_{v_M}\eta$, i.e.\ $d\sigma_M=-\io_{v_M}\eta$: precisely Rogers'
Hamiltonian condition for the $(n{-}1)$-plectic manifold $(M,\eta)$, whose
Hamiltonian forms have degree $n-2$. The charges become
$Q_x=-\int_Tf_1^M(x)$ over $(n{-}2)$-cycles in $M$, the absolute charges up to
the global sign of Remark~\ref{rem:edge-sign}.

\subsubsection*{E. Quasi-Hamiltonian geometry}
The target edge deletes the space and leaves only the Cartan $3$-form on $G$
with its one-step extension $\widetilde\mu$; nothing is classified, but the
universal data of Proposition~\ref{prop:universality} survives, which is the
formal reason it is universal. The source edge is the case of the trivial group:
$N=G=\{e\}$ is a point, $\eta=0$, and a quasi-Hamiltonian space degenerates to a
manifold with an invariant closed $2$-form, i.e.\ a presymplectic manifold, in
accordance with Proposition~\ref{prop:edge-nondeg}(ii) and
Remark~\ref{rem:degree-shift}.

\subsection{What the edges do not recover}\label{subsec:not-recovered}

Two entries of Table~\ref{tab:edges} are not recoveries, and they are the
informative ones.

\subsubsection*{D. Rigidity fails at the target edge --- and this is sharp}
Theorem~\ref{thm:rigidity} rests on \eqref{eq:H0}: a $\dF$-closed function
vanishes when $N$ is connected and $M\neq\varnothing$. At the target edge the
second hypothesis fails, and by
Proposition~\ref{prop:edge-complexes}(i) one has $H^0(\Om(F))=H^0_{\mathrm{dR}}(N)=\R$
for connected $N$. Every conclusion of the theorem fails accordingly:
comoment maps are unique only modulo $\gdual$-valued constants, the
Kostant--Souriau cocycle reappears as the obstruction to strictness, and
equivariance becomes a genuine condition. The target edge is therefore not a
degenerate instance of Theorem~\ref{thm:rigidity} but a \emph{witness for the
sharpness of its hypotheses}: the assumption $M\neq\varnothing$ is doing all the
work, and Remark~\ref{rem:rigiditygeom} explains why --- the source pins the
target constants to zero.

At the source edge the hypothesis \eqref{eq:H0} does hold, by
Proposition~\ref{prop:edge-complexes}(ii), and the theorem is true but nearly
empty: as noted in Remark~\ref{rem:degree-shift}, relative degree $n=1$ over a
point forces $\dim M=1$ under nondegeneracy. Dropping nondegeneracy, the
statement is still meaningful and correct, and reads as follows.

\begin{corollary}[Rigidity at the source edge]\label{cor:rigidity-source}
Let $N=\mathrm{pt}$, let $M$ be a connected manifold with a $G$-invariant closed
$1$-form $\eta$, and let $\rvp=(0,\eta)\in\Om^2(F)$. Then a comoment map
exists, is unique, and is given by
\[
  f(x)=\io_{v_x}\eta\in\R,
\]
which is a constant for every $x\in\g$; it is automatically a Lie algebra
homomorphism into the abelian Lie algebra $\R$, so $f$ vanishes on
$[\g,\g]$, and it is automatically equivariant.
\end{corollary}

\begin{proof}
The comoment equation \eqref{eq:comoment} reads
$\dF f(x)=(0,f(x))$ on the left and
$-\io_{v_x}\rvp=(0,\io_{v_x}\eta)$ on the right, so it says exactly that the
function $\io_{v_x}\eta$ is the constant $f(x)$. That $\io_{v_x}\eta$ is locally
constant is automatic: $d\,\io_{v_x}\eta=\Lder_{v_x}\eta-\io_{v_x}d\eta=0$ by
invariance and closedness; connectedness of $M$ makes it constant, so a comoment
map exists. Uniqueness, strictness and equivariance are
Theorem~\ref{thm:rigidity}(i)--(iii), whose hypotheses hold by
Proposition~\ref{prop:edge-complexes}(ii). Concretely, the bracket
$\{f(x),f(y)\}=\io(v_x\wedge v_y)\rvp$ vanishes because $\eta$ has degree one
and cannot absorb two contractions; strictness therefore reads $f([x,y])=0$,
which is also visible directly from
$\io_{v_{[x,y]}}\eta=\Lder_{v_x}\io_{v_y}\eta-\io_{v_y}\Lder_{v_x}\eta=0$, the
first term being the derivative of a constant.
\end{proof}

\subsubsection*{The defect homomorphism is invisible at both edges}
By Corollary~\ref{cor:edge-periods} the defect homomorphism $\Theta_\rvp$ is
zero at the target edge (because $K=0$) and reduces to the absolute period map
at the source edge (because $F_*=0$). Its genuinely relative values occur only
when $F_*$ is neither injective nor zero on $H_n$, that is, only strictly
between the edges. Example~\ref{ex:disc} is the smallest such situation, and
Example~\ref{ex:torus} shows what happens when one drifts back toward the target
edge behaviour. The same remark applies to the bulk--boundary splitting of
charges \eqref{eq:bulkboundary-diagram}, which collapses to a single term at
either edge, and to the Wess--Zumino three-level structure of
Theorem~\ref{thm:WZ}, whose first and third levels coincide at both edges and
differ only in between.

\begin{remark}[Summary]\label{rem:edge-summary}
The two edges therefore serve three distinct purposes. They are
\emph{consistency checks}: every formula of this paper must, and does, return
the known absolute formula at each edge, which is how the sign conventions of
Section~\ref{sec:intro} were verified. They are \emph{sharpness witnesses}: the
failure of rigidity at $M=\varnothing$ shows that the hypothesis
$M\neq\varnothing$ cannot be dropped. And they \emph{localize the new content}:
the defect homomorphism, the bulk--boundary splitting, and the gap between the
levels of Theorem~\ref{thm:WZ} all vanish at the edges, so whatever the relative
theory contributes beyond the absolute ones is supported strictly in the
interior.
\end{remark}

\end{document}